\documentclass[11pt,leqno]{article}
\usepackage{amsmath,amsthm,amsfonts,amssymb}
\usepackage{comment}

\newcommand{\eps}{{\varepsilon}}

\def\eps{\epsilon }

\newcommand\adots{\mathinner{\mkern2mu\raise1pt\hbox{.}
\mkern3mu\raise4pt\hbox{.}\mkern1mu\raise7pt\hbox{.}}}

\newtheorem{theo}{Theorem}[section]
\newtheorem{prop}[theo]{Proposition}
\newtheorem{cor}[theo]{Corollary}
\newtheorem{lem}[theo]{Lemma}
\newtheorem{defn}[theo]{Definition}
\newtheorem{ass}[theo]{Assumption}

\newtheorem{rem}[theo]{Remark}

\newtheorem{nota}[theo]{Notations}

\numberwithin{equation}{section}

\begin{document}


\title{\textbf{Reflection of conormal pulse solutions  to  large variable-coefficient semilinear hyperbolic systems}}

\author{Mark Williams}

\author{ 
Mark {\sc Williams}\thanks{ Mathematics Department, University of North Carolina, CB 3250, Phillips Hall, 
Chapel Hill, NC 27599. USA. Email: {\tt williams@email.unc.edu}.}}

\maketitle

\begin{abstract}

We provide a rigorous justication of nonlinear geometric optics expansions for reflecting \emph{pulses}  in space dimensions $n>1$.  The pulses arise as solutions to  variable coefficient semilinear first-order hyperbolic systems.   The justification applies to $N\times N$  systems with $N$ interacting pulses which depend on phases that may be nonlinear.   The \emph{coherence} assumption made in a number of earlier works  is dropped.    We consider problems in which incoming pulses are generated from pulse boundary data as well as problems in which a single outgoing pulse reflects off a possibly curved boundary to produce a number of incoming pulses.    Although we focus here on boundary problems, it is clear that similar results hold by similar methods for the Cauchy problem for $N\times N$ systems in free space.

\end{abstract}

\tableofcontents

\section{Introduction and guide to the paper}

The goal of this paper is to provide a rigorous justication of nonlinear geometric optics expansions for reflecting \emph{pulses}   in space dimensions $n>1$.  The pulses arise as solutions to  variable coefficient semilinear first-order hyperbolic systems.   The justification applies to $N\times N$  systems with $N$ interacting pulses which depend on phases that may be nonlinear.      The \emph{coherence} assumption made in earlier works such as \cite{hmr,jmr,jmr2}  is dropped.    Although we focus here on boundary problems, it is clear that similar results hold by similar methods for the Cauchy problem for $N\times N$ systems in free space.  

To explain some of these terms and place the problem in context, we consider first an $N\times N$  semilinear hyperbolic system in free space $\mathbb{R}_t\times \mathbb{R}^n_x$:
\begin{align}\label{i1}
\begin{split}
\; \mathcal{L}(t,x,\partial_t,\partial_x)u^\eps:=\partial_t u^\eps+\sum^n_{j=1}A_j(t,x)\partial_j u^\eps=f(t,x,u^\eps).
\end{split}
\end{align}
In nonlinear geometric optics one considers families of solutions that oscillate with wavelength proportional to $\eps>0$, and attempts to show that they have asymptotic expansions of the form\begin{align}\label{i2}
u^\eps(t,x)={U}_0(t,x,\theta)|_{\theta=\frac{\psi(t,x)}{\eps}}+o(1) \text{ as }\eps\to 0,
\end{align}
which are valid on a fixed space-time domain $U$ independent of $\eps$.  Here $o(1)$ refers to some appropriate norm;  in this paper we use the norm $L^\infty\cap N^m(U)$, where $N^m(U)$ is the conormal space defined in Definition \ref{43b}; see also Remark \ref{43f}.\footnote{We study here only the regime of ``weakly nonlinear geometric optics", where $u^\eps = \eps^p{U}_0(t,x,\frac{\psi}{\eps})+o(\eps^p)$, with  $p=0$  chosen so that ${U}_0$ satisfies a nonlinear equation, while for larger $p$ the equation would be linear.    
In a quasilinear problem, we would replace \eqref{i2} by 
$u^\eps(t,x)=u_0(t,x)+\eps {U}_0(t,x,\frac{\psi}{\eps})+o(\eps)$, with $u_0$ a given background state.}
 Here $\theta=(\theta_1,\dots,\theta_N)$, the smooth component functions $\psi_j$ of $\psi=(\psi_1,\dots,\psi_N)$ are called \emph{phases}, and the explicitly constructed function ${U}_0(t,x,\theta)$ is called the leading  \emph{profile}.   When ${U}_0$ is periodic in $\theta$ (or almost-periodic),  we refer to the wave $u^\eps$ as  a wavetrain; when ${U}_0$ decays to $0$ as $|\theta|\to \infty$, we refer to the wave as a pulse.   This includes the case where ${U}_0$ has compact support in $\theta$.

Starting in the early 1990s rigorous justifications of expansions like \eqref{i2} (and similar expansions in the quasilinear case) were provided in a series of papers by Joly, M\'etivier, and Rauch.   In \cite{jmr} the authors constructed a number of examples showing that in the case of wavetrains, without strong assumptions on the space of phases, the Cauchy problem is ill-posed; solutions $u^\eps$ may fail to exist on a domain independent of $\eps$ as $\eps\to 0$ because of \emph{focusing}.   The paper \cite{jmr} justified expansions like \eqref{i2} on a fixed domain independent of $\eps$ under the following \emph{coherence assumption} on the space of phases.\\

   If $\Omega\subset \mathbb{R}^{1+n}$ is open and $\Phi\subset C^\infty(\Omega,\mathbb{R})$ is a real vector space, then $\Phi$ is \emph{$\mathcal{L}-$coherent} when for all $\phi\in \Phi\setminus\{0\}$ one of the following two conditions holds:\\

(i) at every point $(t,x)\in\Omega$ both $\det \mathcal{L}(t,x,d\phi(t,x))=0$ and $d\phi(t,x)\neq 0$;

(ii) at every point $(t,x)\in\Omega$, $\det \mathcal{L}(t,x,d\phi(t,x))\neq 0$.\\

Obvious and important examples of $\mathcal{L}$-coherence occur when:

 (A) $\mathcal{L}$ is a variable coefficient operator and $\Phi$ is a \emph{one}-dimensional space of characteristic phases spanned by a single nondegenerate phase;\footnote{When $\phi$ satisfies $\det \mathcal{L}(t,x,d\phi(t,x))=0$ on $\Omega$, we call it a \emph{characteristic phase}.   When $\phi$ satisfies $d\phi(t,x) \neq 0$ for all $(t,x)\in\Omega$, we say it is \emph{nondegenerate}.}

 (B) $\Phi$ is a space of \emph{linear} phases $\phi(t,x)=\alpha\cdot (t,x)$, $\alpha\in\mathbb{R}^{1+n}$,  and the operator $\mathcal{L}$ in \eqref{i1} either has constant coefficients, or is replaced by a quasilinear operator where each $A_j$ is a function of $u$ alone and $u$ is taken close to a constant background state $u_0$.   

It is quite difficult to find other examples of spaces of coherent phases, but a few examples do exist for problems in free space; see section 3 of \cite{jmr}.   

The coherence condition, which requires that linear combinations of characteristic phases satisfy $\det \mathcal{L}(t,x,d\phi(t,x))=0$ either 
everywhere in $\Omega$ or nowhere in $\Omega$, is a natural condition for wavetrains, but not for pulses.   To see this let $a(t,x)$ $b(t,x)$, $g(\theta)$, and $h(\theta)$ be compactly supported smooth functions, let $\psi_1(t,x)$ and $\psi_2(t,x)$ be characteristic phases, and consider the wavetrain and pulse interactions given respectively by
\begin{align}\label{i3} 
\begin{split}
&(i)\; a(t,x)e^{i\frac{m\psi_1}{\eps}}\cdot b(t,x)e^{i\frac{n\psi_2}{\eps}}= a(t,x)b(t,x)e^{i\frac{m\psi_1+n\psi_2}{\eps}}\\
&(ii)\;  a(t,x)g\left(\frac{\psi_1}{\eps}\right)\cdot b(t,x)h\left(\frac{\psi_2}{\eps}\right)=a(t,x)b(t,x) \;g\left(\frac{\psi_1}{\eps}\right) h\left(\frac{\psi_2}{\eps}\right).
\end{split}
\end{align}
Even though the coherence condition appears to be less relevant in the case of pulses, most of the existing  rigorous  geometric optics results for pulses concern  problems 
where either (A) holds (e.g., \cite{ar2,altermanrauch} or (B)  holds (e.g., \cite{CW1,CW2,willig}).\footnote{See also \cite{cr} and its companion papers,  which give a rigorous treatment of focusing spherical pulses, a case where coherence does not hold.}


        Without assuming coherence we now consider  the variable-coefficient  boundary problem that will be our focus in this paper, an $N\times N$ strictly hyperbolic system on the half-space $\mathbb{R}_t\times \mathbb{R}^n_+=\{(t,x):x_n\geq 0\}$:\footnote{We show in section \ref{geometric} that \eqref{ia1} is  local model to which more general continuation problems can be reduced by a change of variables; see also Remark \ref{geo}.  The   assumptions  on $(L,B)$, $f(t,x,u)$ and $g(t,x',\theta)$ are stated in Theorem \ref{mr}.} 
 \begin{align}\label{ia1}
\begin{split}
&(a)\; \mathcal{L}(t,x,\partial_t,\partial_x)u^\eps:=\partial_t u^\eps+\sum^n_{j=1}A_j(t,x)\partial_j u^\eps=f(t,x,u^\eps)\text{ in }x_n>0\\
&(b)\; B(t,x')u^\eps|_{x_n=0}=g(t,x',\theta_0)|_{\theta_0=\frac{t}{\eps}}:=b^\eps(t,x'),\\\;\; 
&(c)\; u^\eps=0\text{ in }t<0.
\end{split}
\end{align}

If we take \emph{any} $C^1$ function of the form $U_0(t,x,\frac{\psi}{\eps})$ and substitute it for $u^\eps$ in  \eqref{ia1}(a), terms  of order $\frac{1}{\eps}$ will appear on the left.  To make those terms vanish, the classical strategy is to choose  the leading profile of the form
\begin{align}\label{i4}
{U}_0(t,x,\theta)=\sum^N_{k=1}\sigma_k(t,x,\theta_k)r_k(t,x),
\end{align}
where each $r_k(t,x)\in\mathbb{R}^N$ is an  eigenvector associated to an eigenvalue $\lambda_k(t,x,\partial_{t,x'}\psi_k)$ of an appropriate operator depending on $\psi_k$, and $\psi_k$ solves an ``eikonal problem" of the form
\begin{align}\label{i4z}
\begin{split}
&\partial_{x_n}\psi_k=-\lambda_k(t,x,\partial_{t,x'}\psi_k)\\
&\psi_k|_{x_n=0}=t.
\end{split}
\end{align}
For $r_k$ and $\psi_k$ as constructed in section \ref{eikonal},  the terms of order $\frac{1}{\eps}$ vanish no matter how the scalar functions $\sigma_k$ in \eqref{i4} are chosen, and we obtain
\begin{align}\label{i4y}
\mathcal{L}(t,x,\partial_t,\partial_x) \;U_0\left(t,x,\frac{\psi}{\eps}\right)=f\left(t,x,U_0\left(t,x,\frac{\psi}{\eps}\right)\right)+R^\eps_a(t,x), \text{ where }R^\eps_a=O(1)\text{ as }\eps\to 0.
\end{align}
The remainder $R^\eps_a$ can be written as a sum of terms, some of which depend only on a single $\psi_k$ (the ``single-phase terms") and some of which depend on at least two distinct $\psi_k$ (the ``multiphase-terms'').  In order to prove \eqref{i2} the part of $R^\eps_a$ that is no smaller than $O(1)$ must be solved away.   The scalar profiles $\sigma_k(t,x,\theta_k)$ are chosen to satisfy ``transport equations" which effectively solve away the part of $R^\eps_a$ consisting of terms  of the form 
\begin{align}\label{i4x}
q(t,x,\sigma_k,\partial_{t,x}\sigma_k)|_{\theta_k=\frac{\psi_k}{\eps}} \cdot r_k(t,x),
\end{align}
 which depend on a single $\psi_k$ for some $k\in \{1,\dots,N\}$ \emph{and} are polarized in the direction of $r_k$.\footnote{Terms like  \eqref{i4x} make up $E\mathcal{F}|_{\theta=\frac{\psi}{\eps}}$ for $E\mathcal{F}$ as in \eqref{c3y}.   The $\sigma_k$ can be chosen so that $U_0\left(t,x,\frac{\psi}{\eps}\right)$ satisfies the boundary and initial conditions in \eqref{ia1} exactly. }  

The transport equations \eqref{c12} satisfied by the $\sigma_k(x,\theta_k)$ are derived in sections \ref{strategy}, \ref{e}, and \ref{U0}.  These equations are completely decoupled.  Provided \eqref{i2} holds, 
this shows that unlike wavetrains, pulses do not interact at leading order, and  ``resonances"  among phases have no effect on the leading profile $U_0$.\footnote{A resonance occurs when an integer combination of characteristic phases is again a characteristic phase.}    The decoupling reflects the fact that pulses interact on much smaller sets than wavetrains, sets that get smaller, in fact,  as $\eps$ decreases.   For example, in \eqref{i3} if we assume that $g$ and $h$ are supported in $|\theta|\leq 1$, we
have
\begin{align}\label{i5}
\begin{split}
&(i)\;\mathrm{supp}\; a(t,x)b(t,x)e^{i\frac{m\psi_1+n\psi_2}{\eps}}=\mathrm{supp}\;a(t,x)b(t,x),\text{ but }\\
&(ii)\;\mathrm{supp}\;a(t,x)b(t,x)\; g\left(\frac{\psi_1}{\eps}\right) h\left(\frac{\psi_2}{\eps}\right)\subset \mathrm{supp}\;a(t,x)b(t,x)\cap \{(t,x):|\psi_1|\leq \eps\text{ and }|\psi_2|\leq \eps\}.
\end{split}
\end{align}

To solve away the remaining parts of $R^\eps_a$ that are no smaller than $O(1)$, we construct a \emph{corrector} $U^\eps_1(t,x)$ such that the approximate solution
\begin{align}\label{i6}
u^\eps_a(t,x):=U_0(t,x,\frac{\psi}{\eps})+\eps U^\eps_1(t,x)
\end{align}
gives a remainder that is $o(1)$ rather than $O(1)$:
\begin{align}\label{i5b}
\mathcal{L}(t,x,\partial_t,\partial_x)u^\eps_a=f(t,x,u^\eps_a)+r^\eps_a(t,x), \text{ where }r^\eps_a=o(1)\text{ as }\eps\to 0.
\end{align}
In the single phase problem considered in \cite{ar2}, where $\theta$ and $\psi$ are  scalar and the leading profile $U_0(t,x,\theta)$ is compactly supported in $\theta$, a corrector of the form $U^\eps_1(t,x)=V(t,x,\theta)|_{\theta=\frac{\psi}{\eps}}$ was constructed, where $V$ is  bounded but no longer compactly supported in $\theta$.   Considering that pulses do not interact at leading order, it is natural in the multiphase case, where $\theta=(\theta_1,\dots,\theta_N)$, to  look for a corrector of the form
\begin{align}\label{i7}
\sum^N_{k=1}V_k(t,x,\theta_k)|_{\theta_k=\frac{\psi_k}{\eps}}:=V(t,x,\theta)|_{\theta=\frac{\psi}{\eps}},
\end{align}
where the individual $V_k$ are constructed like the corrector $V_k$ in \cite{ar2}.   Such a corrector, constructed in section \ref{V},  is useful for solving away the single-phase part of $R^\eps_a$ that is left after all polarized terms of the form  \ref{i4x} are removed.\footnote{This is the single-phase part of $(I-E)\mathcal{F}|_{\theta=\frac{\psi}{\eps}}$ in \eqref{c3y}.} But a corrector of the form \eqref{i7} is useless for solving away the multiphase or interacting part of $R^\eps_a$ which may contain,  for example,  $O(1)$ products like \eqref{i3}(ii), with factors depending on two or more linearly independent phases $\psi_k$.    Thus, we look for a corrector of the form
\begin{align}\label{i8}
U^\eps_1(t,x)=V(t,x,\theta)|_{\theta=\frac{\psi}{\eps}}+\mathcal{W}^\eps(t,x),
\end{align}
where the  noninteraction term $V$  solves away the unpolarized, single-phase part of $R^\eps_a$, and the interaction term $\mathcal{W}^\eps$  solves away the multiphase or interaction part" of $R^\eps_a$.   In fact, the multiphase part of $R^\eps_a$ is not solved away completely, but there remains only a harmless piece, $r^\eps_a(t,x)$ in \eqref{i5b}, of size $O(\sqrt{\eps})=o(1)$ in $L^\infty\cap N^m$; see Proposition \ref{39a}.  A formula for $r^\eps_a$ is given just above \eqref{a9z}.


Thus,  even though pulses do not interact at leading order,  their interactions determine the part of the corrector given by $\mathcal{W}^\eps(t,x)$.   The main novelties of this paper are contained in section \ref{multicorrector}, which constructs $\mathcal{W}^\eps(t,x)$, and in sections \ref{H2}, \ref{L20}, and \ref{E1}, which give various conormal estimates involving $\mathcal{W}^\eps(t,x)$.    The construction in section \ref{multicorrector} draws some inspiration from \cite{CW1}, which considered the corresponding quasilinear multiphase reflection problem with \emph{linear} phases (case (B) of $\mathcal{L}-$coherence).  But the construction and estimation of $\mathcal{W}^\eps$ is more challenging here because the operator $\mathcal{L}(t,x,\partial_t,\partial_x)$ in \eqref{ia1}  has variable coefficients,  and the characteristic phases $\psi_k$ obtained by solving the eikonal problems \eqref{i4z} are now nonlinear.  
The construction of $\mathcal{W}^\eps$ takes advantage of the fact that the phases $\psi_k(t,x)$ are all nearly equal to $t$ in the 
``interaction region" $I_\eps:=\{(t,x):|t|\lesssim \eps, |x_n|\lesssim \eps\}$; see Remark \ref{explain}.

We give the construction of the approximate solution $u^\eps_a$ first for the case where the nonlinear function  $f(t,x,u)$ in 
\eqref{ia1} is at most \emph{quadratic} in $u$.  This allows us to see the pulse interactions explicitly.  In section \ref{genf} we explain how the construction and estimation of $u^\eps_a$  can be extended to the case where $f(t,x,u)$ in \eqref{ia1} is a general smooth function such that $f(t,x,0)=0$.    

The exact solution $u^\eps\in L^\infty\cap N^m(U)$ to the boundary problem \eqref{ia1} is obtained in section \ref{bpexact}  on a fixed open set $U\ni 0$ independent of $\eps$ by direct application of a result of \cite{metajm}.

\begin{rem}[Conormal spaces and progressing waves]\label{43f}
\textup{The paper \cite{metajm} and its companion \cite{metduke} were concerned with a problem seemingly quite different from the one considered here.  
Building on earlier work of Bony \cite{bony,bony2} and Beals-M\'etivier \cite{bm} which described  the propagation, interaction, and reflection of smooth enough conormal progressing wave solutions to $N\times N$ semilinear hyperbolic systems,  M\'etivier in the companion papers  \cite{metduke, metajm} was able to extend such results to the case of  \emph{discontinuous} progressing waves.\footnote{Rauch-Reed \cite{rr} had earlier studied discontinuous progressing waves for $2\times 2$ semilinear systems.} } 

\textup{We now describe the spaces $N^m(U)$ introduced in \cite{metajm}, where $U\ni 0$ is any  small enough open set in $\mathbb{R}^{1+n}_+$.\footnote{These spaces along with their natural norms are discussed in more detail in section \ref{Nm}.} Let $\Sigma_j=\{\psi_j=0\}$ be one of the $N$ characteristic surfaces defined by $\psi_j$ as in \eqref{i4z}, and let $\mathcal{M}_j$ denote the space of smooth vector fields on $U\subset\mathbb{R}^{1+n}$ tangent to both $\Sigma_j$ and $\Delta:=\Sigma_j\cap \{x_n=0\}$.   Define  
$N^m(U,\mathcal{M}_j)$ as the set of $u\in L^2(U)$ such that $M^Iu\in L^2(U)$, where $M^I$ denotes any composition of $\leq m$ elements of $\mathcal{M}_j$.   Let $N^m(U)$ denote the set of $u\in L^2(U)$ such that $u=\sum^N_{j=1} u_j$ for some $u_j\in N^m(U,\mathcal{M}_j)$.     Under certain conditions involving $U$ and $\Delta$,  the spaces $L^\infty\cap N^m(U)$ are algebras for $m>\frac{n+5}{2}$; see Proposition \ref{35a}.}

\textup{Clearly, elements of $L^\infty\cap N^m(U,\mathcal{M}_j)$
or $L^\infty\cap N^m(U)$ may be discontinuous across $\Sigma_j$ even when $m$ is large.
The exact and approximate solutions we consider here are in $C^\infty(U)$, but they depend on a small parameter $\eps$.  Instead of being discontinuous across $\Sigma_j$, they exhibit a ``fast transition region" near $\Sigma_j$ that gets thinner as $\eps\to 0$. }  
\end{rem}

If $h^\eps$ is a function that depends on the small parameter $\eps>0$ and $p\geq 0$, we 
write
\begin{align}\label{43g}
|h^\eps|_{L^\infty\cap N^m(U)}\lesssim \eps^p
\end{align}
if there exist  positive constants $\eps_0$ and $C$ such that
\begin{align}
|h^\eps|_{L^\infty\cap N^m(U)}\leq C\eps^p \text{ for all }\eps\in (0,\eps_0].
\end{align}
The final step is to estimate  the error $u^\eps-u^\eps_a:=w^\eps$.  We do this by applying the $L^\infty\cap N^m(U)$  estimates of \cite{metajm} to the boundary problem \eqref{error2} satisfied by $w^\eps$.   This yields
\begin{align}\label{i9}
|w^\eps|_{L^\infty\cap N^m(U)} \lesssim \sqrt{\eps}, \text{ where }m>\frac{n+5}{2}.
\end{align}
Since $|\eps U^\eps_1(t,x)|_{L^\infty\cap N^m(U)}\lesssim \eps$ (see Proposition \ref{32a}), this completes the proof of \eqref{i2} for the boundary  problem \eqref{ia1} and provides the rate of convergence $\sqrt{\eps}$.   This rate is faster than the rate found in \cite{CW1} by a quite different error analysis.\footnote{The error analysis in \cite{CW1} did not use conormal estimates.  In the quasilinear problem considered there,  we constructed the exact solution as $u^\eps(t,x)=u_0+\eps\mathcal{U}^\eps(t,x,\theta_0)|_{\theta_0=\frac{\beta\cdot (t,x')}{\eps}}$, and studied the ``singular system" satisfied by $\mathcal{U}^\eps(t,x,\theta_0)$.   The error analysis was based on estimates for that singular system proved by using the pulse calculus of \cite{CGW2} to construct singular Kreiss symmetrizers.  This kind of analysis breaks down in the present setting.} 

These arguments yield our first main result, which we state  here for the model  problem \eqref{ia1}.\footnote{The assumptions in Theorem \ref{mr}  are explained in section \ref{assumptions}.}

\begin{theo}[Pulse generation at the boundary]\label{mr}
Consider  the boundary problem \eqref{ia1} on the half-space $\mathbb{R}_t\times \mathbb{R}^n_+$, where $\mathcal{L}(t,x,\partial_t,\partial_x)$  has real $C^\infty$ $N\times N$ matrix coefficients and is strictly hyperbolic with respect to $t$, the boundary is noncharacteristic, and $B(t,x')$ is a real $C^\infty$ matrix of size $p\times N$ for some $p\leq N$.   Suppose $f(t,x,u)$ and $g(t,x,\theta_0)$ are $C^\infty$ functions valued in $\mathbb{R}^N$ and $\mathbb{R}^p$ respectively, where $f(t,x,0)=0$ and $g(t,x',\theta_0)$ is supported in $\{(t,x,\theta_0):t\geq 0\text{ and }|\theta_0|\leq 1.\}$.   Suppose also that $(\mathcal{L},B)$ satisfies the uniform Lopatinski condition at $(t,x)=0$.\footnote{This condition determines $p\leq N$.  The number $p$ is equal to the number of positive eigenvalues of the matrix $A_n$ in \eqref{ia1}; see Remark \ref{incoming}.}   

Then the exact solution $u^\eps$ of \eqref{ia1} has the expansion
\eqref{i2}.   More precisely,   if $U_0(t,x,\theta)$ as in \eqref{i4} denotes the leading profile whose construction is outlined above, and $\psi=(\psi_1,\dots,\psi_N)$ 
for $\psi_k$ satisfying \eqref{i4z}, then for $m>\frac{n+5}{2}$ there exist $\eps_0>0$ and an open set $U\ni 0$  independent of $\eps$ such that
\begin{align}\label{mra}
\left|u^\eps-U_0(t,x,\theta)|_{\theta=\frac{\psi}{\eps}}\right|_{L^\infty\cap N^m(U)}\lesssim \sqrt{\eps}\text{ for }\eps\in (0,\eps_0].
\end{align}
In the expansion \eqref{i4} of $U_0$, only the \emph{incoming} profiles can be nonzero and, except for special choices of $g(t,x',\theta_0)$, they are all nonzero.\footnote{A careful definition of incoming/outgoing profiles is given in Remark \ref{incoming}.  Roughly, a profile $\sigma_k(t,x,\theta_k)$ is incoming when its support propagates along curves that enter the domain $x_n\geq 0$ as $t$ increases, and is outgoing in the opposite case.   When $\sigma_k$ is incoming, we call the associated phase $\psi_k$ and surface $\Sigma_k=\{\psi_k=0\}$ incoming as well.}  

\end{theo}

\begin{rem}\label{geo}

\textup{\textbf{More general domains $\mathcal{D}$   and initial surfaces $S$.}\;\;Theorem \ref{mr}  is local, so because we consider variable coefficient operators,  it applies equally well to the following more general class of  semilinear  boundary problems on $\mathbb{R}^{n+1}_z$.   Let 
\begin{align}
\mathcal{P}(z,\partial_z)=\sum^n_{j=0}P_j(z)\partial_{z_j},
\end{align}
where the $P_j$ are real $C^\infty$ $N\times N$ matrices.  Suppose we are given a smooth hypersurface 
$S=\{z:\alpha(z)=0\}$ that is spacelike at $z=0$, and a smooth noncharacteristic hypersurface $b\mathcal{D}=\{z:\beta=0\}$  that is noncharacteristic at $0$ and intersects $S$ transversally at $0$.\footnote{The surface $S=\{\alpha=0\}$ is  \emph{spacelike} at $0$ if $\mathcal{P}(0,\partial_z)$ is strictly hyperbolic in the direction $d\alpha(0)\neq 0$.  The hypersurface $b\mathcal{D}$ is noncharacteristic at $0$ if $\det(P(0,d\beta(0)))\neq 0$.}    Set $\mathcal{D}=\{z:\beta\geq 0\}$ and for a small enough open set $\mathcal{O}\ni 0$ let  $\Delta:=S\cap b\mathcal{D}\cap \mathcal{O}$, a smooth codimension-two manifold closed in $\mathcal{O}$. 
Let $B(z)$ be a real $C^\infty$ $p\times N$ matrix, and consider the following pulse generation problem on $\mathcal{O}\cap \mathcal{D}$:
\begin{align}\label{pcp}
\begin{split}
&\mathcal{P}(z,\partial_z)u^\eps=f(z,u^\eps)\text{ in }\beta>0\\
&B(z)u^\eps|_{\beta=0}=g\left(z,\frac{\psi_0(z)}{\eps}\right)\text{ on }\beta=0\\
&u^\eps=0 \text{ in }\alpha<0.
\end{split}
\end{align}
for $\psi_0$ as in (b) below. 
We show in section \ref{geometric} that \eqref{pcp}  can be reduced by a change of variables to a model problem  \eqref{ia1} satisfying the hypotheses of Theorem \ref{mr} 
whenever the following conditions are satisfied:}  


\textup{(a) $(\mathcal{P},B)$ satisfies the uniform Lopatinski condition at $0\in b\mathcal{D}\cap S$ when $\alpha$ is taken as a time variable; see Remark \ref{ulcoord}.}

\textup{(b) The map $\psi_0:b\mathcal{D}\to \mathbb{R}$ is $C^\infty$ and satisfies
$\Delta=\{z\in b\mathcal{D}:\psi_0(z)=0\}$ and $d\psi_0(0)\neq 0$.}



\textup{(c) The functions  $f(z,u)$ and $g(z,\theta_0)$ are $C^\infty$ functions valued in $\mathbb{R}^N$ and $\mathbb{R}^p$ respectively, where $f(z,0)=0$ and $g(z,\theta_0)$ is supported in $\{(z,\theta_0):\alpha\geq 0\text{ and }|\theta_0|\leq 1.\}$.}

\end{rem}


Next we state a closely related reflection result for the following problem on $\mathcal{O}\cap \mathcal{D}$:
\begin{align}\label{pcp3}
\begin{split}
&\mathcal{P}(z,\partial_z)u^\eps=f(z,u^\eps)\text{ in }\beta>0\\
&B(z)u^\eps|_{\beta=0}=0 \text{ on }\beta=0\\
&u^\eps=u^\eps_o \text{ in }\alpha< -\gamma \text{ for some }\gamma>0,
\end{split}
\end{align}
where $(\mathcal{P},B)$, $f$,  $\alpha$, and $\beta$ are as in Remark \ref{geo}, and $u^\eps_0$ is a given outgoing pulse concentrated near a characteristic surface $\Sigma=\{z:\zeta=0\}$. We assume that $\Sigma\ni 0$ intersects $b\mathcal{D}=\{\beta=0\}\ni 0$ transversally starting at ``time" $\alpha=0$, where $\alpha(0)=0$.\footnote{We do \emph{not} assume that $\Sigma\cap \{\beta=0\}=\Sigma\cap \{\alpha=0\}$ or that $\{\alpha=0\}\cap\{\beta=0\}=\Sigma\cap\{\beta=0\}$ near $0$.} Let $\mathcal{M}_\Sigma$ denote the set of smooth vector fields on $\mathcal{O}$ that are tangent to $\Sigma$.  

\begin{theo}[Reflection of pulses]\label{reflection}
For $m>\frac{n+5}{2}$ consider the problem \eqref{pcp3}, where the outgoing pulse is assumed to satisfy $\mathcal{P}(z,\partial_z)u^\eps_o=f(z,u^\eps_0)$ on $\mathcal{O}\ni 0$ and to
have an expansion
\begin{align}\label{out}
u^\eps_o(z)=\tau\left(z,\frac{\zeta}{\eps}\right)r(z)+w^\eps_o(z)\text{ with }|w^\eps_o|_{L^\infty\cap N^m(\mathcal{O},\mathcal{M}_\Sigma)}=o(1)\text{ as }\eps\to 0
\end{align}
for some smooth scalar outgoing profile $\tau(z,\theta_{out})$ supported in $|\theta_{out}|\leq 1$.  We assume that $\tau$ satisfies a  transport equation, namely \eqref{r4}, on $\mathcal{O}$.  To insure corner compatibility we assume that both  $u^\eps_0|_{\alpha<-\gamma}$ and $\tau|_{\alpha<-\gamma}$ vanish on a neighborhood of $\{\beta=0\}$.  We also assume that $u^\eps_0$ is chosen so that the solution to \eqref{pcp3} exists on some neighborhood of $0$ independent of $\eps$.\footnote{  This can be arranged by an application of Theorem 2.1.5 of \cite{metajm}.  The point is to insure that the solution exists long enough for reflection to occur.   The outgoing pulse \eqref{out} can be constructed as in \cite{ar2}.}   

Then  there exists an open set $U\ni 0$  such that the exact solution $u^\eps$ of \eqref{pcp3} satisfies
\begin{align}\label{expansion}
\left|u^\eps(z)-\left[\sum^p_{k=1}\sigma_k\left(z,\frac{\psi_k}{\eps}\right)r_k(z)+\tau\left(z,\frac{\zeta}{\eps}\right)r(z)\right]\right|_{L^\infty\cap N^m(U)}= o(1)\text{ as }\eps \to 0, 
\end{align}
where the $\psi_k$  and $\sigma_k$  are smooth \emph{incoming} characteristic phases and scalar profiles constructed as in section \ref{application}.    If as in the construction of \cite{ar2} we have $|w^\eps_o|_{L^\infty\cap N^m(\mathcal{O},\mathcal{M}_\Sigma)}\lesssim \eps$ in \eqref{out}, then we obtain the rate of convergence $\sqrt{\eps}$ in \eqref{expansion}.

Let $\tilde \Delta$ be the closed codimension-two submanifold of $U$ given by $\Sigma\cap b\mathcal{D}$, and set 
$\Sigma_N:=\Sigma=\{\zeta=0\}$ and $\Sigma_k=\{\psi_k=0\}$ for $k=1,\dots,p$.  Then we have on $U$:
\begin{align}\label{inters}
\Sigma_i\cap\Sigma_j=\tilde \Delta \text{ for }i\neq j\text{ and } \Sigma_i\cap\{\beta=0\}=\tilde\Delta\text{ for all }i,
\end{align}
where all intersections are transversal.

\end{theo}

It is not clear to us how to derive Theorem \ref{reflection} as a strict corollary of Theorem \ref{mr}.  Instead, we show in section \ref{application} that after some geometric preparation,  Theorem \ref{reflection} can be obtained as  a corollary of the proof of Theorem \ref{mr}.

\begin{rem}\label{extend}
\textup{\textbf{Some extensions.}}


\textup{\textbf{(1) The Cauchy problem.}  With slight modification the methods of this paper yield a rigorous construction of pulses with nonlinear phases for the Cauchy problem for variable-coefficient strictly hyperbolic $N\times N$ systems in dimensions $n>1$.  Pulse initial data at $t=0$ of the form $u^\eps|_{t=0}=g(x,\frac{\psi_0(x)}{\eps})$ gives rise to pulses in $t>0$ concentrated on the $N$ characteristic hypersurfaces emanating from $\Delta=\{(t,x):t=\psi_0(x)=0\}$.   Conormal estimates of \cite{metduke} can be used for the error analysis.}

\textup{\textbf{(2) Quasilinear or nonstrictly hyperbolic problems.} \;\;
We expect that  the construction and estimation of the approximate solution $u^\eps_a$ to the semilinear problem considered here extends in a straightforward manner to variable-coefficient, \emph{quasilinear} hyperbolic systems satisfying our other assumptions.   But to carry out a similar error analysis, one would need to perform the clearly nontrivial task of extending the results of \cite{metajm} from the semilinear to the quasilinear case.     Modifying  the strict hyperbolicity assumption
made here would require a corresponding modification of \cite{metajm}, which assumes strict hyperbolicity. }  

\end{rem}


\section{Assumptions and some notation} \label{assumptions}



We now list with some discussion the structural assumptions made in Theorem \ref{mr} on the operators 
\begin{align}\label{sh0}
 \mathcal{L}(t,x,\partial_t,\partial_x)=\partial_t +\sum^n_{j=1}A_j(t,x)\partial_j\text{ and }B(t,x')
\end{align}
appearing in \eqref{ia1}.

\begin{ass}[Strict hyperbolicity]\label{sh}
Let $(\tau,\xi)\in \mathbb{R}_\tau\times\mathbb{R}^n_\xi$ denote variables dual to $(t,x)$.   The operator $\mathcal{L}(t,x,\partial_t,\partial_x)$  is \emph{strictly hyperbolic with respect to $t$} on an open neighborhood $\mathcal{O}\subset \mathbb{R}^{1+n}$ of $(t,x)=0$.   That is,   there exist  functions
$\tau_i(t,x,\xi): C^\infty(\mathcal{O}\times (\mathbb{R}^n_\xi\setminus 0))\to \mathbb{R}$, $i=1,\dots, N$, positively homogeneous of degree one in $\xi$,  such that 
\begin{align}\label{sh1}
\begin{split}
&\tau_1<\tau_2<\dots< \tau_N\text{ on }\mathcal{O}\times (\mathbb{R}^n_\xi\setminus 0)\\
&\det\left(\tau I+\sum^n_{j=1}A_j(t,x)\xi_j\right)=\prod^N_{i=1}(\tau+\tau_i(t,x,\xi)).
\end{split}
\end{align}
Thus, for each $(t,x,\xi)\in \mathcal{O}\times (\mathbb{R}^n_\xi\setminus 0)$, the matrix $\sum^n_{j=1}A_j(t,x)\xi_j$ has $N$ distinct \emph{real} eigenvalues $\tau_i(t,x,\xi)$.\footnote{If $\mathcal{L}(t,x,\partial_t,\partial_x)=A_0(t,x)\partial_t+\sum^n_{j=1}A_j(t,x)\partial_{x_j}$, we say $\mathcal{L}$ is strictly hyperbolic with respect to $t$ if $A_0$ is invertible and the condition \eqref{sh1}  is satisfied by $A_0^{-1}\mathcal{L}$.}
\end{ass}



\begin{ass}[Noncharacteristic boundary]\label{nonch}
The matrix $A_n(t,x)$  is invertible 
on $\mathcal{O}$. 
\end{ass}

\begin{ass}[Uniform Lopatinski condition]\label{ul}
The pair of operators $(\mathcal{L},B)$ satisfies the uniform Lopatinski (a.k.a. uniform Kreiss) condition at $(t,x)=0$.
\end{ass}

We refer to \cite{kreiss}, to Definition 3.8, Chapter 7 of \cite{CP}, or to Remark 9.7 of \cite{bs} for the exact statement of this condition, which guarantees that the $L^2$ norm of the solution $u$ of  the linear  boundary problem on the half-space $\mathbb{R}\times \mathbb{R}^n_+$ $$\mathcal{L}(0,\partial_t,\partial_x)u=f \text{ in }x_n>0,\; B(0)u=g\text{ on }x_n=0,\; u=0 \text{ in }t<0$$
can be estimated, along with $\langle u|_{x_n=0}\rangle_{L^2}$, in terms of the $L^2$ norms of $f$ and $g$.\footnote{Here take $f$ and $g$ to be in $L^2$ on their respective domains and equal to zero in $t<0$.}  This condition is assumed in a few of the results of \cite{metajm} that we apply in this paper.   We also need to use it at one point in the construction of the leading profile $U_0$; assumption \ref{ul}  implies that the $p\times p$ boundary matrix $\begin{pmatrix}Br_1&\cdots&Br_p\end{pmatrix}$ in \eqref{c11} is invertible for $y$ near $0$.

\begin{rem}\label{ulcoord}
\textup{ The uniform Lopatinski condition at $(t,x)=0$ is a coordinate invariant, ``open"  condition that depends, in the notation of Remark \ref{geo},  only on $(\mathcal{L},B)$, a noncharacteristic hypersurface $b\mathcal{D}$, and a spacelike hypersurface $S$ that intersects $b\mathcal{D}$ transversally at $(t,x)=0$; see section 8.1 of Chapter 7 of \cite{CP}.\footnote{When the uniform Lopatinski condition holds at $(t,x)=0$, it holds in an open neighborhood of $0$.}  We use this coordinate invariance in section \ref{geometric} to reduce the problem \eqref{pcp} to our model problem.}


\end{rem}

In most of the paper instead of the $(t,x)$ notation for points in $\mathbb{R}^{1+n}$, we will use the following notation, which is more convenient for boundary problems.  Set 
\begin{align}\label{ynota}
\begin{split}
&y=(y_0,y'',y_n), \text{ where }t=y_0,\; x'=y'', \text{ and }x_n=y_n\\
&y'=(y_0,y'').  
\end{split}
\end{align}
Denote the corresponding dual variables by $\eta=(\eta_0,\eta'',\eta_n)=(\eta',\eta_n)$.



We now rewrite \eqref{ia1} after applying $A_n^{-1}$ on the left and redefining $A_n^{-1}f$ to be $f$ as 
\begin{align}\label{a1f}
\begin{split}
&L(y,\partial_{y})u^\eps=\partial_nu^\eps+\sum^{n-1}_{j=0}B_j(y)\partial_ju^\eps=f(y,u^\eps)\text{ in }y_n>0,\;\;\text{ where }B_0=A_n^{-1}\text{ and }\partial_0=\partial_t\\
&B(y')u^\eps|_{y_n=0}=g(y',\theta_0)|_{\theta_0=\frac{y_0}{\eps}}:=b^\eps(y')\\
&u^\eps=0\text{ in }y_0<0.
\end{split}
\end{align}

We end this section by collecting   some notation.

\begin{nota}

1.  Let $U\ni 0$ always denote some bounded connected open set in the half-space $$\mathbb{R}^{1+n}_+=\mathbb{R}_{y_0}\times \mathbb{R}^n_+=\{y:y_n\geq 0\}.$$   
Let $bU=\{y':(y',0)\in U\cap \{y_n=0\}\}$.
The set $U$ may need to be taken sufficiently small (independently of $\eps$!) for its use  in a given context to make sense.

2. If $f^\eps:U\to \mathbb{R}^M$ and $p\geq 0$, we write 
$|f^\eps(y)|\lesssim \eps^p$ if there exist positive constants $\eps_0<1$ and $C$ such that $|f^\eps(y)|\leq C\eps^p$ for all $y\in U$ and $\eps\in (0,\eps_0]$.
Similarly, if $\|\cdot \|$ is some norm, we write $\|f^\eps\|\lesssim \eps^p$ if there exist positive constants $\eps_0<1$ and $C$such that $\|f^\eps\|\leq C\eps^p$ for all $\eps\in (0,\eps_0]$.

3.  Denote by $C^\infty_b({U},\mathbb{R}^M)$ the set of $C^\infty$ functions on $U$ and valued in $\mathbb{R}^M$, which are bounded along with all their derivatives on $U$. 

4.  The phases $\psi_k$ constructed in section \ref{eikonal} satisfy $\psi_k(y',0)=y_0$ for all $k$.   We sometimes (for example, in \eqref{a9w}) write $\psi_0(y')$ in place of $y_0$ as a reminder of this.

5.  We often write $d'=\partial_{y'}$.

6. Set $\mathbb{N}_0=\{0,1,2,3,\dots\}$.

\end{nota}

\section{Construction of the approximate solution $u^\eps_a$}

In this section we construct $u^\eps_a=U_0(y,\theta)|_{\theta=\frac{\psi}{\eps}}+\eps U^\eps_1(y)$ such that \eqref{i5b} holds.
The leading term is constructed as a sum of $N$ pieces
 \begin{align}\label{a7x}
U_0\left(y,\frac{\psi}{\eps}\right) = \sum^N_{k=1}\sigma_k\left(y,\frac{\psi_k(y)}{\eps}\right)r_k(y), \;k=1,\dots,N,
 \end{align}
 where each piece is a pulse concentrated near the codimension-one surface in spacetime defined by $\Sigma_k=\{y:\psi_k(y)=0\}$.   The first step is to construct  the characteristic phases $\psi_k$ that define these $N$ surfaces.  The $N$ surfaces all contain the codimension-two surface defined by $\Delta:=\{y:y_0=y_n=0\}$.  The construction takes place on a small enough neighborhood of  the origin $y=0$.

\subsection{Eikonal equations for the phases}\label{eikonal}

 The functions $\psi_k$ are constructed as solutions of a nonlinear Cauchy problem classically known as an \emph{eikonal} problem.     
 To formulate this problem we first define the eigenvectors $r_k(y)$ that appear in \eqref{a7x} along with the associated eigenvalues.  
With $L$ as in \eqref{a1f} we set
\begin{align}\label{a2}
L(y,\eta)=\eta_n I+\sum^{n-1}_{j=0}B_j(y)\eta_j:=\eta_n I+\mathcal{A}(y,\eta').
\end{align}
We claim that if $\eta'$ satisfies $|\eta''|\leq \delta |\eta_0|$ for $\delta>0$ small enough, then the matrix $\mathcal{A}(y,\eta')$ has $N$ distinct \emph{real} eigenvalues $\lambda_m(y,\eta')$, $m=1,\dots,N$; see Remark \ref{hyp}.  Let $L_m(y,\eta')$, $R_m(y,\eta')$
denote associated  left and right eigenvectors, which can be chosen positively homogeneous of degree zero in $\eta'$ and such that $L_k(y,\eta')R_j(y,\eta')=\delta_{kj}$. \footnote{We take the $L_m$ to be row vectors, and the $R_m$ to be column vectors, so for example, $L_m\mathcal{A}=\lambda_mL_m$ at $(y,\eta')$.}  

\begin{rem}\label{hyp}
\textup{It is not generally true that $\mathcal{A}(y,\eta')$ has $N$ distinct real eigenvalues for \emph{all} $\eta'\neq 0$.  
This is true, though, for the particular  directions $\eta'=(\pm1,0,\dots,0)$.   Indeed, note that 
$$\mathcal{A}(y,\pm 1,0,\dots,0)=\pm B_0(y)=\pm A_n^{-1}(y),$$
and observe that $A_n(y)$ has $N$ distinct real eigenvalues as a consequence of  strict hyperbolicity, that is,  Assumption \ref{sh}.  Moreover,  since $y_n=0$ is noncharacteristic, those eigenvalues must all be nonzero. It follows that $\mathcal{A}(y,\eta')$ has $N$ distinct real nonzero eigenvalues whenever $\eta'=(\eta_0,\eta'')$ satisfies
\begin{align}\label{gam}
\eta\in \Gamma:=\{\eta':|\eta''|\leq \delta |\eta_0|\} 
\end{align}
for $\delta>0$ small enough.}

\end{rem}

   We define the following projection operators for use in section \ref{e}:

\begin{defn}\label{pi}
(a) We define $\Pi_k(y,\eta')$ to be the projection on $\mathrm{span}\;R_k(y,\eta')$ in the decomposition 
$$\mathbb{R}^N=\oplus^N_{m=1}\mathrm{span}\;R_m(y,\eta').$$
Explicitly, for $x\in\mathbb{R}^N$ we have $\Pi_k(y,\eta')x=(L_k(y,\eta')x)R_k(y,\eta')$.

(b)  With $d'=\partial_{y'}$ set $r_k(y):=R_k(y,d'\psi_k(y))$,  $l_k(y):=L_k(y,d'\psi_k(y))$, and $\pi_k(y)=\Pi_k(y,d'\psi_k(y))$. 
\end{defn}

The next (classical) lemma \cite{lax} and its corollary will be used in section \ref{U0}  to determine ``transport equations" satisfied by the leading profiles.

\begin{lem}\label{a2a}
For $j=0,\dots,n-1$ we have
\begin{align}\label{a2b}
L_m(y,\eta')B_j(y)R_m(y,\eta')=\partial_{\eta_j}\lambda_m(y,\eta').
\end{align} 
\end{lem}
\begin{proof}
Differentiate the equation 
$$0=\left[-\lambda_m(y,\eta')I+\sum^{n-1}_{j=0}B_j(y)\eta_j)\right]R_m(y,\eta')=L(y,\eta',-\lambda_m)R_m$$
with respect to $\eta_j$ to obtain
$$\left[-\partial_{\eta_j}\lambda_m(y,\eta')I+B_j(y)\eta_j)\right]R_m(y,\eta')+L(y,\eta',-\lambda_m) \partial_{\eta_j}R_m(y,\eta')=0.$$
Apply $L_m(y,\eta')$ on the left to obtain \eqref{a2b}.
\end{proof}

The following immediate corollary is used in section \ref{U0}.
\begin{cor}\label{a2c}
We have $l_m(y)B_j(y)r_m(y)=\partial_{\eta_j}\lambda_m(y,d'\psi_m).$
\end{cor}

We want to construct characteristic phases  $\psi_k(y)$ satisfying the  initial value problem near $y=0$:
\begin{align}\label{a3}
\begin{split}
&(a) \det L(y,d\psi_k)=\det\left(\partial_n\psi_k I+\sum^{n-1}_{j=0}B_j(y)\partial_j\psi_k\right)=0\\
&(b) \psi_k|_{y_n=0}=y_0.
\end{split}
\end{align}
We arrange for \eqref{a3} to hold by solving the eikonal initial value problem for the unknown $\psi_k$:
\begin{align}\label{a4}
\partial_n\psi_k=-\lambda_k(y,d' \psi_k), \;\; \psi_k|_{y_n=0}=y_0.
\end{align}
Since $L(y,\partial_y)$ has variable coefficients, the interior phases $\psi_k$ will be nonlinear. 

The initial condition and the  implicit function theorem imply that near $y=0$ we can write
\begin{align}\label{a5}
\psi_k(y)=(y_0-\phi_k(y'',y_n))\beta_k(y), 
\end{align}
for some smooth $\phi_k$ and $\beta_k$.   Since 
\begin{align}\label{a5z}
\psi_k|_{y_n=0}=y_0=(y_0-\phi_k(y'',0))\beta_k(y_0,y'',0),
\end{align}
we conclude $\phi_k(y'',0)=0$ by setting $y_0=\phi_k(y'',0)$ in \eqref{a5z}, and thus \eqref{a5z} implies $\beta_k(y_0,y'',0)=1$. 
Using the positive homogeneity of $\lambda_k$ of degree one, we see that \eqref{a4} implies\footnote{ To see this substitute $\partial_n\psi_k=(-\partial_n\phi_k)\beta_k+(y_0-\phi_k)\partial_n\beta_k)$ in \eqref{a4} and set $y_0=\phi_k$.}  
\begin{align}\label{a7}
\partial_n\phi_k(y'',y_n)=\lambda_k(\phi_k,y'',y_n,1,-d_{y''}\phi_k), \;\;\phi_k(y'',0)=0.
\end{align}

\begin{rem}\label{aa7}
\textup{We use the functions $\phi_k$ in section \ref{O0} to show that near $y=0$ the characteristic surfaces $\Sigma_j=\{\psi_j=0\}$ 
satisfy 
\begin{align}\label{aa8}
\Sigma_j\cap \Sigma_k=\Delta \text{ for }j\neq k, \;\Sigma_k\cap\{x_n=0\}=\Delta, \; \Sigma_k\cap \{y_0=0\}=\Delta
\end{align}
where  all intersections are transversal.  Observe that while the surface $x_0-\phi_j=0$ is characteristic, the surface $x_0-\phi_j=c$ need not be characteristic for $c\neq 0$.   The functions $\psi_j$, however, have the property that $\psi_j=c$ is characteristic for all $c$ near $0$, a property that is needed in the construction of the approximate solution $u^\eps_a$. }
\end{rem}

\begin{rem}[Caution]\label{caution}
\textup{ By Definition \ref{pi} for a given $k$ we have 
\begin{align}\label{40b}
\mathbb{R}^N=\oplus^N_{m=1}\mathrm{span}\;R_m(y,d'\psi_k(y)),
\end{align}
where $R_k(y,d'\psi_k(y))=r_k(y)$, but $R_m(y,d'\psi_k(y))\neq r_m(y)$ for $m\neq k$.   However, \eqref{40b} implies that at $y=(y',0)$ we have 
\begin{align}\label{40c}
\mathbb{R}^N=\oplus^N_{m=1}\mathrm{span}\;r_m(y',0),
\end{align}
since $d'\psi_k(y',0)=d'y_0$ for all $k$.  
For $y$ near $0$ \eqref{40c} implies by continuity
\begin{align}\label{40d}
\mathbb{R}^N=\oplus^N_{m=1}\mathrm{span}\;r_m(y),
\end{align}
\emph{but} the projection on  $\mathrm{span}\;r_k(y)$ in the decomposition \eqref{40d} is \emph{not} $\pi_k(y)$, unless $y=(y',0)$.\footnote{The projection $\pi_k(y):=\Pi_k(y,d'\psi_k(y))$ is, rather, the projection on the $k-$th summand in the decomposition \eqref{40b}.} 
 Similarly, while we have $I=\sum_k\pi_k(y',0)$,   this equality \emph{fails} for general $y$ near $0$.    While $l_k(y',0)r_j(y',0)=\delta_{jk}$, this equality also fails for general $y$ near $0$.}

\end{rem}





\subsection{Preliminary computations}\label{strategy}

In this section we carry out some preliminary computations needed for the construction of the approximate solution $u^\eps_a$ to the system \eqref{ia1}.  
We look for an approximate solution of the form 
\begin{align}\label{ua}
\begin{split}
&u^\eps_a=U_0(y,\theta)|_{\theta=\frac{\psi}{\eps}}+\eps U^\eps_1(y),
 \text{ where }U_0(y,\theta)=\sum^N_{k=1}\sigma_k(y,\theta_k)r_k(y).
\end{split}
\end{align}
The $\sigma_k(y,\theta_k)$ are scalar functions with compact support in $\theta_k$ that will be determined by solving ``transport equations". 

Below we let $\theta_0$ and $\xi_n$ be placeholders for $\frac{y_0}{\eps}$ and $\frac{y_n}{\eps}$ respectively.
The ``corrector" $U^\eps_1$ has the form
\begin{align}\label{corr1}
U^\eps_1(y)=\mathcal{U}^\eps_1(\theta,\theta_0,\xi_n)|_{\theta=\frac{\psi}{\eps}, \theta_0=\frac{\psi_0}{\eps},\xi_n=\frac{y_n}{\eps}},
\end{align}
with
\begin{align}\label{corrector}
\begin{split}
&\mathcal{U}^\eps_1(y,\theta,\theta_0,\xi_n)):=V(y,\theta)+W^\eps(y,\theta_0,\xi_n), \text{ where }\\
&\quad V(y,\theta)=\sum^N_{k=1}V_k(y,\theta_k)\text{ and }\\
&\quad W^\eps(y,\theta_0,\xi_n)=\chi^\eps(y_0,y_n)W(y,\theta_0,\xi_n)\text{ with }W(y,\theta_0,\xi_n)=\sum^N_{k=1}t_k(y,\theta_0,\xi_n)r_k(y',0)
\end{split}
\end{align}
for a cutoff function $\chi^\eps$, vector functions $V_k$, and scalar functions $t_k$ to be constructed.

We want $L(y,\partial_y)u^\eps_a-f(y,u^\eps_a)$ to be ``small".\footnote{This means $o(1)$ in $L^\infty\cap N^m(U)$.}   To compute $L(y,\partial_y)u^\eps_a$  we begin by computing (set $B_n(y)=I$) for a fixed $k\in \{1,\dots,N\}$:
\begin{align}\label{a8} 
\begin{split}
&B_j\partial_j[\sigma_k(y,\frac{\psi_k(y)}{\eps})r_k(y)]=\left[\frac{1}{\eps}(\partial_j\psi_k(y) (\partial_{\theta_k} \sigma_k(y,\theta_k))B_jr_k+(\partial_j\sigma_k)B_jr_k+\sigma_kB_j\partial_jr_k\right]|_{\theta=\frac{\psi}{\eps}},\text{ so }\\
&L(y,\partial_y)[\sigma_k(y,\frac{\psi_k(y)}{\eps})r_k(y)]:=\left[\frac{1}{\eps}L(y,d\psi_k)(\partial_{\theta_k}\sigma_k)r_k+\sum^n_{j=0}(\partial_j\sigma_k) B_jr_k+\sum^n_{j=0}\sigma_kB_j\partial_jr_k\right]|_{\theta=\frac{\psi}{\eps}}=\\
&\qquad  \left[\frac{1}{\eps}L(y,d\psi_k)(\partial_{\theta_k}\sigma_k)r_k+(L(y,\partial_y)\sigma_k)r_k+\sigma_kL(y,\partial_y)r_k\right]|_{\theta=\frac{\psi}{\eps}}.
\end{split}
\end{align}
Define (with $\psi_0(y')=y_0$)
\begin{align}\label{a9w}
\mathcal{L}_1(y,d\psi,\partial_\theta)=\sum^N_{k=1}L(y,d\psi_k)\partial_{\theta_k} \text{ and }\mathcal{L}_2(y,d\psi_0,\partial_{\theta_0,\xi_n})=L(y,d\psi_0)\partial_{\theta_0}+\partial_{\xi_n}.
\end{align}
and observe that \eqref{a8} implies in view of the polarization of $U_0$:
\begin{align}\label{a9}
\begin{split}
&L(y,\partial_y)U_0(y,\frac{d\psi}{\eps})=\left[\frac{1}{\eps}\mathcal{L}_1(y,d\psi,\partial_\theta)U_0(y,\theta)+L(y,\partial_y)U_0(y,\theta)\right]|_{\theta=\frac{\psi}{\eps}}\\
&\quad \left[\frac{1}{\eps}\mathcal{L}_1(y,d\psi,\partial_\theta)U_0(y,\theta)+\sum^N_{k=1}(L(y,\partial_y)\sigma_k(y,\theta_k))r_k(y)+\sum^3_{k=1}\sigma_k(y,\theta_k)L(y,\partial_y)r_k(y)\right]|_{\theta=\frac{\psi}{\eps}}=\\
&\qquad \left[\sum^N_{k=1}(L(y,\partial_y)\sigma_k(y,\theta_k))r_k(y)+\sum^N_{k=1}\sigma_k(y,\theta_k)L(y,\partial_y)r_k(y)\right]|_{\theta=\frac{\psi}{\eps}}.
\end{split}
\end{align}
Similarly we have
\begin{align}
\begin{split}
&L(y,\partial_y) V(y,\frac{\psi}{\eps}) =\left[\frac{1}{\eps}\mathcal{L}_1(y,d\psi,\partial_\theta)V(y,\theta)+L(y,\partial_y)V(y,\theta)\right]|_{\theta=\frac{\psi}{\eps}}\\
&L(y,\partial_y)W^\eps(y,\frac{\psi_0}{\eps},\frac{y_n}{\eps}) =\left[\frac{1}{\eps}\mathcal{L}_2(y,d\psi_0,\partial_{\theta_0,\xi_n})W^\eps(y,\theta_0,\xi_n)+L(y,\partial_y)W^\eps(y,\theta_0,\xi_n))\right]|_{\theta_0=\frac{\psi_0}{\eps}, \xi_n=\frac{y_n}{\eps}}.
\end{split}
\end{align}
We can expand:\footnote{Here $K(y,U_0,\eps \mathcal{U}^\eps_1)=\int^1_0 f_u(y,U_0+s\eps \mathcal{U}^\eps_1)ds.$}
\begin{align}
f(y,u^\eps_a)=[f(y,U_0(y,\theta))+\eps K(y,U_0,\eps \mathcal{U}^\eps_1)\mathcal{U}^\eps_1]|_{\theta=\frac{\psi}{\eps}, \theta_0=\frac{\psi_0}{\eps},\xi_n=\frac{y_n}{\eps}}.
\end{align}
Thus, we have $L(y,\partial_y)u^\eps_a-f(y,u^\eps_a)=\eps^{-1}\mathcal{E}_{-1}^\eps+\eps^0\mathcal{E}_0^\eps+\eps\mathcal{E}_1^\eps:=r^\eps_a$,  where 
\begin{align}\label{a9z}
\begin{split}
&\mathcal{E}^\eps_{-1}=0,\\
&\mathcal{E}_0^\eps=\left[\sum^N_{k=1}(L(y,\partial_y)\sigma_k)r_k+\sum^N_{k=1}\sigma_k L(y,\partial_y)r_k-f(y,U_0)+\mathcal{L}_1(y,d\psi,\partial_\theta)V+\mathcal{L}_2(y,d\psi_0,\partial_{\theta_0,\xi_n})W^\eps\right]|_{\theta=\frac{\psi}{\eps}, \theta_0=\frac{\psi_0}{\eps},\xi_n=\frac{y_n}{\eps}}\\
&\mathcal{E}_1^\eps=\left[ -K(y,U_0,\eps \mathcal{U}^\eps_1)\mathcal{U}^\eps_1+L(y,\partial_y) V+L(y,\partial_y)W^\eps\right]|_{\theta=\frac{\psi}{\eps}, \theta_0=\frac{\psi_0}{\eps},\xi_n=\frac{y_n}{\eps}}.
\end{split}
\end{align}
Our next task is to choose $U_0$, $V$, and $W^\eps$ to make $\mathcal{E}_0^\eps$ as small as possible.

\subsection{Temporary choice of a quadratic nonlinearity} We now make a preliminary choice of $f(y,u)$ as an at most quadratic function of $u$ 
in order to make all pulse interactions readily visible in the construction of the corrector. We show in section \ref{genf} that everything generalizes to the case where $f$ is any smooth function as in Theorem \ref{mr}.

Any function $u:\mathbb{R}^{1+n}_y\to\mathbb{R}^N$  can be written $$u(y)=\sum_ks_k(y)r_k(y)$$ for $r_k$ as in Definition \ref{pi}  and some scalar functions $s_k$. Set $s=(s_1,\dots,s_N)$.
We take $f$  for now to have only   linear and a quadratic parts with respect to $s$:
\begin{align}\label{c1}
\begin{split}
&f(y,u):=f^l(y,s)+f^q(y,s), \text{ where }\\
&f^l(y,s)=\sum_k f^l_k(y,s)r_k=\sum_k\left(\sum_mf^l_{km}(y)s_m\right)r_k\\
&f^q(y,s)=\sum_k f^q_k(y,s)r_k=\sum_k\left(\sum_{m\leq p}f^q_{kmp}(y)s_ms_p\right)r_k.
\end{split}
\end{align}
In the case where $u$ is a function of $(y,\theta)$, the same definition of $f(y,u)$ applies with $s_k=s_k(y,\theta)$.

\subsection{Definition of the projection operator $E$}\label{e}
In this section we define a projection operator $E=P\circ S$ (or $PS$)  that we will use as follows in our effort to make $\mathcal{E}_0^\eps$ small.
We first write (with $U_0=U_0(y,\theta)$):
\begin{align}\label{c3y}
\begin{split}
&\mathcal{F}(y,U_0,\partial_y U_0):=L(y,\partial_y)U_0-f(y,U_0)=\\
&\qquad \sum^N_{k=1}(L(y,\partial_y)\sigma_k)r_k+  \sum^N_{k=1}\sigma_k L(y,\partial_y)r_k-f(y,U_0)=E\mathcal{F}+(1-E)\mathcal{F}.
\end{split}
\end{align}
We will construct  $U_0$ to satisfy
\begin{align}\label{c3} 
E\mathcal{F}=PS\mathcal{F}=0.
\end{align}
Next we write\footnote{A belated motivation for the decomposition of $\mathcal{F}$ given by \eqref{c3}, \eqref{c3z} is given in Remark \ref{d10y}.} 
\begin{align}\label{c3z}
\begin{split}
&(1-E)\mathcal{F}=H_1+H_2,  \text{ where }\\
&H_1=(1-E)S\mathcal{F} =(1-P)S\mathcal{F}\text{ and }H_2=(1-E)(1-S)\mathcal{F}=(1-S)\mathcal{F}.
\end{split}
\end{align}
We will construct  $V$ and $W$ to satisfy:\footnote{See Remark \ref{c5}.}
\begin{align}\label{c4}
\begin{split}
&(a) \mathcal{L}_1(y,d\psi,\partial_\theta)V+H_1=0\\
&(b) \mathcal{L}_2(y',0,d\psi_0,\partial_{\theta_0,\xi_n})W+\mathcal{H}_2=0,
\end{split}
\end{align} 
where $\mathcal{H}_2(y,\theta_0,\xi_n)$ is a slightly modified version of $H_2(y,\theta)$ specified later; see \eqref{d1}. 
\begin{rem}\label{c4y}\label{explain}
\textup{We use the operator  $\mathcal{L}_{2,0}:=\mathcal{L}_2(y',0,d\psi_0,\partial_{\theta_0,\xi_n})$  in \eqref{c4}(b) rather than
$\mathcal{L}_2=\mathcal{L}_2(y,d\psi_0,\partial_{\theta_0,\xi_n})$, because $\mathcal{L}_{2,0}$ is diagonalizable by the matrix 
$\begin{pmatrix}r_1(y',0)&\cdots &r_N(y',0)\end{pmatrix}$, while we do not see how to diagonalize $\mathcal{L}_2$; see   
\eqref{d10}.   This property of the $r_k(y',0)$ also explains their use, rather than the $r_k(y),$  in \eqref{d9}.}

\textup{The use of  $\mathcal{L}_{2,0}$ here introduces an error (see \eqref{d13})  that is estimated in section \ref{L20}.}
\end{rem}

To define $E=PS$ we first define the action of $S$ on $f(y,U_0)$ using \eqref{c1}, except that now we write $\sigma=(\sigma_1,\dots,\sigma_N)$ instead of $s=(s_1,\dots,s_N)$.     It is convenient to introduce the following notation:
\begin{defn}\label{c4w}
If $f(y,\sigma)$ is any element of $C^\infty(\mathbb{R}^{1+n}_+\times\mathbb{R}^N,\mathbb{R}^N)$  (or $C^\infty(\mathbb{R}^{1+n}_+\times\mathbb{R}^N,\mathbb{R})$), 
set $f^k(y,\sigma_k):=f(y,\sigma_ke_k)$, where $e_k$ is the $k-$th standard basis vector of $\mathbb{R}^N$.   When $f$ is either $f^l$ or $f^q$, we write $f^{l,k}$ or $f^{q,k}$ for $f^k$.
\end{defn}

Roughly, $S$ \emph{selects} the part of $f(y,U_0)$ that can be written as a sum of terms involving just one of the $\sigma_k$.
We set
\begin{align}\label{c4aa}
\begin{split}
&Sf(y,U_0):=\sum^N_{k=1}f^{l,k}(y,\sigma_k)+\sum^N_{k=1}f^{q,k}(y,\sigma_k)=f^l(y,\sigma)+\sum^N_{k=1}f^{q,k}(y,\sigma_k).
\end{split}
\end{align}
The operator $P$ then \emph{polarizes} these terms by applying the projection $\pi_k(y)$ (as in Definition \ref{pi}(b)) to any term that depends only on $\sigma_k$.   Thus,  we obtain
\begin{align}\label{c4d}
\begin{split}
&Ef(y,U_0)=PSf(y,U_0)=\sum^N_{k=1}\pi_k(y)[f^{l,k}(y,\sigma_k)+f^{q,k}(y,\sigma_k)].
\end{split}
\end{align}
The other terms in $\mathcal{F}$,  namely 
\begin{align}\label{c4a}
\sum^N_{k=1}(L(y,\partial_y)\sigma_k)r_k:=\ell_1(y,\partial_yU_0)\text{ and }     \sum^N_{k=1}\sigma_k L(y,\partial_y)r_k:=\ell_2(y,U_0), 
\end{align}
are linear functions of $\sigma_y$ and $\sigma$ respectively.   Parallel to \eqref{c4d} we define
\begin{align}\label{c5z}
\begin{split}
&E\ell_1(y,\partial_yU_0):=\sum^N_{k=1}\pi_k(y)\left[(L(y,\partial_y)\sigma_k(y,\theta_k))r_k(y)\right]\\
&E\ell_2(y,U_0)\sum^N_{k=1}\pi_k(y)\left[\sigma_k(y,\theta_k)L(y,\partial_y)r_k(y)\right].
\end{split}
\end{align}


\begin{rem}\label{c5}
\textup{(a) From \eqref{c3} and \eqref{c4} we see that the choice of $U_0$ makes the polarized part of $S\mathcal{F}$ vanish, and  $V$ ``solves away" the unpolarized part of $S\mathcal{F}$, namely $H_1=(1-P)SF$.  Similarly, $W$ solves away ``most of" the multiphase part of $\mathcal{F}$, namely $H_2=(1-S)\mathcal{F}$.}

\textup{(b) The operators $\mathcal{L}_i$, $i=1,2$ in \eqref{c4} are singular.     The equation  \eqref{c3} both helps to determine the $\sigma_k$ \emph{and} acts as a solvability condition that will allow us to solve the equations \eqref{c4} using bounded profiles $V$ and $W$; see \eqref{c14}.}

\end{rem}

\subsection{Construction of the leading profile $U_0(y,\theta)$}\label{U0}

We now construct $U_0(y,\theta)$ to satisfy $E\mathcal{F}=0$ \eqref{c3} with a  boundary condition corresponding to 
\begin{align}\label{c8}
BU_0(y,\frac{\psi}{\eps})|_{y_n=0}=g(y',\frac{y_0}{\eps})\;\;(\text{recall }\psi_k|_{y_n=0}=y_0).  
\end{align}
That is, we want $U_0(y,\theta)$ to satisfy
\begin{align}\label{c9}
\begin{split}
&(a) E\left[L(y,\partial_y)U_0-f(y,U_0)\right]=0 \text{ in }y_n>0,\\
&(b) BU_0(y',0,\theta_0,\dots,\theta_0)=g(y',\theta_0)\\
&(c) U_0=0\text{ in }y_0<0.
\end{split}
\end{align}

To solve \eqref{c9}  we first use \eqref{c4d} to compute
\begin{align}\label{c6}
\begin{split}
&EL(y,\partial_y)U_0=E\left[\sum^N_{k=1}(L(y,\partial_y)\sigma_k(y,\theta_k))r_k(y)+\sum^N_{k=1}\sigma_k(y,\theta_k)L(y,\partial_y)r_k(y)\right]=\\
&\quad \sum^N_{k=1}\pi_k(y)\left[(L(y,\partial_y)\sigma_k(y,\theta_k))r_k(y)\right]+\sum^N_{k=1}\pi_k(y)\left[\sigma_k(y,\theta_k)L(y,\partial_y)r_k(y)\right]=\\
&\quad\quad \sum^N_{k=1}l_k(y)\left[\sum^n_{j=0}(\partial_j\sigma_k) B_j(y)r_k(y)\right]r_k(y)+\sum^N_{k=1}l_k(y)\left[\sigma_kL(y,\partial_y)r_k\right]r_k(y)=\\
&\quad\quad\quad \sum^N_{k=1} \left[X_k(y,\partial_y)\sigma_k +c_k(y)\sigma_k\right]r_k(y).
\end{split}
\end{align}
In the last line $X_k(y,\partial_y)$, the characteristic vector field associated to the phase $\psi_k$, and the coefficients $c_k(y)$ are defined by
\begin{align}
\begin{split}
&X_k(y,\partial_y)=\partial_n+\sum^{n-1}_{j=0}\partial_{\eta_j}\lambda_k(y,d'\psi_k)\partial_j\\
&c_k(y)=l_k(y)L(y,\partial_y)r_k(y).
\end{split}
\end{align}
We have used Corollary \ref{a2c} here to obtain $X_k$.

\begin{rem}[Incoming vs. outgoing]\label{incoming}

\textup{(a) A phase $\psi_k$  is said to be \emph{incoming} (resp., \emph{outgoing}) when 
the group velocity described by the corresponding vector field $X_k$ is incoming (resp., outgoing). 
We say $X_k$ (or its group velocity) is {incoming} (resp. {outgoing}) when the coefficients of $\partial_n$ and $\partial_{0}=\partial_t$ in $X_k$ have the same (resp. opposite) signs.   
Corollary \ref{a2c} implies that $p$ is equal to the number of positive eigenvalues of $A_n$ in \eqref{ia1}.   
Relabeling if necessary,  we  take the vector fields $X_k$, $k=1,\dots,p$ to be incoming and the others to be outgoing. } 

\textup{(b) Observe that $X_k(y,\partial_y)$ is {tangent} to the characteristic surfaces $\psi_k(y)=c$.\footnote{Use the eikonal equation and the Euler identity for positively homogeneous functions to see this.}}

\end{rem}

Next we use \eqref{c4d} to compute 
\begin{align}\label{c7}
Ef(y,U_0)=\sum^N_{k=1}l_k(y)\left[f^{l,k}(y,\sigma_k)+f^{q,k}(y,\sigma_k)\right]r_k:=\sum^N_{k=1}[d_k(y)\sigma_k+e_k(y)\sigma^2_k]r_k.
\end{align}
With \eqref{c6} and \eqref{c7}  we see that \eqref{c9}(a) is equivalent to
\begin{align}\label{c10}
X_k(y,\partial_y)\sigma_k+c_k(y)\sigma_k-(d_k(y)\sigma_k+e_k(y)\sigma_k^2)=0, \; k=1,\dots,N.
\end{align}
The boundary condition \eqref{c9}(b) is  satisfied provided 
\begin{align}\label{c11} 
\begin{split}
&BU_0(y',0,\theta_0,\dots,\theta_0)=g(y',\theta_0)\Leftrightarrow \sum^N_{k=1}\sigma_kBr_k=g\Leftrightarrow\\
&\qquad \begin{pmatrix}Br_1&\dots&Br_p\end{pmatrix}\begin{pmatrix}\sigma_1\\\vdots\\\sigma_p\end{pmatrix}=g-\sum^N_{k=p+1}\sigma_kBr_k.
\end{split}
\end{align}
We claim $\sigma_k(y,\theta_k)=0$ for $k\geq p+1$.   This  follows from \eqref{c10} since $X_k(y,\partial_y)$ is  \emph{outgoing} for such $k$ and $\sigma_k=0$ in $y_0<0$. The uniform Lopatinski condition implies $\begin{pmatrix}Br_1&\dots&Br_p\end{pmatrix}$ is invertible, so \eqref{c11} determines explicit functions $\sigma^*_1(y',\theta_0)$, $\dots$, $\sigma^*_p(y',\theta_0)$ supported in $y_0\geq 0$ as boundary values of $\sigma_1$, $\dots$, $\sigma_p$.  Except for special choices of $g(y',\theta_0)$, all these boundary values will be nonzero.

Summarizing, we have shown that outgoing profiles $\sigma_k=0$ for $k\geq p+1$,  and that incoming profiles $\sigma_1$, $\dots$,  $\sigma_p$ are determined by solving:
\begin{align}\label{c12}
\begin{split}
&X_k(y,\partial_y)\sigma_k+c_k(y)\sigma_k-(d_k\sigma_k+e_k\sigma_k^2)=0\\
&\sigma_k(y',0,\theta_k)=\sigma_k^*(y',\theta_k) \text{ for }k=1,\dots,p.
\end{split}
\end{align}
This problem, in which $\theta_k$ is just a parameter,  has $C^\infty$ solutions on a fixed time interval, so we now have $U_0(y,\theta)$ defined and satisfying \eqref{c9}  for $y$ in some open set $U\ni 0$.\footnote{The method of characteristics can be used to solve \eqref{c12}.}



\subsection{Construction of the corrector $\eps U^\eps_1(y)$.}\label{corr2}
We begin by constructing the part of the corrector $U^\eps_1$ determined by $V(y,\theta)=\sum^N_{k=1}V_k(y,\theta_k)$ as in \eqref{corrector}.

\subsubsection{The noninteraction term $V(y,\theta)$.} \label{V}
  
The construction of each $V_k$ is essentially the same as the construction of the corrector  in  \cite{ar2}  for the one phase case.  Recall (Definition \ref{pi})  that for $\eta'$ near $(1,0,\dots,0)$ the matrix $\mathcal{A}(y,\eta')$ has $N$ eigenvalues $\lambda_k(y,\eta')$, and corresponding left and right eigenvectors 
$L_k(y,\eta')$, $R_k(y,\eta')$ such that $L_jR_k=\delta_{jk}$.  We have denoted by $\Pi_k(y,\eta')$ the projection on $\mathrm{span}\;R_k(y,\eta')$ in the decomposition 
$$\mathbb{R}^N=\oplus^N_{m=1}\mathrm{span}\;R_m(y,\eta').$$
Thus, for $x\in\mathbb{R}^N$ we have $\Pi_k(y,\eta')x=(L_k(y,\eta')x)R_k(y,\eta')$.

We want $V(y,\theta)$ to satisfy
\begin{align}\label{c13} 
\mathcal{L}_1(y,d\psi,\partial_\theta)V=-H_1,
\end{align}
for $H_1$ as in \eqref{c3z} and $\mathcal{L}_1(y,d\psi,\partial_\theta)=\sum^N_{k=1}L(y,d\psi_k)\partial_{\theta_k}$.  

First we define a partial inverse $Q_k(y)$ for each matrix $L(y,d\psi_k)$ such that
\begin{align}\label{c14}
L(y,d\psi_k)Q_k(y)=Q_k(y)L(y,\psi_k)=1-\Pi_k(y,d'\psi_k)=1-\pi_k(y).
\end{align}
For each $k$ we have 
\begin{align}\label{c15}
\begin{split}
&\mathcal{A}(y,\eta')=\sum^N_{j=1}\lambda_j(y,\eta')\Pi_j(y,\eta')\Rightarrow \mathcal{A}(y,d'\psi_k)=\sum^N_{j=1}\lambda_j(y,d'\psi_k)\Pi_j(y,d'\psi_k),
\end{split}
\end{align}
so using the eikonal equation \eqref{a4} we obtain
\begin{align}
L(y,d\psi_k)=\partial_n\psi_k I+\mathcal{A}(y,d'\psi_k)=\sum^N_{j=1}\left(\lambda_j(y,d'\psi_k)-\lambda_k(y,d'\psi_k)\right)\Pi_j(y,d'\psi_k).
\end{align}
Thus,  
\begin{align}\label{Qk}
Q_k(y):=\sum_{j\neq k}\left(\lambda_j(y,d'\psi_k)-\lambda_k(y,d'\psi_k)\right)^{-1}\Pi_j(y,d'\psi_k)
\end{align}
satisfies \eqref{c14}.

    Since $H_1(y,\theta)=(1-P)S\mathcal{F}$, it has the form
    \begin{align}\label{c16}
    H_1(y,\theta)=\sum^N_{k=1}H_{1k}(y,\theta_k).
    \end{align}
With $f^{l,k}$ and $f^{q,k}$ as in Definition \ref{c4w}, we have 
\begin{align}\label{c16z}
H_{1k}(y,\theta_k)=(1-\pi_k(y))\left[(L(y,\partial_y)\sigma_k)r_k+\sigma_kL(y,\partial_y)r_k-\left(f^{l,k}(y,\sigma_k)+f^{q,k}(y,\sigma_k)\right)\right].
\end{align}
  Using this and \eqref{c14}, we see that if we define
\begin{align}\label{Vk}
V_k(y,\theta_k):=-\int^{\theta_k}_{-\infty}Q_k(y)H_{1k}(y,s)ds,
\end{align}
then $V(y,\theta)=\sum_kV_k(y,\theta_k)$ satisfies \eqref{c13}.  

\begin{rem}\label{decV}

\textup{(a) For each $k$, $H_{1k}(y,\theta_k)$ is a finite sum of terms of the form
\begin{align}\label{c17}
(\partial_y\sigma_k(y,\theta_k))C(y),\; \sigma_k(y,\theta_k) C(y), \text{ or }\sigma_k^2(y,\theta_k) C(y),
\end{align}
where $C(y)$ is an  $\mathbb{R}^N$-valued $C^\infty$ function such that $\pi_k(y)C(y)=0$, and which varies from term to term.\footnote{Observe that a vector $Y$ is in the range of $L(y,d\psi_k)$ if and only if $\pi_k(y)Y=0$.}}

\textup{(b) Since each profile $\sigma_k$ has compact support in $\theta_k$,  the integral in \eqref{Vk} is finite. 
The functions $V_k(y,\theta_k)$ vanish in $\theta_k<-M$ for $M>0$ large enough, but have nonzero finite limits as $\theta_k\to +\infty$.}

\end{rem}

\subsubsection{Choice of $\Omega^0$.}\label{O0}


The remaining part of the construction of $u^\eps_a$ and the error analysis will take place within a small enough, and in particular  \emph{bounded}, connected open neighborhood of $0$,  $\Omega^0\subset\{y:y_n\geq 0\}$, that is independent of $\eps$.  Here we specify the choice of $\Omega^0$.   First we require that the structural assumptions (Assumptions \ref{sh}, \ref{nonch}, \ref{ul}) 
all hold on $\Omega^0$ We also require that
the solutions $\psi_k$  to the eikonal equation exist on $\Omega^0$ and satisfy   
\eqref{a4}-\eqref{a7} there.  
Moreover, we require 
\begin{align}\label{28kw}
\mathbb{R}^N=\oplus^N_{k=1}\mathrm{span}\;r_k(y)=\oplus^N_{k=1}\mathrm{span}\;r_k(y',0)\text{ for }y\in \Omega^0.
\end{align}


The next lemma collects some obvious relations satisfied by the $\psi_j$, $\phi_j$ and needed below.

\begin{lem}[Relations between phases]\label{relations}
Let $\omega_k(y'):=\partial_n\psi_k(y',0)$. For $y\in\Omega^0$ we have for all $k$:
\begin{align}\label{28k}
\begin{split}
&(a)\psi_k(y)=(y_0-\phi_k(y'',y_n))\beta_k(y)\text{ where }\phi_k(y'',0)=0\text{ and }\beta_k(y',0)=1\\
&(b)\psi_k(y)=y_0-\phi_k(y'',y_n)+O(y_0y_n)+O(y_n^2)\\
&(c)\omega_k(y')y_n=-\phi_k(y'',y_n)+O(y_0y_n)+O(y_n^2)=-\partial_n\phi(y'',0)y_n+O(y_0y_n)+O(y_n^2).\\
&(d) \phi_k(y'',y_n)=\partial_n\phi_k(y'',0)y_n+O(y^2_n).
\end{split}
\end{align}
\end{lem}

\begin{proof}
For (a) recall \eqref{a5} and the discussion just after it.  For (b) write $\beta_k(y)=1+O(y_n)$ and use part (a).  For (c) write
\begin{align}
\psi_k(y)=\psi_k(y',0)+\partial_n\psi_k(y',0)y_n+O(y_n^2)=y_0+\omega_k(y')y_n+O(y_n^2)
\end{align}
and use part (b). Part (d) is immediate.

\end{proof}

Let $\gamma_k$, $k=1,\dots,N$ denote the distinct real eigenvalues of $B_0(0)=\mathcal{A}(0,1,0)$ ($1=\eta_0$ here).    They are nonzero because $y_n=0$ is noncharacteristic.   We have 
\begin{align}
\partial_n\phi_k(0)=\lambda_k(0,1,0)=\gamma_k,
\end{align}
so if $\Omega^0$ is small enough, the characteristic surfaces 
\begin{align}\label{28kz}
\Sigma_j=\{y:y_0=\phi_j(y'',y_n)\}
\end{align}
 will be disjoint on  $y_n>0$.  To see this, first observe that we  can choose $\Omega^0$ so that there exist positive constants $\delta$, $\kappa$ such that on $\Omega^0$:
\begin{align}\label{phases}
\begin{split}
&(a) \partial_{y''}\phi_k<\kappa\text{ and }|\partial_n\phi_k-\gamma_k|<\frac{\delta}{3}, \text{ where }\delta<\frac{1}{4}\inf_{k\neq j}|\gamma_j-\gamma_k|\\
&(b) |\lambda_k(y,1,\eta'')-\gamma_k|<\frac{\delta}{3}\text{ for }|\eta''|< \kappa \eta_0=\kappa\cdot  1.
\end{split}
\end{align}
In view of Lemma \ref{relations} we can choose $\Omega^0$ so that additionally:\footnote{The condition \eqref{phases2}(d) is a slight variant of \eqref{phases}(a).   It  does not appear in \cite{metajm}  but is needed here for the error analysis.}
\begin{align}\label{phases2}
\begin{split}
(c) |\phi_k(y'',y_n)-\gamma_ky_n|<\frac{\delta}{2} y_n\\
(d) |-\omega_k(y')y_n-\gamma_k y_n|<\frac{\delta}{2} y_n.
\end{split}
\end{align}
The disjointness of the  surfaces $\Sigma_j$ in $\Omega^0$ is obvious from \eqref{phases2}(c)  and the choice of $\delta$.
We now  complete the specification  of $\Omega^0$  by shrinking  it if necessary so that both 
$U_0(y,\theta)$ and $V(y,\theta)$ lie in  $C^\infty_b(\Omega^0\times\mathbb{R}^N_\theta).$

We conclude this section by describing a partition of unity subordinate to a cover of $y_n\geq 0$ by wedges that we need for the error analysis.\footnote{A similar partition is used in \cite{metajm}.}   For $\delta$ as above let $I_j=(\gamma_j-\delta,\gamma_j+\delta)$.   There exist functions 
$\tilde\chi_j\in C^\infty_c(\mathbb{R})$, $j=1,\dots,N$ and $\tilde\chi_0\in C^\infty(\mathbb{R})$ such that
\begin{align}\label{partition}
\begin{split}
&\sum^N_{j=0}\tilde \chi_j(t)=1\text{ and }\tilde\chi_j=1\text{ on }\overline{I_j}\\
&\mathrm{supp}\;\tilde\chi_j\cap\overline{I_k}=\emptyset\text{ if }j\neq k,
\end{split}
\end{align}
Set $\chi_j(y):=\tilde\chi_j(\frac{y_0}{y_n})$ and define wedges
\begin{align}\label{wedge}
\begin{split}
&W_j':=\{y:\frac{y_0}{y_n}\in I_j, y_n>0 \}, \;j=1,\dots, N\\
&W_j:=\{y:\frac{y_0}{y_n}\in \mathrm{supp}\;\chi_j, y_n>0\}, \;j=0,1,\dots, N.
\end{split}
\end{align}
So in $y_n>0$ we have $\Sigma_j \subset W_j'\subset \subset W_j$ for $j=1,\dots,N$.

The next definition is formulated in \cite{metajm}.   
\begin{defn}\label{regul}
(a)  Let $\mathcal{M}_0$ denote the space of vector fields on $\Omega^0$ with $C^\infty_b(\Omega^0,\mathbb{R})$ coefficients that are tangent to $\Delta=\{y:y_0=y_n=0\}$.  

(b) Denote by $\Lambda(\Omega^0)$ the space of functions $a\in C^\infty(\Omega^0\cap \{y_n>0\})$ such that $a\in L^\infty(\Omega^0)$ and $M_1M_2\cdots M_p a\in L^\infty(\Omega^0)$ for all finite sequences $M_1,M_2,...$ of vector fields in $\mathcal{M}_0$.

\end{defn}

Observe that the functions $\chi_j$ defined above satisfy
\begin{align}\label{reg2}
\chi_j\in \Lambda(\Omega^0).
\end{align}



\subsubsection{The interaction term $W^\eps(y,\theta_0,\xi_n)$.}\label{multicorrector}

We now construct $W(y,\theta_0,\xi_n)$ and $W^\eps(y,\theta_0,\xi_n)=\chi^\eps(y_0,y_n)W$, where the cutoff $\chi^\eps$ is chosen below.   We use these profiles to define the  piece of the corrector denoted by $\mathcal{W}^\eps$ in \eqref{i8}  as follows:
\begin{align}\label{d00}
\mathcal{W}^\eps(y):=W^\eps(y,\theta_0,\xi_n)|_{\theta_0=\frac{y_0}{\eps},\xi_n=\frac{y_n}{\eps}}.
\end{align}

Recall from \eqref{c4}(b) that the profile $W$ is constructed to satisfy a problem of the form
\begin{align}
\mathcal{L}_2(y',0,d\psi_0,\partial_{\theta_0,\xi_n})W=-\mathcal{H}_2.
\end{align}
In order to define $\mathcal{H}_2$ we first expand
\begin{align}\label{d0}
\begin{split}
&\psi_k(y)=\psi_k(y',0)+\partial_n\psi_k(y',0)y_n+r_{\psi_k}(y)=y_0+\omega_k(y')y_n+r_{\psi_k}(y),
\end{split}
\end{align}
where $\omega_k(y'):=\partial_n\psi_k(y',0)$ and $r_{\psi_k}(y)=O(y_n^2)$ near $y=0$.    Using $H_2(y,\theta)=H_2(y,\theta_1,\dots,\theta_N)$ we set 
\begin{align}\label{d1}
\mathcal{H}_2(y,\theta_0,\xi_n):=H_2(y,\theta_0+\omega_1(y')\xi_n,\dots,\theta_0+\omega_N(y')\xi_n).
\end{align}

\begin{rem}
\textup{Although  $$H_2(y,\theta)|_{\theta=\frac{\psi}{\eps}}\neq \mathcal{H}_2(y,\theta_0,\xi_n)|_{\theta_0=\frac{\psi_0}{\eps},\xi_n=\frac{y_n}{\eps}}$$
for all $y$, \eqref{d0} implies that equality holds at $y_n=0$.}
\end{rem}

To see why $\mathcal{H}_2$ might be a useful substitute for $H_2$,  we first write out $H_2$.  For $\mathcal{F}$ as in \eqref{c3y}
we have $H_2(y,\theta)=(1-S)\mathcal{F}$ from \eqref{c3z}.   Thus, the parts of $\mathcal{F}$ that are linear in $(\sigma,\sigma_y)$ or quadratic in a single $\sigma_k$ cancel out in $(1-S)\mathcal{F}$.  Explicitly, we compute (with $f^{q,k}(y,\sigma_k)$ as in \eqref{c16z} and using \eqref{c1}):
\begin{align}\label{d2}
\begin{split}
&H_2=(1-S)\mathcal{F}=L(y,\partial_y)U_0-f(y,U_0)-\left[L(y,\partial_y)U_0-\left(f^l(y,\sigma)+\sum_kf^{q,k}(y,\sigma_k)\right)\right]=\\
&\qquad -\left(f^q(y,\sigma)-\sum_kf^{q,k}(y,\sigma_k)\right)=-\sum_k\left(\sum_{m<p}f^q_{kmp}(y)\sigma_m\sigma_p\right)r_k.
\end{split}
\end{align}

\begin{rem}\label{d2a}

 \textup{We can write $H_2(y,\theta)=\sum_k H_{2,k}(y,\theta)r_k(y)$, where each $H_{2,k}(y,\theta)$ is a finite sum of terms of the form
\begin{align}\label{d2z}
\sigma_m(y,\theta_m) \sigma_p(y,\theta_p) c(y),\;m\neq p
\end{align}
with $c(y)$ a smooth real-valued function that varies from term to term.  Thus, $\mathcal{H}_2(y,\theta_0,\xi_n)=\sum_k \mathcal{H}_{2,k}(y,\theta_0,\xi_n)r_k(y)$, where each $\mathcal{H}_{2,k}$ is a finite sum of terms of the form
\begin{align}\label{d2y}
\sigma_m(y,\theta_0+\omega_m(y')\xi_n) \sigma_p(y,\theta_0+\omega_p(y')\xi_n) c(y), \;m\neq p
\end{align}
with $c(y)$ a smooth real-valued function that varies from term to term.}

\end{rem}

Since $g(y',\theta_0)$ has support in $|\theta_0|\leq 1$, the equations \eqref{c12}  imply that the $\sigma_k(y,\theta_k)$ have support in $|\theta_k|\leq 1$.  We show in Proposition \ref{suppprop} that this implies products like 
$$\sigma_m(y,\frac{\psi_m}{\eps}) \sigma_p(y,\frac{\psi_p}{\eps}), \; m\neq p$$
are supported in a small neighborhood of $\Delta$.

\begin{prop}\label{suppprop}
Let $\Omega^0$ be as chosen in section \ref{O0}.  There exists $M>0$ such that  the following statements hold on $\Omega^0$:

(a) In $y_n\geq M\eps$ we have
\begin{align}\label{d3}
\mathrm{supp}\;\sigma_m\left(y,\frac{\psi_m}{\eps}\right)\cap\mathrm{supp}\; \sigma_p\left(y,\frac{\psi_p}{\eps}\right)=\emptyset, \text{ when }m\neq p.
\end{align}
(b) All products $$\sigma_m\left(y,\frac{\psi_m}{\eps}\right) \sigma_p\left(y,\frac{\psi_p}{\eps}\right), \; m\neq p,$$ and thus $H_2\left(y,\frac{\psi}{\eps}\right),$
are supported in the region $I^\eps:=\{y:|y_n|\leq M\eps, |y_0|\leq M\eps\}$.
The same applies to the products $$\sigma_m\left(y,\frac{y_0+\omega_m(y')y_n}{\eps}\right)\sigma_p\left(y,\frac{y_0+\omega_p(y')y_n}{\eps}\right) , \; m\neq p$$ and to  $\mathcal{H}_2(y,\frac{y_0}{\eps},\frac{y_n}{\eps})$.

(c) In $y_n\leq M\eps$ we have for each $m$ 
\begin{align}\label{d3b}
\left|\sigma_m\left(y,\frac{\psi_m}{\eps}\right)-\sigma_m\left(y,\frac{y_0+\omega_m(y')y_n}{\eps}\right)\right|\lesssim \eps.
\end{align}

(d)  In $y_n\leq M\eps$ we have 
\begin{align}\label{d3c}
\left|H_2\left(y,\frac{\psi}{\eps}\right)-\mathcal{H}_2\left(y,\frac{y_0}{\eps},\frac{y_n}{\eps}\right)\right|\lesssim \eps.
\end{align}

\end{prop}

\begin{proof}
\textbf{a. } Recall that $\psi_k(y)=(y_0-\phi_k(y'',y_n))\beta_k(y)$, where $\beta_k$ is smooth and positive near $y=0$.  
Since $|\psi_k|\lesssim \eps$ on $\mathrm{supp}\;\sigma_k\left(y,\frac{\psi_k}{\eps}\right)$ for each $k$, we have 
\begin{align}\label{d4}
|y_0-\phi_k(y'',y_n)|\lesssim \eps \text{ on }\mathrm{supp}\;\sigma_k\left(y,\frac{\psi_k}{\eps}\right)\text{ for each }k. 
\end{align}
By \eqref{phases2}(c) we  can write 
\begin{align}\label{d5}
\begin{split}
&\left(\gamma_k-\frac{\delta}{2}\right)y_n < \phi_k <\left(\gamma_k+\frac{\delta}{2}\right)y_n \text{ for all }k \Rightarrow |\phi_m-\phi_p|\geq 5\delta y_n \text{ for }m\neq p,
\end{split}
\end{align}
and this implies
\begin{align}\label{d6}
|\phi_m(y)-\phi_p(y)|\geq 5\delta M\eps\text{ for }y_n\geq M\eps.
\end{align}
If $y\in \mathrm{supp}\;\sigma_m\left(y,\frac{\psi_m}{\eps}\right)\cap\mathrm{supp}\; \sigma_p\left(y,\frac{\psi_p}{\eps}\right)$ and $y_n \geq M\eps$ for a large enough choice of $M$, then  \eqref{d6} contradicts the fact that $y$ satisfies \eqref{d4} for $m\neq p$.

\textbf{b. }From \eqref{d4} and \eqref{d5} we see that that for all $k$ we have: $|y_0|\lesssim\eps$ on $\mathrm{supp}\; \sigma_k\left(y,\frac{\psi_k}{\eps}\right)$ when $y_n\lesssim \eps$.     The  claim about the first group of products now follows from part (a).\footnote{In this proof the constant $M$ may increase from part to part.}  The claim about the second group of products follows by a similar argument.

\textbf{c. }Using \eqref{d0} we have
\begin{align}\label{d7}
\begin{split}
&\left|\sigma_m\left(y,\frac{\psi_m}{\eps}\right)-\sigma_m\left(y,\frac{y_0+\omega_m(y')y_n}{\eps}\right)\right|=\\
&\qquad \qquad \partial_{\theta_m}\sigma_m\left(y,\frac{y_0+\omega_m(y')y_n}{\eps}\right)\frac{r_{\psi_m}(y)}{\eps}+O\left(\frac{r_{\psi_m(y)}}{\eps}\right)^2.
\end{split}
\end{align}
This implies \eqref{d3b} since $|r_{\psi_m}(y)|=O(|y_n|^2)\leq \eps^{2}$ in $y_n\lesssim \eps$.

\textbf{d. }Part (d) follows immediately from Remark  \ref{d2a} and \eqref{d3b}.

\end{proof}

\begin{defn}\label{d8}
In view of Proposition \ref{suppprop} (b) we refer to $I^\eps:=\{y:|y_n|\leq M\eps, |y_0|\leq M\eps\}$ as the \emph{interaction region}.
\end{defn}

Proposition \ref{suppprop}(d) indicates that $\mathcal{H}_2$ may be an acceptable substitute for $H_2$ in the construction of $W$. 
We can obtain a solution of 
\begin{align}\label{d8a}
\mathcal{L}_2(y',0,d\psi_0,\partial_{\theta_0,\xi_n})W=-\mathcal{H}_2, 
\end{align}
for $\mathcal{H}_2$ as in \eqref{d1} as follows.   First write\footnote{Recall \eqref{28kw} and Remark \ref{explain}.}
\begin{align}\label{d9}
\mathcal{H}_2(y,\theta_0,\xi_n)=\sum_k\mathcal{H}^k_{2}(y,\theta_0,\xi_n)r_k(y',0).
\end{align}
Define $W$, \emph{formally} at first,  by 
\begin{align}\label{w}
\begin{split}
&W(y,\theta_0,\xi_n)=\sum_k t_k(y,\theta_0,\xi_n)r_k(y',0), \text{ where }\\
&t_k(y,\theta_0,\xi_n)=-\int^{\xi_n}_{+\infty}\mathcal{H}^k_2\left(y,\theta_0+\omega_k(y')(\xi_n-s),s\right)ds.
\end{split}
\end{align}
Since\footnote{Here we use the fact that for all $k$ we have $d'\psi_k(y',0)=d'\psi_0(y')=(1,0)$.    Also recall \eqref{a2} and Definition \ref{pi}.} 
\begin{align}\label{d10}
\begin{split}
&\mathcal{L}_2(y',0,d\psi_0,\partial_{\theta_0,\xi_n})=\partial_{\xi_n}+\mathcal{A}(y',0,d'\psi_0)\partial_{\theta_0}\text{ and }\\&\mathcal{A}(y',0,d'\psi_0)r_k(y',0)=-\omega_k(y')r_k(y',0),
\end{split}
\end{align}
we see that for any $k$ 
\begin{align}
\mathcal{L}_2(y',0,d\psi_0,\partial_{\theta_0,\xi_n})t_k(y,\theta_0,\xi_n)r_k(y',0)=[\partial_{\xi_n}-\omega_k(y')\partial_{\theta_0}]t_k r_k(y',0).
\end{align}
Thus, $W$ as in \eqref{w} is by inspection a formal solution of \eqref{d8a}.    

It remains to examine the integral in \eqref{w}.   Although $\mathcal{H}^k_2\neq \mathcal{H}_{2,k}$ for $\mathcal{H}_{2,k}$ as in Remark \ref{d2a}, it is still true that $\mathcal{H}^k_2$ is a finite sum of terms of the form \eqref{d2y}; the only difference is that the smooth coefficients $c(y)$ change.    For a particular $k$ the contribution of a term \eqref{d2y} to the integral in \eqref{w} is 
\begin{align}\label{d10z}
c(y)\int^{\xi_n}_{+\infty}\sigma_m\left(y,\theta_0+\omega_k(y')\xi_n+s(\omega_m(y')-\omega_k(y'))\right)\sigma_p\left(y,\theta_0+\omega_k(y')\xi_n+s(\omega_p(y')-\omega_k(y'))\right) ds:=A.
\end{align}
At most one of $m,p$ can equal $k$ and the $\sigma_l$ are bounded with compact support in $\theta_l$.  Suppose $m\neq k$. Then 
\begin{align}
|A|\lesssim \int^\infty_{-\infty}|\sigma_m(y,\theta_0+\omega_k\xi_n+s(\omega_m-\omega_k))|ds=\int^\infty_{-\infty}|\sigma_m(y,t)|\frac{dt}{|\omega_m-\omega_k|}\leq K,
\end{align}
where $K$ is independent of $(y,\theta_0,\xi_n)$.   Here we used that $|\omega_m(y')-\omega_k(y')|\geq 3\delta>0$ for $y'$ near $0$.\footnote{See \eqref{phases} and \eqref{phases2}.}So $W$ is indeed a bounded solution of \eqref{d8a}.    

\begin{rem}\label{d10yy}
\textup{Similar estimates show that $t_k(y,\theta_0,\xi_n)$ is $C^\infty$ on $\Omega^0\times \mathbb{R}^2_{\theta_0,\xi_n}$.
Although $t_k$ is bounded on this set, that is not true of its derivatives.   For example, applying  $\partial_{y'}$ to the integral in \eqref{d10z} pulls out a factor of $(\partial_{y'}\omega_k) \xi_n$.   The derivatives are bounded on bounded subsets of $\Omega^0\times \mathbb{R}^2_{\theta_0,\xi_n}$ though.  That will be useful because factors like $\frac{y_n}{\eps}$ are bounded uniformly with respect to $\eps\in (0,\eps_0]$ on the interaction region $I_\eps$.}
\end{rem}

Finally, we modify $W$ to obtain $W^\eps$ as follows.
\begin{align}\label{29aa}
W^\eps(y,\theta_0,\xi_n):=\chi\left(\frac{y_0}{\sqrt{\eps}}\right)\chi\left(\frac{y_n}{\sqrt{\eps}}\right)W(y,\theta_0,\xi_n):=\chi^\eps(y_0,y_n)W(y,\theta_0,\xi_n),
\end{align}
 where $\chi(s)\in C^\infty_c(\mathbb{R})$ is $1$ for $|s|\leq 1$, $0$ for $|s|\geq 2$.    This enforces $|y_0|\lesssim \sqrt{\eps}$, $|y_n|\lesssim \sqrt{\eps}$ on $\mathrm{supp}\;W^\eps$.

 Now 
\begin{align}\label{29b}
\mathcal{L}_2(y',0,d\psi_0,\partial_{\theta_0,\xi_n})W^\eps=-\chi\left(\frac{y_0}{\sqrt{\eps}}\right)\chi\left(\frac{y_n}{\sqrt{\eps}}\right)\mathcal{H}_2(y,\theta_0,\xi_n), 
\end{align}
but note that because of Proposition \ref{suppprop}(b):
\begin{align}\label{29bb}
\left[-\chi\left(\frac{y_0}{\sqrt{\eps}}\right)\chi\left(\frac{y_n}{\sqrt{\eps}}\right)\mathcal{H}_2(y,\theta_0,\xi_n)\right]|_{\theta_0=\frac{y_0}{\eps}, \xi_n=\frac{y_n}{\eps}}=-\mathcal{H}_2(y,\theta_0,\xi_n)|_{\theta_0=\frac{y_0}{\eps}, \xi_n=\frac{y_n}{\eps}}
\end{align} 
for $\eps_0$ small.   Thus, we have 
\begin{align}\label{29c}
[\mathcal{L}_2(y',0,d\psi_0,\partial_{\theta_0,\xi_n})W^\eps]|_{\theta_0=\frac{y_0}{\eps}, \xi_n=\frac{y_n}{\eps}}
=[\mathcal{L}_2(y',0,d\psi_0,\partial_{\theta_0,\xi_n})W]|_{\theta_0=\frac{y_0}{\eps}, \xi_n=\frac{y_n}{\eps}}.
\end{align}

\begin{rem}
\textup{The cutoff $\chi^\eps(y_0,y_n)$ is introduced in \eqref{29aa} for later use in the error analysis.  It is needed for the estimates of $W^\eps$ and the  part of $\mathcal{E}_0^\eps$ given by $(\mathcal{L}_2-\mathcal{L}_{2,0})W^\eps$; see \eqref{30h}(a) and the footnote there.   The cutoff has no effect on the estimate of the  part of $\mathcal{E}_0^\eps$ given by $H_2-\mathcal{H}_2$, because of \eqref{29bb}.  We will show $\chi^\eps$ is harmless in the estimate of $\eps\mathcal{E}_1^\eps$, even though we must take into account the effect of applying vector fields to it.}

\end{rem}

This completes the construction of the approximate solution $u^\eps_a$ as in \eqref{ua}.

\begin{rem}\label{d10y}
\textup{Observe that if any term in the expansion of $\mathcal{H}^k_2$ had the form $\sigma^2_k\left(y,\theta_0+\omega_k(y')\xi_n\right)c(y)$, then the contribution of that term to the integral in \eqref{w} would not be finite (set $m=p=k$ in \eqref{d10z}). 
But such terms are absent because of the cancellations in $H_2=(1-S)\mathcal{F}$.}     

\textup{Scalar factors of the form $\sigma^2_k\left(y,\theta_0+\omega_k(y')\right)c(y)$ do appear in terms of $E\mathcal{F}$ and  in $H_1$ \eqref{c16z}, but those contributions to $\mathcal{F}$ are 
solved away, respectively,  by the leading profile equations \eqref{c12} and by the choice of $V$ \eqref{c13}.   This provides a belated motivation for the decomposition of $\mathcal{F}$ given in \eqref{c3}, \eqref{c3z}; see also Remark \ref{decV}.}

\end{rem}

\subsection{Summary} 
We recall from \eqref{a9z} that $u^\eps_a(y)=U_0(y,\theta)|_{\theta=\frac{\psi}{\eps}}+\eps U^\eps_1(y)$ satisfies
\begin{align}\label{d11}
L(y,\partial_y)u^\eps_a-f(y,u^\eps_a)=\eps^{-1}\mathcal{E}^\eps_{-1}+\mathcal{E}_0^\eps+\eps\mathcal{E}_1^\eps.
\end{align}
We saw that $\mathcal{E}^\eps_{-1}=0$ by virtue of the polarization of the terms of $U_0$.  With $\mathcal{F}=L(y,\partial_y)U_0-f(y,U_0)$ we have 
\begin{align}\label{d12}
\mathcal{E}_0^\eps(y)=\left[\mathcal{F}(y,\theta)+\mathcal{L}_1(y,d\psi,\partial_\theta)V(y,\theta)+\mathcal{L}_2(y,d\psi_0,\partial_{\theta_0,\xi_n})W^\eps(y,\theta_0,\xi_n)\right]|_{\theta=\frac{\psi}{\eps},\theta_0=\frac{\psi_0}{\eps},\xi_n=\frac{y_n}{\eps}}.
\end{align}
The profile $U_0$ was constructed to make $E\mathcal{F}=0$.   We wrote $(1-E)\mathcal{F}=H_1+H_2$ and constructed $V_1$ such that $\mathcal{L}_1V_1+H_1=0$.  Denoting  $\mathcal{L}_2(y',0,d\psi_0,\partial_{\theta_0,\xi_n})$ by $\mathcal{L}_{2,0}$ we can therefore write 
\begin{align}\label{d13}
\mathcal{E}_0^\eps(y)=\left[H_2+\mathcal{L}_2W^\eps\right]|_{\theta=\frac{\psi}{\eps},\theta_0=\frac{\psi_0}{\eps},\xi_n=\frac{y_n}{\eps}}=\left[(H_2-\mathcal{H}_2)+(\mathcal{L}_2-\mathcal{L}_{2,0})W^\eps\right]|_{\theta=\frac{\psi}{\eps},\theta_0=\frac{\psi_0}{\eps},\xi_n=\frac{y_n}{\eps}}.
\end{align}
Hence on the open set  $\Omega^0\ni 0$, chosen  as in section \ref{O0}, the approximate solution $u_a$ satisfies
\begin{align}\label{d14}
\begin{split}
&L(y,\partial_y)u^\eps_a=f(y,u^\eps_a)+r^\eps_a\text{ in }y_n>0\\
&B(y')u^\eps_a|_{y_n=0}=b^\eps(y')+\eps B(y')U^\eps_1(y',0),\\
&u^\eps_a=0\text{ in }y_0<0,
\end{split}
\end{align}
where 
\begin{align}\label{d15}
\begin{split}
&r^\eps_a(y)=\mathcal{E}_0^\eps(y)+\eps\mathcal{E}_1^\eps(y)\text{ with }\\
&\mathcal{E}_0^\eps(y)=\left[(H_2-\mathcal{H}_2)+(\mathcal{L}_2-\mathcal{L}_{2,0})W^\eps\right]|_{\theta=\frac{\psi}{\eps},\theta_0=\frac{\psi_0}{\eps},\xi_n=\frac{y_n}{\eps}}\\
&\mathcal{E}_1^\eps(y)=\left[ -K(y,U_0,\eps \mathcal{U}^\eps_1)\mathcal{U}^\eps_1+L(y,\partial_y) V+L(y,\partial_y)W^\eps\right]|_{\theta=\frac{\psi}{\eps}, \theta_0=\frac{\psi_0}{\eps},\xi_n=\frac{y_n}{\eps}}.
\end{split}
\end{align}

\begin{rem}\label{d15z}
\textup{At this point it is clear that \emph{for each $\eps$}, $u^\eps_a\in C^\infty_b(\Omega^0)$, and that $u^\eps_a\in L^\infty(\Omega^0)$ uniformly with respect to small $\epsilon$.  Indeed, the same is true for the individual pieces
$$U_0\left(y,\frac{\psi}{\eps}\right), V\left(y,\frac{\psi}{\eps}\right), \text{ and }W^\eps\left(y,\frac{y_0}{\eps},\frac{y_n}{\eps}\right).$$
We show in section \ref{Ua} that  for any $r\in \mathbb{N}_0$, $u^\eps_a\in N^m(\Omega^0)$ uniformly with respect to small $\eps$.}
\end{rem}


\section{Exact solution and error analysis}\label{ea}

To complete the proof of Theorem \ref{mr} it remains to estimate the exact solution $u^\eps$ to \eqref{ia1} and the difference $u^\eps-u^\eps_a$ in $L^\infty\cap N^m(U)$, where $m>\frac{n+5}{2}$ and $U$ is some neighborhood of $0$.


In section \ref{bpexact} we apply a result from \cite{metajm} to obtain $u^\eps\in L^\infty\cap N^m(\Omega_{T_1})$  on a certain domain of determinacy $\Omega_{T_1}\subset \Omega^0$ for that problem.\footnote{The sets $\Omega_T$   are defined in \eqref{a1z}. }
We define the error problem to be the problem satisfied by the difference $w^\eps:=u^\eps-u^\eps_a$.   We have 
\begin{align}\label{e1}
f(y,u^\eps)-f(y,u^\eps_a)=\left(\int^1_0\partial_uf(y,u^\eps_a+s(u^\eps-u^\eps_a))ds\right)w^\eps:=D(y,u^\eps,u^\eps_a)w^\eps,
\end{align}
so the error problem, a linear problem for $w^\eps$,  is
\begin{align}\label{error}
\begin{split}
&L(y,\partial_y)w^\eps=D(y,u^\eps,u^\eps_a)w^\eps-r^\eps_a\text{ in }y_n>0\\
&B(y')w^\eps|_{y_n=0}=-\eps B(y')U^\eps_1(y',0)\\
&w^\eps=0\text{ in }y_0<0.
\end{split}
\end{align}
We will solve this problem and estimate $w^\eps$ on $U:=\Omega_T\subset\Omega_{T_1}\subset \Omega^0$ for some small enough $T>0$ independent of $\eps$.  We have
\begin{align}\label{e2}
u^\eps=u^\eps_a+w^\eps=U_0(y,\theta)|_{\theta=\frac{\psi}{\eps}}+\eps U^\eps_1(y)+w^\eps,
\end{align}
where $U^\eps_1$ is bounded in $L^\infty(\Omega^0)$ uniformly for $\eps$ small by the above construction.  Thus, if we can show that $w^\eps$ is small in $L^\infty(\Omega_T)$ uniformly for $\eps$ small,  we will have shown that the exact solution is close to $U_0(y,\theta)|_{\theta=\frac{\psi}{\eps}}$ in a useful sense.   To control the $L^\infty$ norm of $w^\eps$, we need to control
the $L^\infty\cap N^m$ norms of the interior and boundary forcing terms in \eqref{error}; see Proposition \ref{37b}.
Thus, we need to estimate $u^\eps_a$ (section \ref{Ua}) and $r^\eps_a$ (section \ref{ra}) in the conormal spaces $N^m(\Omega^0)$.  
The error $w^\eps$ is estimated in section \ref{weps}.

\subsection{Spaces of conormal distributions}\label{Nm}

Here we give a precise description, based on \cite{metajm}, of the conormal spaces $N^m(\Omega^0)$. 
In this section the characteristic surfaces $\Sigma_j$  are as in \eqref{28kz} and the codimension two surface $\Delta$ is $\{y:y_0=y_n=0\}$.  

\begin{defn}\label{43a}
(a) For each $j=1,\dots,N$ let $\mathcal{M}_j$ denote the space of vector fields on $\Omega^0$ with coefficients in $C^\infty_b(\Omega^0)$ that are tangent to $\Sigma_j$ and $\Delta$.     A set of generators of $\mathcal{M}_j$ over $C^\infty_b(\Omega^0)$ is given by:\footnote{The spaces $\mathcal{M}_j$, $\mathcal{M}_0$, $\mathcal{M}'$ are $C^\infty_b(\Omega^0)$ (or $C^\infty_b(b\Omega^0)$)  modules which are closed under multiplication on the left by elements of $C^\infty_b$ and closed under the Lie bracket $[V,W]=VW-WV$.}

$M_0=(y_0-\phi_j)\partial_0$

$M_\nu=\partial_\nu+(\partial_\nu\phi_j)\partial_0, \nu=1,\dots,n-1$

$M_n=y_n(\partial_n+(\partial_n\phi_j)\partial_0)$

$M_{n+1}=(y_0-\phi_j)\partial_n$.\\

(b) Let $\mathcal{M}_0$ denote the space of vector fields on $\Omega^0$ with coefficients in $C^\infty_b(\Omega^0)$ that are tangent to $\Delta$.     A set of generators is given by 
$$y_0\partial_{0},\; y_0\partial_n,\; y_n\partial_0, \;y_n\partial_n, \;\partial_l, l=1\dots,n-1.$$

(c) Let $\mathcal{M}'$  denote the space of vector fields on $b\Omega^0$ with coefficients in $C^\infty_b(b\Omega^0)$ that are tangent to $\Delta$.  A set of   generators of $\mathcal{M}'$ is given by
$$m_0=y_0\partial_0, m_\nu=\partial_\nu,\; \nu=1,\dots,n-1.$$

\end{defn}

\begin{defn}[Conormal spaces]\label{43b}
Let $U\subset \Omega^0$ be some neighborhood of $0$.  

(a) For $m\in\mathbb{N}_0$ let $N^m(U,\mathcal{M}_j)$ denote the set of $u\in L^2(U)$ such that 
$V_1V_2\cdots V_k u\in L^2(U)$ for any choice of $V_i\in \mathcal{M}_j$ and $k\leq m$. 
To define a norm on $N^m(U,\mathcal{M}_j)$ let $M_0,\dots,M_{n+1}$ denote the generators of $\mathcal{M}_j$ given in Definition \ref{43a}.   Then set
\begin{align}\label{43c}
|u|^2_{N^m(U,\mathcal{M}_j)}=\sum_{|\alpha|\leq m}|(M_0,\dots,M_{n+1})^\alpha u|^2_{L^2(U)}.
\end{align}

(b) The spaces $N^m(U,\mathcal{M}_0)$ and $N^m(bU,\mathcal{M}')$ and their norms $|u|_{N^m(U,\mathcal{M}_0)}$ and $\langle u\rangle_{N^m(bU,\mathcal{M}')}$ are defined analogously using the generators of $\mathcal{M}_0$ and $\mathcal{M}'$ given in Definition \ref{43a}.  Usually, we will write simply $\langle u\rangle_{N^m(bU)}$ in place of $\langle u\rangle_{N^m(bU,\mathcal{M}')}$.


(c) Let $N^m(U)$ be the set of $u\in L^2(U)$ such that $u=\sum^N_{j=1}u_j$ for some choice of  $u_j\in N^m(U,\mathcal{M}_j)$. 
Define
\begin{align}\label{43d}
|u|_{N^m(U)}=\inf \sum^N_{j=1}|u_j|_{N^m(U,\mathcal{M}_j)},
\end{align}
where the inf is taken over all decompositions $u=\sum^N_{j=1}u_j$ as above.

(d) For $m\in\mathbb{N}_0$ and $j\in \{0,1,\dots,N\}$ let $N^m_\infty(U,\mathcal{M}_j)$ denote the set of $u\in L^\infty(U)$ such that 
$V_1V_2\cdots V_k u\in L^\infty(U)$ for any choice of $V_i\in \mathcal{M}_j$ and $k\leq m$. 
With generators $M_0,\dots,M_{n+1}$ as in (a), define
\begin{align}\label{43h}
|u|^2_{N^m_\infty(U,\mathcal{M}_j)}=\sum_{|\alpha|\leq m}|(M_0,\dots,M_{n+1})^\alpha u|^2_{L^\infty(U)}.
\end{align}
The space $N^m_\infty(bU,\mathcal{M}')$ and its norm $\langle u\rangle_{N^m_\infty(bU,\mathcal{M}')}$ (or simply $\langle u\rangle_{N^m_\infty(bU)}$) are defined similarly.   

(e) If $u\in L^\infty\cap N^m(U,\mathcal{M}_j)$, we define $$|u|_{L^\infty\cap N^m(U,\mathcal{M}_j)}=|u|_{L^\infty(U)}+|u|_{N^m(U,\mathcal{M}_j)},$$
and do similarly for the other $N^m$ norms.

\end{defn}

The set $\Omega^0$ is bounded, so for $U\subset \Omega^0$ we have the obvious estimate
\begin{align}\label{44a}
|u|_{N^m(U,\mathcal{M}_j)}\leq \sqrt{|\Omega^0|}\;|u|_{N^m_\infty(U,\mathcal{M}_j)}, 
\end{align}
where $|\Omega^0|$ is the Lebesgue measure of $\Omega^0$.   Later we often use \eqref{44a} to estimate $N^m(U,\mathcal{M}_j)$ norms.

We will need to use the following lemma, which is stated in \cite{metajm} and proved in \cite{metduke}.  The lemma uses the partition of unity 
$\sum^N_{j=0}\chi_j(y)$ on $\{y_n>0\}$,  where the $\chi_j$ are defined just below \eqref{partition}.

\begin{lem}\label{43e}
Let $u\in L^2(U)$.  Then $u\in N^m(U)$ if and only if $\chi_ju\in N^m(U,\mathcal{M}_j)$ for $j=0,\dots,N$.  Moreover, 
the norm $|u|_{N^m(U)}$ \eqref{43d} is equivalent to the norm $\sum^N_{j=0}|\chi_j u|_{N^m(U,\mathcal{M}_j)}$ with constants independent of $U$.
\end{lem}

\subsection{Exact solution $u^\eps$}\label{bpexact}

In this section we apply a result of \cite{metajm} to obtain an exact solution of \eqref{ia1} on an appropriate domain of determinacy  for that problem.  
First, for positive constants $T_0$, $\alpha$, and $T<T_0$  we define
\begin{align}\label{a1z}
\begin{split}
&\Omega=\{y: -T_0<y_0<T_0-\alpha |y'',y_n|\}\cap\{y_n\geq 0\}\text{ and  }b\Omega=\{y':(y',0)\in \Omega\cap\{y_n=0\}\}\\
&\Omega_T=\Omega\cap \{y_0<T\}\text{ and }b\Omega_T=\{y':(y',0)\in \Omega_T\cap\{y_n=0\}\}.
\end{split}
\end{align}
The constants $T_0$ and $\alpha$ are chosen so that $\Omega$, and thus $\Omega_T$,  is a domain of determinacy for 
the problem \eqref{ia1}.   We also require $\Omega\subset \Omega^0$ for $\Omega^0$ as chosen in section \ref{O0}. 

Using \eqref{44a} it is easy to check that for each $m\in \mathbb{N}_0$, the boundary datum in \eqref{ia1} satisfies
\begin{align}\label{ex1}
\langle b^\eps\rangle_{L^\infty\cap N^m(b\Omega)}\lesssim 1.
\end{align}
The next proposition is just a rephrasing of Theorem 2.1.1 of \cite{metajm} in a form suitable for the problem \eqref{ia1}. 
\begin{prop}[Exact solution]\label{ex2}
Suppose $m>\frac{n+5}{2}$.    Then \eqref{ex1} implies  that 
if $T_0$ and $\alpha$ in \eqref{a1z} are small enough (depending on $({L},B)$),  the problem \eqref{ia1} has an exact solution $u^\eps$ in $L^\infty\cap N^m(\Omega_{T_1})$ for some $0<T_1<T_0$ that satisfies
\begin{align}\label{ex3}
|u^\eps|_{L^\infty\cap N^m(\Omega_{T_1})}\lesssim 1.
\end{align}

\end{prop}


\subsection{Estimate of the approximate solution $u^\eps_a$.}\label{Ua}

At the moment we know $u^\eps_a$ has the properties described in Remark \ref{d15z}.   We will show:
\begin{prop}\label{32a}
For all $r\in \mathbb{N}_0:=\{0,1,2,3,\dots\}$ we have $|u^\eps_a(y)|_{N^r(\Omega^0)}\lesssim 1.$  Moreover, 
\begin{align}
\begin{split}
&\left|U_0\left(y,\frac{\psi}{\eps}\right)\right|_{N^r(\Omega^0)}\lesssim 1,\;\left|V\left(y,\frac{\psi}{\eps}\right)\right|_{N^r(\Omega^0)}\lesssim 1,\text{ and }
\left|W^\eps\left(y,\frac{y_0}{\eps},\frac{y_n}{\eps}\right)\right|_{N^r(\Omega^0)}\lesssim 1,\\
&\text{ and }\left\langle V\left(y',0,\frac{y_0}{\eps},\dots,\frac{y_0}{\eps}\right)\right\rangle_{N^r(b\Omega^0)}\lesssim 1,\text{ and }
\left\langle W^\eps\left(y',0,\frac{y_0}{\eps},0\right)\right\rangle|_{N^r(b\Omega^0)}\lesssim 1.
\end{split}
\end{align}
\end{prop}

The proof is contained in Propositions \ref{32b}, \ref{32h}, and \ref{33cc} below.\footnote{The proofs show that in each case the $N^r$ norm can be replaced by an $N^r_\infty$ norm.}
\begin{prop}\label{32b}
For all $r\in \mathbb{N}_0:=\{0,1,2,3,\dots\}$ we have
\begin{align}
\left|U_0\left(y,\frac{\psi}{\eps}\right)\right|_{N^r(\Omega^0)}\lesssim 1.
\end{align}
\end{prop}

\begin{proof}
\textbf{1. }We have $U_0\left(y,\frac{\psi}{\eps}\right)=\sum_k\sigma_k\left(y,\frac{\psi_k}{\eps}\right)r_k(y)$, so it suffices to show for each $k$ that 
\begin{align}\label{32c}
\left|\sigma_k\left(y,\frac{\psi_k}{\eps}\right)\right|_{N^r(\Omega^0,\mathcal{M}_k)}\lesssim 1.
\end{align}

\textbf{2. Preparation.} Recall that a set of  generators of $\mathcal{M}_k$ is given by $M_0,\dots,M_{n+1}$, where 

$M_0=(y_0-\phi_k)\partial_0$

$M_\nu=\partial_\nu+(\partial_\nu\phi_k)\partial_0, \;\nu=1,\dots,n-1$

$M_n=y_n(\partial_n+(\partial_n\phi_k)\partial_0)$

$M_{n+1}=(y_0-\phi_k)\partial_n$.\\

Also,  generators of $\mathcal{M}'$ are given by $m_0=y_0\partial_0$, $m_\nu=\partial_\nu$, $\nu=1,\dots,n-1$.\\

For a fixed $k$ let 
\begin{align}\label{32dz}
J_\eps:=\{y\in\Omega^0:|y_0-\phi_k|\lesssim \eps\},\;\;  bJ_\eps:=\{y'\in b\Omega^0:|y_0|\lesssim \eps\}
\end{align}
 and  set\footnote{Here  $C_\gamma$ is independent of $\eps$. Both $J_\eps$ and $\Gamma^\eps(\Omega^0)$ depend on $k$, but we suppress that in the notation.}
\begin{align}\label{32d}
\begin{split}
&\Gamma^\eps(\Omega^0)=\{G^\eps\in C^\infty(\Omega^0,\mathbb{R}):|(M_0,\dots,M_{n+1})^\gamma G^\eps|_{L^\infty(\Omega^0\cap J_\eps)}\lesssim C_\gamma\text{ for }\eps\in (0,1], \;\gamma\in\mathbb{N}_0^{n+2}\}\\
&b\Gamma^\eps=\{g^\eps\in C^\infty(b\Omega^0,\mathbb{R}):|(m_0,\dots,m_{n-1})^\gamma g^\eps|_{L^\infty(b\Omega^0\cap bJ_\eps)}\lesssim C_\gamma\text{ for }\eps\in (0,1], \;\gamma\in\mathbb{N}_0^{n}\}.
\end{split}
\end{align}
\begin{rem}\label{32e}
\textup{Observe that smooth $\eps-$independent functions $c(y)$ lie in $\Gamma^\eps(\Omega^0)$ and $\frac{y_0-\phi_k}{\eps}\in \Gamma^\eps(\Omega^0)$.   Moreover, $\Gamma^\eps(\Omega^0)$ is mapped to itself by elements of $\mathcal{M}_k$ and is closed under products.}
\end{rem}
For the next step we define the set of ``profile-type" functions
\begin{align}\label{32f}
\begin{split}
&\mathcal{P}_k:=\{P\in C^\infty_b(\Omega^0\times\mathbb{R}_{\theta_k},\mathbb{R}): P(y,\theta_k)\text{ has support in }|\theta_k|\leq 1\}\\
&b\mathcal{P}_k:=\{p\in C^\infty_b(b\Omega^0\times\mathbb{R}_{\theta_0},\mathbb{R}): p(y',\theta_0)\text{ has support in }|\theta_0|\leq 1\}.
\end{split}
\end{align}

\textbf{3. }Using $\psi_k=(y_0-\phi_k)\beta_k$ we compute, for example,
\begin{align}\label{32g}
M_\nu \sigma_k\left(y,\frac{\psi_k}{\eps}\right)=(M_\nu\sigma_k(y,\theta_k))|_{\theta_k=\frac{\psi_k}{\eps}}+\partial_{\theta_k}\sigma_k\left(y,\frac{\psi_k}{\eps}\right)\cdot \frac{y_0-\phi_k}{\eps}M_\nu\beta_k(y).
\end{align}
This has the form $P(y,\frac{\psi_k}{\eps})+P(y,\frac{\psi_k}{\eps}) G^\eps(y)$, where $P\in \mathcal{P}_k$, $G^\eps\in\Gamma^\eps(\Omega^0)$ here and below can change from term to term.   A similar result is obtained for any choice of generator $M_j$ in \eqref{32g}. Thus, with Remark \ref{32e} an induction on $|\gamma|$ shows that 
$(M_0,\dots,M_{n+1})^\gamma  \sigma_k\left(y,\frac{\psi_k}{\eps}\right)$ is a finite sum of terms of the form 
$P(y,\frac{\psi_k}{\eps}) G^\eps(y)$. 
We have 
\begin{align}\label{32hh}
\left|P(y,\frac{\psi_k}{\eps}) G^\eps(y)\right|_{L^\infty(\Omega^0)}\lesssim 1,
\end{align}
since $P(y,\frac{\psi_k}{\eps})$ has support in $J_\eps$, so by the estimate \eqref{44a}, this concludes the proof.\footnote{Henceforth, we will usually apply \eqref{44a} without comment.} 

\end{proof}

\subsubsection{Estimate of $V(y,\frac{\psi}{\eps})$.}

Next we estimate the noninteraction term of $U^\eps_1$.

\begin{prop}\label{32h}
For all $r\in \mathbb{N}_0:=\{0,1,2,3,\dots\}$ we have
\begin{align}\label{32i}
\begin{split}
&(a) \left|V\left(y,\frac{\psi}{\eps}\right)\right|_{N^r(\Omega^0)}\lesssim 1\\
&(b) \left\langle V\left(y',0,\frac{y_0}{\eps},\dots,\frac{y_0}{\eps}\right)\right\rangle_{N^r(b\Omega^0)}\lesssim 1.
\end{split}
\end{align}
\end{prop}

\begin{proof}
\textbf{1. }We have $V(y,\theta)=\sum_kV_k(y,\theta_k)$ and 
$$V_k\left(y,\frac{\psi_k}{\eps}\right)=-\int^{\frac{\psi_k}{\eps}}_{-\infty}Q_k(y)H_{1k}(y,s)ds.$$
By Remark \ref{decV} this integral can be written as a finite sum of terms of the form
$$\left(\int^{\frac{\psi_k}{\eps}}_{-\infty}P(y,s)ds\right) \;C_k(y),$$
where $P(y,\theta_k)\in \mathcal{P}_k$ and $C_k(y)$ is a smooth $\eps-$independent function of $y$.   We can ignore $C_k(y)$ in proving \eqref{32i}, so it will suffice to show for each $k$ that 
\begin{align}\label{32ii}
\left|\int^{\frac{\psi_k}{\eps}}_{-\infty}P(y,s)ds\right|_{N^r(\Omega^0,\mathcal{M}_k)}\lesssim 1.
\end{align}

\textbf{2. }We compute for example 
\begin{align}\label{32j}
M_0\int^{\frac{\psi_k}{\eps}}_{-\infty}P(y,s)ds=\int^{\frac{\psi_k}{\eps}}_{-\infty}M_0P(y,s)ds+\left(M_0\frac{\psi_k}{\eps}\right) P\left(y,\frac{\psi_k}{\eps}\right),
\end{align}
which has the form $\int^{\frac{\psi_k}{\eps}}_{-\infty}P(y,s)ds+G^\eps(y)P\left(y,\frac{\psi_k}{\eps}\right)$, where $G^\eps\in\Gamma^\eps(\Omega^0)$.   A similar result for the other choices of $M_j$ and an induction on $|\gamma|$ show that $(M_0,\dots,M_{n+1})^\gamma \int^{\frac{\psi_k}{\eps}}_{-\infty}P(y,s)ds$ is a finite sum of terms of the form\footnote{Recall \eqref{32g}.}
$$\int^{\frac{\psi_k}{\eps}}_{-\infty}P(y,s)ds \text{ or }G^\eps(y)P\left(y,\frac{\psi_k}{\eps}\right)$$
An obvious estimate of the integral and  \eqref{32hh} yields \eqref{32ii}.

A parallel argument using $b\Gamma^\eps$ \eqref{32d} and $b\mathcal{P}_k$ \eqref{32f} proves \eqref{32i}(b).
\end{proof}

As an immediate corollary of the preceding proof we give an estimate of a piece of $\mathcal{E}_1^\eps$.
\begin{cor}\label{33a}
For all $r\in \mathbb{N}_0:=\{0,1,2,3,\dots\}$ we have
\begin{align}\label{33b}
\left|(L(y,\partial_y)V(y,\theta))|_{\theta=\frac{\psi}{\eps}}\right|_{N^r(\Omega^0)}\lesssim 1.
\end{align}

\end{cor}

\subsubsection{Estimates of $(H_2-\mathcal{H}_2)|_{\theta=\frac{\psi}{\eps},\theta_0=\frac{y_0}{\eps}, \xi_n=\frac{y_n}{\eps}}$ and $W^\eps|_{\theta_0=\frac{y_0}{\eps},\xi_n=\frac{y_n}{\eps}}$.}\label{H2}  
We give the estimate of $W^\eps$ in Proposition \ref{33cc} below.  First we do a simpler, related estimate of the contribution to $\mathcal{E}_0^\eps$ given by the $H_2-\mathcal{H}_2$ term.

By Proposition \ref{suppprop} the term $(H_2-\mathcal{H}_2)|_{\theta=\frac{\psi}{\eps},\theta_0=\frac{y_0}{\eps},\xi_n=\frac{y_n}{\eps}}$ is supported in the interaction region $I_\eps$, a fact that simplifies the analysis and makes possible an estimate of the $N^k(\Omega_T,\mathcal{M}_0)$ norm.\footnote{This dominates the $N^k(\Omega^0)$ norm since $\mathcal{M}_k\subset \mathcal{M}_0$ for all $k$.}  Recall that a set of generators of $\mathcal{M}_0$ is given by $\{V_j,j=1,\dots, n+3\}$, which are respectively
\begin{align}\label{gen}
y_0\partial_{0},\; y_0\partial_n,\; y_n\partial_0, \;y_n\partial_n, \;\partial_l, l=1\dots,n-1.
\end{align}

With 
$\psi_{m,0}(y)=y_0+\omega_m(y')y_n$, by \eqref{d0} we can write 
\begin{align}
\begin{split}
&\sigma_m(y,\frac{\psi_m}{\eps})-\sigma_m(y,\frac{\psi_{m,0}}{\eps})=\left(\int^1_0\partial_{\theta_m}\sigma_m(y,\frac{\psi_{m,0}}{\eps}+s(\frac{\psi_{m}}{\eps}-\frac{\psi_{m,0}}{\eps}))ds\right) c(y)\frac{y_n^2}{\eps}=
f\left(y,\frac{y_0}{\eps},\frac{y_n}{\eps}\right)\frac{y_n^2}{\eps},
\end{split}
\end{align}
where $f(y,\theta_0,\xi_n)$ is a smooth function of its arguments which is bounded along with its derivatives on \emph{bounded} subsets of $\Omega^0\times\mathbb{R}^2_{\theta_0,\xi_n}$.\footnote{With $\psi_m-\psi_{m,0}=c(y)y_n^2$, we have  $f(y,\theta_0,\xi_n)=c(y) \int^1_0\partial_{\theta_m}\sigma_m(y,\theta_0+\omega_m(y')\xi_n+sc(y)y_n\xi_n)ds.$ } 
 Let $\gamma$ be a multi-index with $|\gamma|=k\in\mathbb{N}_0$.  Then for some constant $C$ (possibly $0$), 
\begin{align}\label{27b}
(V_1,\dots,V_{n+3})^\gamma \left(\frac{y_n^2}{\eps}\right)= C\frac{(y_0,y_n)^\kappa}{\eps}\text{ for some multi-index }|\kappa|=2.
\end{align}
Observe for example that 
\begin{align}
y_0\partial_0 \left(f(y,\frac{y_0}{\eps},\frac{y_n}{\eps})\right)=\left(y_0\partial_0f+\partial_{\theta_0}f\;\theta_0\right)|_{\theta_0=\frac{y_0}{\eps},\xi_n=\frac{y_n}{\eps}}.
\end{align}
Similar computations with the other generators of $\mathcal{M}_0$ and an easy induction on $k=|\gamma|$ yield
\begin{align}\label{27c}
(V_1,\dots,V_{n+3})^\gamma\left[f\left(y,\frac{y_0}{\eps},\frac{y_n}{\eps}\right)\right]=\sum^k_{L=0}\sum_{|\alpha|=L}f_{L,\alpha}\left(y,\frac{y_0}{\eps},\frac{y_n}{\eps}\right)\cdot \frac{(y_0,y_n)^\alpha}{\eps^L}, 
\end{align}
where  $f_{L,\alpha}(y,\theta_0,\xi_n)$ has the same properties  as $f$.  Thus,\footnote{Here we use that $\frac{y_n}{\eps}$, $\frac{y_0}{\eps}$ are bounded on $I_\eps$.}
\begin{align}\label{27d}
\left|(V_1,\dots,V_{n+3})^\gamma\left[f\left(y,\frac{y_0}{\eps},\frac{y_n}{\eps}\right)\right]\right|_{N^k(\Omega^0\cap I_\eps,\mathcal{M}_0)}\lesssim 1, \text{ since }\left|\frac{(y_0,y_n)^\alpha}{\eps^L}\right|\lesssim 1\text{ on }I_\eps \text{ when }|\alpha|=L.
\end{align}
We also have $\left|\frac{(y_0,y_n)^\kappa}{\eps}\right|\lesssim \eps$ on $I_\eps$ when $|\kappa|=2$.
With  \eqref{27d} this implies
\begin{align}\label{27a}
\left|\sigma_m(y,\frac{\psi_m}{\eps})-\sigma_m(y,\frac{\psi_{m,0}}{\eps})\right|_{N^k(\Omega^0\cap I_\eps,\mathcal{M}_0)}\lesssim \eps.
\end{align}

\begin{prop}\label{27za}
We have
\begin{align}\label{27z}
\left|(H_2-\mathcal{H}_2)|_{\theta=\frac{\psi}{\eps},\theta_0=\frac{\psi_0}{\eps},\xi_n=\frac{y_n}{\eps}}\right|_{N^k(\Omega^0,\mathcal{M}_0)}\lesssim \eps.
\end{align}
\end{prop}

\begin{proof}
Using Remark \ref{d2a} we reduce to estimating in $\Omega_T\cap I_\eps$ products like
\begin{align}
\begin{split}
&\left[\sigma_m(y,\frac{\psi_m}{\eps})-\sigma_m(y,\frac{\psi_{m,0}}{\eps})\right]\cdot \sigma_p(y,\frac{\psi_p}{\eps})c(y)=\\
 &\qquad f\left(y,\frac{y_0}{\eps},\frac{y_n}{\eps}\right)\frac{y_n^2}{\eps}\cdot g\left(y,\frac{y_0}{\eps},\frac{y_n}{\eps}\right)=h\left(y,\frac{y_0}{\eps},\frac{y_n}{\eps}\right)\frac{y_n^2}{\eps},
\end{split}
\end{align}
where $g(y,\theta_0,\xi_n)$ and $h(y,\theta_0,\xi_n)$ have the same properties as $f$.  
So essentially the same estimate as the one that gave \eqref{27a} yields the result.
\end{proof}



Next we estimate 
\begin{align}\label{33c}
W^\eps(y,\theta_0,\xi_n)|_{\theta_0=\frac{y_0}{\eps},\xi_n =\frac{y_n}{\eps}}=\chi^\eps(y_0,y_n)\sum_kt_k(y,\theta_0,\xi_n)r_k(y',0)|_{\theta_0=\frac{y_0}{\eps},\xi_n =\frac{y_n}{\eps}},
\end{align}
where each  $t_k(y,\theta_0,\xi_n)$ is $C^\infty$ with derivatives that are bounded on bounded subsets of $\Omega^0\times\mathbb{R}^2_{\theta_0,\xi_n}$ by   Remark \ref{d10yy}.

\begin{prop}\label{33cc}
For all $r\in \mathbb{N}_0:=\{0,1,2,3,\dots\}$ we have
\begin{align}
\begin{split}
&(a) \left|W^\eps\left(y,\frac{y_0}{\eps},\frac{y_n}{\eps}\right)\right|_{N^r(\Omega^0)}\lesssim 1\\
&(b) \left\langle W^\eps\left(y',0,\frac{y_0}{\eps},0\right)\right\rangle_{N^r(b\Omega^0)}\lesssim 1.
\end{split}
\end{align}

\end{prop}

This will be an immediate consequence of Propositions \ref{28r}, \ref{30a}, and \ref{36a}  below. 
In preparation for the proof note that by \eqref{d10z} each $t_k(y,\theta_0,\xi_n)$ is a finite sum of terms of the form
\begin{align}\label{33d}
c(y)\int^{\xi_n}_{+\infty}\sigma_m\left(y,\theta_0+\omega_k(y')\xi_n+s(\omega_m(y')-\omega_k(y'))\right)\sigma_p\left(y,\theta_0+\omega_k(y')\xi_n+s(\omega_p(y')-\omega_k(y'))\right) ds,
\end{align}
where $m\neq p$.   Unfortunately, $t_k(y,\theta_0,\xi_n)|_{\theta_0=\frac{y_0}{\eps},\xi_n =\frac{y_n}{\eps}}$ is \emph{not} supported in the interaction region $I_\eps$.\footnote{If it were, we could essentially repeat the estimates of section \ref{H2} here.} But we do have:

\begin{lem}\label{28g}
For $y\in \Omega^0$ we have for every $k$:
\begin{align}\label{28h}
|y_0+\omega_k(y')y_n|\lesssim \eps\text{ on the support of }t_k(y,\theta_0,\xi_n)|_{\theta_0=\frac{y_0}{\eps},\xi_n =\frac{y_n}{\eps}}.
\end{align}
In fact, there exists $M>0$ such that if $|y_0+\omega_k(y')y_n|\geq M\eps$, then the \emph{integrand} in \eqref{33d} vanishes for all $s\in\mathbb{R}$ when $\theta_0=\frac{y_0}{\eps},\xi_n =\frac{y_n}{\eps}$.

\end{lem}

\begin{proof}
It suffices to show there exists $M>0$ such that if $|y_0+\omega_k(y')y_n|\geq M\eps$, then no $s\in \mathbb{R}$ lies in the $s-$support of both factors 
of the integrand in \eqref{33d} when they are evaluated at $\theta_0=\frac{y_0}{\eps},\xi_n =\frac{y_n}{\eps}$.

\textbf{Case 1: both $m$ and $p$ are $\neq k$.}
As before let $\psi_{k,0}=y_0+\omega_k(y')y_n$.  Suppose both $\omega_m$ and $\omega_p$ are $>\omega_k$.\footnote{It will be clear that the other possible orderings can be handled the same way.} On the $s-$support of the first factor we have\footnote{The $\omega_k$ are well separated on $\Omega^0$ by \eqref{phases} and \eqref{phases2}.}
\begin{align}\label{28i}
|\psi_{k,0}+s(\omega_m-\omega_k)\eps|\leq 1\Rightarrow \left|s-\frac{\psi_{k,0}/\eps}{\omega_k-\omega_m}\right|\lesssim 1.
\end{align}
Similarly, on the $s-$support of the second factor we have 
\begin{align}\label{28j}
\left|s-\frac{\psi_{k,0}/\eps}{\omega_k-\omega_p}\right|\lesssim 1.
\end{align}
These conditions on $s$ are incompatible if $M$ is chosen sufficiently large and $|\psi_{k,0}|\geq M\eps$.

\textbf{Case 2: $m\neq k$, $p=k$.}\footnote{The remaining cases are handled just like cases 1 and 2.} In this case we can rewrite 
\eqref{33d} as 
\begin{align}
c(y)\sigma_p\left(y,\theta_0+\omega_p(y')\xi_n)\right)\int^{\xi_n}_{+\infty}\sigma_m\left(y,\theta_0+\omega_k(y')\xi_n+s(\omega_m(y')-\omega_k(y'))\right) ds,
\end{align}
The result follows since the factor $\sigma_p\left(y,\frac{\psi_{p,0}}{\eps}\right)$ is supported in $|\psi_{p,0}|\leq \eps$.
\end{proof}

In the next lemma we use the wedge partition of unity $\{\chi_j\}$ defined in section \ref{O0}.\footnote{The  plausibility of Lemma \ref{28m} is clear from a simple picture.}  
\begin{lem}\label{28m}
For $y\in \Omega^0$ we have when $j\neq k$:
\begin{align}\label{28n}
|y_0+\omega_k(y')y_n|\leq M\eps\text{ and }y\in \mathrm{supp}\;\chi_j\Rightarrow |y_0|\lesssim \eps
\text{ and }|y_n|\lesssim \eps.
\end{align}
\end{lem}

\begin{proof}
\textbf{1. }Recall the covering of $\{y_n> 0\}$ by overlapping wedges $W_j$, $j=0,1,\dots,N$ defined in section \ref{O0}.   There we also defined  smaller disjoint wedges $W_j'$ such that 
\begin{align}\label{28o}
\begin{split}
&(a) \Sigma_j\subset W_j'\subset\subset W_j \text{ for }j=1,\dots,N\\
&(b) W_j\cap \overline{W_k'}=\emptyset\text{ for }j\neq k,  \;j=0,\dots,N.
\end{split}
\end{align}
\textbf{2. }Let $\delta>0$ be as chosen in section \ref{O0} and choose $P>0$ such that $\frac{M}{P}<\frac{\delta}{2}$. 
Then
\begin{align}
\begin{split}
&|y_0+\omega_k(y')y_n|\leq M\eps \text{ and }|y_n|\geq P\eps\Rightarrow \left|\frac{y_0}{y_n}+\omega_k(y')\right|<\frac{\delta}{2}\Rightarrow\\
&\left|\frac{y_0}{y_n}-\gamma_k\right|=\left|\frac{y_0}{y_n}+\omega_k(y')-\gamma_k-\omega_k(y')\right|\leq \left|\frac{y_0}{y_n}+\omega_k(y')\right|+|\gamma_k+\omega_k(y')|<\frac {\delta}{2}+\frac{\delta}{2}=\delta.
\end{split}
\end{align}
For the last inequality we used \eqref{phases2}(d).  By definition of the wedges $W'_k$ this implies $y\in W_k'$, so by \eqref{28o}(b) we have $y\notin W_j=\mathrm{supp}\;\chi_j$.   Thus, if \eqref{28n} holds we must have $|y_n|\leq P\eps$, and thus also
$|y_0|\lesssim \eps$.

\end{proof}

In view of \eqref{33c} and \eqref{33d}, to estimate the $N^r(\Omega^0)$ norm of $W^\eps|_{\theta_0=\frac{y_0}{\eps},\xi_n=\frac{y_n}{\eps}}$ it suffices to estimate the $N^r(\Omega^0)$ norm of terms of the form
\begin{align}\label{28p}
T^\eps_{k,m,p}(y):=\chi^\eps(y_0,y_n)\; \left(\int^{\xi_n}_{+\infty}I_{k,m,p}(s;y,\theta_0,\xi_n)ds\right)|_{\theta_0=\frac{y_0}{\eps},\xi_n=\frac{y_n}{\eps}},  \;m\neq p, 
\end{align}
where $\int^{\xi_n}_{+\infty}I_{k,m,p}(s;y,\theta_0,\xi_n)ds$ denotes the integral \eqref{33d}.  
Here and below we ignore smooth functions like $c(y)$ in \eqref{33d}, which can have no significant effect on the estimates.   For a fixed choice of $k$ and $m\neq p$ we use the wedge partition of unity $\{\chi_j\}$, $\chi_j\in\Lambda(\Omega^0)$, constructed in section \ref{O0} to write
\begin{align}\label{28q}
T^\eps_{k,m,p}(y)=\sum^N_{j=0}\chi_j(y)T^\eps_{k,m,p}(y):=\sum_j T^\eps_{k,j}(y)
\end{align}
and proceed to estimate each $T^\eps_{k,j}$ in $N^r(\Omega^0,\mathcal{M}_j)$.   In fact, for $j\neq k$ Lemma \ref{28m} shows $T^\eps_{k,j}$ is supported in the interaction region $I_\eps$, and that will allow us to estimate $T^\eps_{k,j}$ in the stronger norm of $N^r(\Omega^0,\mathcal{M}_0)$.   We will do that first, and then estimate $T^\eps_{k,k}$ in 
$N^r(\Omega^0,\mathcal{M}_k)$.\footnote{The term  $T^\eps_{k,k}$ is not supported in the interaction region $I_\eps$.}

\begin{prop}\label{28r}
Consider $T^\eps_{j,k}$ as in \eqref{28q} when $j\neq k$. For $r\in \mathbb{N}_0$ we have
\begin{align}\label{28rr}
|T^\eps_{j,k}|_{N^r(\Omega^0,\mathcal{M}_0)}\lesssim 1.
\end{align}

\end{prop}

\begin{proof}
Since  $T^\eps_{k,j}$ has support in $I_\eps$, it is unchanged when the cutoff $\chi^\eps(y_0,y_n)$ is removed, so we will ignore that cutoff in this proof.
\footnote{Here we use $\eps\in (0,\eps_0]$ for some small enough $\eps_0<1$.}

We can write
\begin{align}\label{28s}
T^\eps_{k,j}(y)=\chi_j(y) f(y,\frac{y_0}{\eps},\frac{y_n}{\eps}),
\end{align}
where $\chi_j(y)\in \Lambda(\Omega^0)$ and by Remark \ref{d10yy},  $f(y,\theta_0,\xi_n)$ is a smooth function of its arguments with derivatives that are bounded on bounded subsets of $\Omega^0\times\mathbb{R}^2_{\theta_0,\xi_n}$.   We  now argue as in the proof of Proposition \ref{27za}.   In fact,  
we can use \eqref{27d} directly:
\begin{align}\label{28t}
\left|(V_1,\dots,V_{n+3})^\gamma\left[f\left(y,\frac{y_0}{\eps},\frac{y_n}{\eps}\right)\right]\right|_{N^r(\Omega^0\cap I_\eps,\mathcal{M}_0)}\lesssim 1,\;\;|\gamma|=r.
\end{align}
The estimate \eqref{28rr} now follows from  \eqref{28t} and the definition of $\Lambda(\Omega^0)$, Definition \ref{regul}.

\end{proof}

The most challenging term to estimate is the $T^\eps_{k,k}$ term.

\begin{prop}\label{30a}
Consider $T^\eps_{k,k}$ as in \eqref{28q}. For $r\in \mathbb{N}_0$ we have
\begin{align}\label{30b}
|T^\eps_{k,k}|_{N^r(\Omega^0,\mathcal{M}_k)}\lesssim 1.
\end{align}
\end{prop}

In preparation for the proof we recall:
\begin{align}\label{30bb}
\begin{split}
&T^\eps_{k,k}(y)=\chi^\eps(y_0,y_n)\chi_k(y)\cdot  \\
&\quad \int^{\frac{y_n}{\eps}}_{+\infty}\sigma_m\left(y,\frac{\psi_{k,0}}{\eps}+s(\omega_m(y')-\omega_k(y'))\right)\sigma_p\left(y,\frac{\psi_{k,0}}{\eps}+s(\omega_p(y')-\omega_k(y'))\right) ds\\
&\qquad  :=\chi^\eps(y_0,y_n)\chi_k(y)\cdot t^\eps_{k,k}(y).
\end{split}
\end{align}
As before we  denote the generators of $\mathcal{M}_k$ by $M_0,\dots,M_{n+1}$.   They are: 

$M_0=(y_0-\phi_k)\partial_0$

$M_\nu=\partial_\nu+(\partial_\nu\phi_k)\partial_0, \nu=1,\dots,n-1$

$M_n=y_n(\partial_n+(\partial_n\phi_k)\partial_0)$

$M_{n+1}=(y_0-\phi_k)\partial_n$.\\

From Lemma \ref{28g} and the definition of $\chi^\eps$, we obtain that\footnote{Use \eqref{28k}(c) to see that $|y-\phi_k|\lesssim \eps$ on the support of $T^\eps_{k,k}$.}
\begin{align}\label{30c}
\begin{split}
&\mathrm{supp}\;T^\eps_{k,k}(y) \subset I_{\sqrt{\eps}}:=\{y:|y-\phi_k|\lesssim \eps, |y_0|\lesssim\sqrt{\eps}, 0\leq y_n\lesssim \sqrt{\eps}\}\text{ and thus  }\\
&\quad \mathrm{supp}\;(M_0,\dots,M_{n+1})^\gamma T^\eps_{k,k}(y) \subset I_{\sqrt{\eps}} \text{ for any multi-index }\gamma.
\end{split}
\end{align}

\begin{defn}\label{30d}

Let 

(a)  $\Lambda^\eps(\Omega^0)=\{A^\eps\in C^\infty(\Omega^0, \mathbb{R}): |(M_0,\dots,M_{n+1})^\gamma A^\eps|_{L^\infty(\Omega^0\cap I_\eps)}\leq C_\gamma \text{ for }\eps\in (0,1], \gamma\in \mathbb{N}_0^{n+2}\}$
(b) $\Lambda^{\sqrt{\eps}}(\Omega^0)=\{B^\eps\in C^\infty(\Omega^0, \mathbb{R}): |(M_0,\dots,M_{n+1})^\gamma B^\eps|_{L^\infty(\Omega^0\cap I_{\sqrt{\eps}})}\leq C_\gamma \text{ for }\eps\in (0,1], \gamma\in \mathbb{N}_0^{n+2}\}.$
\end{defn}

\begin{rem}\label{30e}
\textup{The following  properties of these spaces are easy to check.}

\textup{1) Each of the spaces in Definition \ref{30d} is mapped to itself by the generators $V_j$ $j=0,\dots,n+1$.}

\textup{2) We have $\Lambda^{\sqrt{\eps}}(\Omega^0)\subset \Lambda^\eps(\Omega^0)$. }

\textup{3) Smooth functions $c(y)$ and functions of the forms $\frac{(y_0,y_n)^\kappa}{\eps}$ with $|\kappa|=2$ or $\frac{y_0-\phi_k}{\eps}$ belong to $\Lambda^{\sqrt{\eps}}(\Omega^0)$.   The cutoff $\chi^\eps(y_0,y_n)\in \Lambda^{\sqrt{\eps}}(\Omega^0).$}

\textup{4) The functions $\frac{y_n}{\eps}, \frac{y_0}{\eps}\in \Lambda^\eps(\Omega^0)\setminus \Lambda^{\sqrt{\eps}}(\Omega^0)$.}

\textup{5)  If $A_1^\eps, A^\eps_2\in \Lambda^\eps(\Omega^0)$, then the product $A_1^\eps\cdot A^\eps_2\in \Lambda^\eps(\Omega^0)$.
Similarly, if $B_1^\eps, B^\eps_2\in \Lambda^{\sqrt{\eps}}(\Omega^0)$, then $B_1^\eps\cdot B^\eps_2\in \Lambda^{\sqrt{\eps}}(\Omega^0)$.}
\end{rem}

\begin{proof}[Proof of Proposition \ref{30a}]
\textbf{1. }Since $\chi_k(y)\in\Lambda(\Omega^0)$, $\chi^\eps(y_0,y_n)\in\Lambda^{\sqrt{\eps}}(\Omega^0)$, and  \eqref{30c} holds,  in order to prove \eqref{30b} it will suffice to prove for $t^\eps_{k,k}$ as in \eqref{30bb} that 
\begin{align}\label{30g}
|(M_0,\dots,M_{n+1})^\gamma t^\eps_{k,k}|_{L^\infty(\Omega^0\cap I_{\sqrt{\eps}})}\lesssim 1 \text{ for any }\gamma.
\end{align}

\textbf{2. }To motivate the rest of the argument we examine the action of $M_\nu$ and $M_0$ on the second argument of $\sigma_m$ in \eqref{30bb}.  Below we let  $c(y)$ denote a smooth function that can change from term to term.  With $\psi_{k,0}(y):=y_0+\omega_k(y')y_n$ we see from \eqref{28k}(c) that 
\begin{align}\label{30gg}
\psi_{k,0}(y)=y_0-\phi_k(y'',y_n)+c(y)y_n^2+c(y)y_0y_n.
\end{align}
Using \eqref{30gg} and $\partial_\nu \phi_k(y'',y_n)=O(y_n)$, we compute\footnote{Note that $M_\nu (y_0-\phi_k)=(\partial_\nu+\partial_\nu\phi_k\partial_0)(y_0-\phi_k)=0$.  Having $\frac{y_n}{\eps}$ on the right in \eqref{30h}(a) would not be good enough.
The cutoff $\chi^\eps(y_0,y_n)$  enforces boundedness of $\frac{y_n^2}{\eps}$ and $\frac{y_0y_n}{\eps}.$}
\begin{align}\label{30h}
\begin{split}
&(a) (\partial_\nu+\partial_\nu\phi_k\partial_0)\left[\frac{\psi_{k,0}}{\eps}+s(\omega_m(y')-\omega_k(y'))\right]=c(y)\frac{y_n^2}{\eps}+c(y)\frac{y_0y_n}{\eps}+sc(y)\\
&(b) (y_0-\phi_k)\partial_0 \left[\frac{\psi_{k,0}}{\eps}+s(\omega_m(y')-\omega_k(y'))\right]=c(y)\frac{y_0-\phi_k}{\eps}+sc(y).
\end{split}
\end{align}
In each case we obtain a function of the form $B^\eps+sB^\eps$, where here and below we let $B^\eps$ denote an element of $\Lambda^{\sqrt{\eps}}(\Omega^0)$  that may change from term to term.
We get  results of the same  form if $M_n$ or $M_{n+1}$ are used here in place of $M_0$ and $M_\nu$. 

\textbf{3. }Suppose first that we let $(M_0,\dots,M_{n+1})^\gamma$ act \emph{only under the integral sign} in the expression for $t^\eps_{k,k}$.  Step \textbf{2} and an induction on $|\gamma|$ show that we obtain a finite sum of terms of the form. $J^\eps(y):=$
\begin{align}\label{30i}
\int^{\frac{y_n}{\eps}}_{+\infty}(\partial^\alpha_{y,\theta_m}\sigma_m)\left(y,\frac{\psi_{k,0}}{\eps}+s(\omega_m(y')-\omega_k(y'))\right)(\partial^\beta_{y,\theta_p}\sigma_p)\left(y,\frac{\psi_{k,0}}{\eps}+s(\omega_p(y')-\omega_k(y'))\right) B^\eps(y)s^lds
\end{align}
for some $l\in\mathbb{N}_0$ and multi-indices $\alpha,\beta$.   To see that $|J^\eps(y)|\lesssim 1$ make the change of variable\footnote{Here we use that at least one of $m$, $p$ is different from $k$, say $m$.}
\begin{align}\label{30j}
t=\frac{\psi_{k,0}}{\eps}+s(\omega_m(y')-\omega_k(y'), \text{ so }s=\frac{t-\frac{\psi_{k,0}}{\eps}}{\omega_m-\omega_k}\text{ and }ds=\frac{dt}{\omega_m-\omega_k}.
\end{align}
Since $|t|\leq 1$ on the support of $(\partial^\alpha_{y,\theta_m}\sigma_m)(y,t)$ and the integrand of \eqref{30i} is nonvanishing only when $|\frac{\psi_{k,0}}{\eps}|\lesssim 1$ (by Lemma \ref{28g}), we have $|s|\lesssim 1$ on the support of that integrand.

\textbf{4. }Only the vector fields $M_n, M_{n+1}$ differentiate the upper limit of integration $\frac{y_n}{\eps}$ in \eqref{30i}.   For example, when $M_n=y_n(\partial_n+\partial_n\phi_k)\partial_0$ is applied to an integral of the form \eqref{30i}, we get an additional term of the form 
\begin{align}\label{30k}
K^\eps(y):=(\partial^\alpha_{y,\theta_m}\sigma_m)\left(y,\frac{\psi_{m,0}}{\eps}\right)(\partial^\beta_{y,\theta_p}\sigma_p)\left(y,\frac{\psi_{p,0}}{\eps}\right) B^\eps(y)\left(\frac{y_n}{\eps}\right)^{l+1}.
\end{align}
Observe that the product of the first two factors \emph{has support in} $I_\eps$ by Prop. \ref{suppprop}(b) and $\frac{y_n}{\eps}\in \Lambda^{{\eps}}(\Omega^0)$. 
The arguments used in the proof of Proposition \ref{27za}, in particular \eqref{27d}, show that 
\begin{align}\label{30l}
\left|(M_0,\dots,M_{n+1})^\gamma \left[(\partial^\alpha_{y,\theta_m}\sigma_m)\left(y,\frac{\psi_{m,0}}{\eps}\right)(\partial^\beta_{y,\theta_p}\sigma_p)\left(y,\frac{\psi_{p,0}}{\eps}\right)\right]\right|_{L^\infty(\Omega^0\cap I_{{\eps}})}\lesssim 1,
\end{align}
and Remark \ref{30e} implies\footnote{Note that $B^\eps$ and $\left(\frac{y_n}{\eps}\right)^{l+1}$ lie in $\Lambda^\eps(\Omega^0)$.}
\begin{align}\label{30m}
\left|(M_0,\dots,M_{n+1})^\gamma \left(B^\eps(y) \left(\frac{y_n}{\eps}\right)^{l+1}\right)\right|_{L^\infty(\Omega^0\cap I_{{\eps}})}\lesssim 1.
\end{align}
Since $K^\eps$ has support in $I_\eps$, the estimates \eqref{30l} and \eqref{30m} imply
\begin{align}\label{30n}
|(M_0,\dots,M_{n+1})^\gamma K^\eps|_{L^\infty(\Omega^0\cap I_{\sqrt{\eps}})}\lesssim 1 \text{ for any }\gamma.
\end{align}
The same kind of argument treats the additional term that arises when $M_{n+1}$ is applied to the upper limit of integration in \eqref{30i}. 
With the estimate of $J^\eps$ in step \textbf{3},  this finishes the proof of \eqref{30g} and Proposition \ref{30a}.
\end{proof}

\begin{prop}\label{36a}
For any $r\in\mathbb{N}_0$ we have 
\begin{align}\label{36b}
\left\langle W^\eps\left(y',0,\frac{y_0}{\eps},0\right)\right\rangle_{N^r(b\Omega^0)}\lesssim 1.
\end{align}
\end{prop}
\begin{proof}
We have $W^\eps(y',0,\frac{y_0}{\eps},0)=\chi(\frac{y_0}{\sqrt{\eps}})\sum_k t_k(y',0,\frac{y_0}{\eps},0)r_k(y',0)$, and by Lemma
\ref{28g}  $|y_0|\lesssim\eps$ on $\mathrm{supp}\;t_k(y',0,\frac{y_0}{\eps},0)$.   Thus, the cutoff $\chi(\frac{y_0}{\sqrt{\eps}})$ may be ignored here.
An argument parallel to the proof of Proposition \ref{28r}, but simpler, now gives \eqref{36b}.  There is no need to use the cutoffs $\chi_j$, the set $bJ_\eps$ \eqref{32dz} is used in place of $I_\eps$ in \eqref{28t}, and $\mathcal{M}'$ is used in place of $\mathcal{M}_0$.

\end{proof}

\subsection{Estimate of the remainder $r^\eps_a$.} \label{ra}

Recall from \eqref{d15} that $r^\eps_a=\mathcal{E}_0^\eps+\eps\mathcal{E}_1^\eps$.   In Proposition \ref{27za} we have already estimated the contribution to $\mathcal{E}_0^\eps$ given by the $H_2-\mathcal{H}_2$ term.  Next we estimate the other contribution to $\mathcal{E}_0^\eps$.

\subsubsection{Estimate of $[(\mathcal{L}_2-\mathcal{L}_{2,0})W^\eps]|_{\theta_0=\frac{y_0}{\eps},\xi_n=\frac{y_n}{\eps}}$.}\label{L20}
  
  With the estimate of $W^\eps$ in hand it is now not hard to prove:
  \begin{prop}\label{L20a}
  For every $r\in\mathbb{N}_0$ we have
  \begin{align}\label{L20b}
  \left|[(\mathcal{L}_2-\mathcal{L}_{2,0})W^\eps]|_{\theta_0=\frac{y_0}{\eps},\xi_n=\frac{y_n}{\eps}}\right|_{N^r(\Omega^0)}\lesssim\sqrt{\eps}.
  \end{align}
  \end{prop}

\begin{proof}
  
  \textbf{1. }Since $\psi_0(y')=y_0$, we have
\begin{align}\label{34a}
(\mathcal{L}_2-\mathcal{L}_{2,0})W^\eps=[\mathcal{A}(y,d'\psi_0)-\mathcal{A}(y',0,d'\psi_0)]\partial_{\theta_0}W^\eps=[B_0(y)-B_0(y',0)]\partial_{\theta_0}W^\eps,
\end{align}
where $W^\eps(y,\theta_0,\xi_n)=\chi^\eps(y_0,y_n)\sum_kt_k(y,\theta_0,\xi_n)r_k(y',0)$ and  $t_k(y,\theta_0,\xi_n)$ is $C^\infty$ with derivatives that are bounded on bounded subsets of $\Omega^0\times\mathbb{R}^2_{\theta_0,\xi_n}$; see Remark \ref{d10yy}.   Thus,
\begin{align}\label{34b}
[(\mathcal{L}_2-\mathcal{L}_{2,0})W^\eps]|_{\theta_0=\frac{y_0}{\eps},\xi_n =\frac{y_n}{\eps}}=\chi^\eps(y_0,y_n)[B_0(y)-B_0(y',0)]\sum_k\partial_{\theta_0}t_k(y,\theta_0,\xi_n)r_k(y',0)|_{\theta_0=\frac{y_0}{\eps},\xi_n =\frac{y_n}{\eps}}.
\end{align}
In view of \eqref{34b}  in place of \eqref{28p} we should now estimate terms of the form 
\begin{align}\label{34c}
\mathcal{T}^\eps_{k,m,p}(y):=\chi^\eps(y_0,y_n)y_n\; \left(\int^{\xi_n}_{+\infty}\mathcal{I}_{k,m,p}(s;y,\theta_0,\xi_n)ds\right)|_{\theta_0=\frac{y_0}{\eps},\xi_n=\frac{y_n}{\eps}},  \;m\neq p, 
\end{align}
where $\int^{\xi_n}_{+\infty}\mathcal{I}_{k,m,p}(s;y,\theta_0,\xi_n)ds$ denotes the integral \eqref{33d} with one of $\sigma_m$ or $\sigma_p$ replaced by $\partial_{\theta_0}\sigma_m$ or $\partial_{\theta_0}\sigma_p$.

\textbf{2. }Parallel to \eqref{28q} we now write
$$\mathcal{T}^\eps_{k,m,p}(y)=\sum_j\chi_j(y)\mathcal{T}^\eps_{k,m,p}(y)=\sum_j\mathcal{T}^\eps_{k,j}(y).$$
We claim that 
\begin{align}\label{34d}
|\mathcal{T}^\eps_{k,j}(y)|_{N^r(\Omega^0,\mathcal{M}_0)}\lesssim \eps \text{ for }j\neq k.  
\end{align}
To see this use the argument of Proposition \ref{28r}, the observation that\footnote{Here the $V_j$ are the generators of $\mathcal{M}_0$.}
\begin{align}\label{34e}
\text{ for }|\gamma|\leq r\;\; (V_1,\dots,V_{n+3})^\gamma y_n=C(y_0,y_n)^\kappa \text{ where }|\kappa|=1,
\end{align}
and the fact $|(y_0,y_n)^\kappa|\lesssim \eps$  on $\mathrm{supp}\;\mathcal{T}^\eps_{k,j}\subset I_\eps$. 

\textbf{3. }Finally, we claim that 
\begin{align}\label{34f}
|\mathcal{T}^\eps_{k,k}(y)|_{N^r(\Omega^0,\mathcal{M}_k)}\lesssim \sqrt{\eps}.  
\end{align}
To see this,  first use an induction on $|\gamma|$ to show $$(M_0,\dots,M_{n+1})^\gamma y_n=(c_\gamma(y)(y_0-\phi_k),d_\gamma(y)y_n)^\kappa$$ for some smooth functions $c_\gamma$, $d_\gamma$ and multi-index $\kappa$ such that $|\kappa|=1$.\footnote{The $M_j$ are the generators of $\mathcal{M}_k.$}  The estimate \eqref{34f} then follows from the argument of Proposition \ref{30a} and  the fact that

$$|(c_\gamma(y)(y_0-\phi_k),d_\gamma(y)y_n)^\kappa|_{L^\infty(\Omega^0\cap I_{\sqrt{\eps}})}\lesssim \sqrt{\eps},$$

\end{proof}

\subsubsection{Estimate of $\eps\mathcal{E}_1^\eps(y)$.}\label{E1}
In  this section we prove the following proposition as a consequence of Corollary \ref{33a} and Propositions \ref{31b} and \ref{35c} below:
 
  \begin{prop}\label{E1a}
  For every $r\in\mathbb{N}_0$ and $\mathcal{E}_1^\eps$ as in \eqref{d15} we have
  \begin{align}\label{E1b}
  |\eps\mathcal{E}_1^\eps(y)|_{N^r(\Omega^0)}\lesssim\sqrt{\eps}.
  \end{align}
  \end{prop}

Consider the term $\eps  \left[L(y,\partial_y)W^\eps(y,\theta_0,\xi_n)\right]|_{\theta_0=\frac{y_0}{\eps},\xi_n=\frac{y_n}{\eps}}$,
where
\begin{align}\label{31a}
L(y,\partial_y)W^\eps=\sum^n_{j=0}B_j(y)\partial_j\left[\chi^\eps(y_0,y_n)\sum_kt_k(y,\theta_0,\xi_n)r_k(y',0)\right].
\end{align}

\begin{prop}\label{31b}
For all $r\in\mathbb{N}_0$ we have
 \begin{align}
 \left|\left[L(y,\partial_y)W^\eps(y,\theta_0,\xi_n)\right]|_{\theta_0=\frac{y_0}{\eps},\xi_n=\frac{y_n}{\eps}}\right|_{N^r(\Omega^0)}\lesssim\frac{1}{\sqrt{\eps}}.
\end{align}
\end{prop}

\begin{proof}

As usual we are  free to ignore smooth $\eps-$independent functions of $y$  in the estimates, so we focus on estimating for example:
\begin{align}\label{31aa}
\begin{split}
&\left[\partial_0\left[\chi\left(\frac{y_0}{\sqrt{\eps}}\right)\chi\left(\frac{y_n}{\sqrt{\eps}}\right)t_k(y,\theta_0,\xi_n)\right]\right]|_{\theta_0=\frac{y_0}{\eps},\xi_n=\frac{y_n}{\eps}}=\left[\frac{1}{\sqrt{\eps}}\chi'\left(\frac{y_0}{\sqrt{\eps}}\right)\chi\left(\frac{y_n}{\sqrt{\eps}}\right)t_k(y,\theta_0,\xi_n)\right]|_{\theta_0=\frac{y_0}{\eps},\xi_n=\frac{y_n}{\eps}}\;+\\
&\qquad \left[\chi\left(\frac{y_0}{\sqrt{\eps}}\right)\chi\left(\frac{y_n}{\sqrt{\eps}}\right)\partial_0t_k(y,\theta_0,\xi_n)\right]|_{\theta_0=\frac{y_0}{\eps},\xi_n=\frac{y_n}{\eps}}:=P^\eps_k(y)+Q^\eps_k(y).
\end{split}
\end{align}
We claim
\begin{align}\label{31c}
|P^\eps_k|_{N^r(\Omega^0)}\lesssim \frac{1}{\sqrt{\eps}}\text{ and }|Q^\eps_k|_{N^r(\Omega^0)}\lesssim \frac{1}{\sqrt{\eps}}.
\end{align}

Each of $P^\eps_k$, $Q^\eps_k$  is similar to  $T^\eps_{k,m,p}$  in \eqref{28p},  and so can be estimated by the argument used to prove Proposition \ref{33cc}. This is clear for $P^\eps_k$, since $\chi'\left(\frac{y_0}{\sqrt{\eps}}\right)\chi\left(\frac{y_n}{\sqrt{\eps}}\right)\in\Lambda^{\sqrt{\eps}}(\Omega^0)$.  To treat $Q^\eps_k$ first use the formula \eqref{d10z} to see that $\partial_0 t_k(y,\theta_0,\xi_n)$ is a sum of terms of the form
\begin{align}\label{31d}
\int^{\xi_n}_{+\infty}\left[\partial_0\sigma_m(\dots)+\partial_{\theta_m}\sigma_m(\dots)\left((\partial_0\omega_k)\xi_n+s(\partial_0\omega_m-\partial_0\omega_k)\right)\right]\sigma_p(\dots)ds+II.
\end{align}
Here the derivatives of $\sigma_m$ are evaluated at $(y,\theta_0+\omega_k(y')\xi_n+s(\omega_m(y')-\omega_k(y')))$, and $II$ is a similar term where $\sigma_p$ is differentiated and not $\sigma_m$. When \eqref{31d} is evaluated at $\theta_0=\frac{y_0}{\eps}$, $\xi_n=\frac{y_n}{\eps}$, the term $\partial_0\omega_k(y')\frac{y_n}{\eps}$ appears.   The support of $\chi_j(y)Q^\eps_k$ when $j\neq k$ is contained in $I_\eps$ by Lemma \ref{28m} and we have $\frac{y_n}{\eps}\in\Lambda^\eps(\Omega^0)$.\footnote{Recall Definition \ref{30d}.}  But the support of $\chi_k(y)Q^\eps_k$ is contained in $I_{\sqrt{\eps}}$, and  we only have
$\frac{y_n}{\eps}=\frac{1}{\sqrt{\eps}}\frac{y_n}{\sqrt{\eps}}$ with $\frac{y_n}{\sqrt{\eps}}\in\Lambda^{\sqrt{\eps}}(\Omega^0)$. 

The  estimates \eqref{31c} still hold when $\partial_0$ is replaced by $\partial_i$, $i\neq 0$ in \eqref{31aa}.
\end{proof}

To estimate the term of $\eps\mathcal{E}_1^\eps$ given by
$\eps K(y,U_0,\eps \mathcal{U}^\eps_1)\mathcal{U}^\eps_1$,  we can apply the following general result from \cite{metajm}.  In Proposition \ref{35a}   $T_0$ need not be small and for any $0<T<T_0$, $\Omega_T$ is a truncated backward cone as defined  in \eqref{a1z}.  

\begin{prop}\label{35a}
 (a) For  $M\in\mathbb{N}$ let $F(y,Z)\in C^\infty(\mathbb{R}^{1+n}_+\times \mathbb{R}^M,\mathbb{R}^N)$.    For  any $0<T_1<T_0$ and $r\in\mathbb{N}_0$ there exist a constant $C_0>0$ and an increasing function $h:\mathbb{R}_+\to \mathbb{R}_+$ such that for any $0<T\leq T_1$,  if $Z\in L^\infty\cap N^r(\Omega_T)$, then $f(y):=F(y,Z(y))$ belongs to $L^\infty\cap N^r(\Omega_T)$ and satisfies\footnote{The constant $C_0$ and function $h$ depend on $T_1$ but not on $T$.}
\begin{align}\label{35b}
\begin{split}
& (i) |f|_{L^\infty(\Omega_T)}\leq h(|Z|_{L^\infty(\Omega_T)}) \\
& (ii) |f|_{N^r(\Omega_T)}\leq C_0+ h(|Z|_{L^\infty(\Omega_T)})|Z|_{N^r(\Omega_T)}.
\end{split}
\end{align}

(b) For $T_1$ as above and any $r\in\mathbb{N}_0$  there exists $C>0$ such for any $0<T\leq T_1$, if  $u$, $v$ are real valued functions in $L^\infty\cap N^r(\Omega_T)$
then 
\begin{align}\label{35bz}
 |uv|_{L^\infty\cap N^r(\Omega_T)}\leq C|u|_{L^\infty\cap N^r(\Omega_T)}|v|_{L^\infty\cap N^r(\Omega_T)}.
\end{align}
\end{prop}

\begin{rem}\label{40a}
\textup{The  estimate \eqref{35b}(i) is trivial.    Estimate \ref{35b}(ii) is stated in Proposition 3.2.2 of \cite{metajm}.     The corresponding estimate for  functions on $\mathbb{R}^{1+n}_+$  is a consequence of Lemmas 3.3.1 and 3.3.3 of the companion paper \cite{metduke}.  The estimate on $\Omega_T$ is then deduced using the bounded extension operators $R_T:L^\infty\cap N^m(\Omega_T)\to L^\infty\cap N^m(\mathbb{R}^{1+n}_+)$ constructed in Proposition 3.2.1 of \cite{metajm}.   The estimate \eqref{35bz}  is not explicitly stated in \cite{metajm} or  \cite{metduke}, but it too is a direct consequence of Lemmas 3.3.1 and 3.3.3 of  \cite{metduke} and Proposition 3.2.1 of \cite{metajm}.}

\end{rem}

\begin{prop}\label{35c}
Let $\Omega_T$ be as in Proposition \ref{35a} and suppose $\Omega_T\subset \Omega^0$ for  $\Omega^0$  as chosen in section \ref{O0}.  For any $r\in \mathbb{N}_0$ we have
\begin{align}\label{35d}
\left|[\eps K(y,U_0,\eps \mathcal{U}^\eps_1)\mathcal{U}^\eps_1]|_{\theta=\frac{\psi}{\eps}, \theta_0=\frac{y_0}{\eps},\xi_n=\frac{y_n}{\eps}} \right|_{N^r(\Omega_T)}\lesssim \eps. 
\end{align}
\end{prop}

\begin{proof}
By construction we have $\left|(U_0,\mathcal{U}^\eps_1)|_{\theta=\frac{\psi}{\eps}, \theta_0=\frac{y_0}{\eps},\xi_n=\frac{y_n}{\eps}}\right|_{L^\infty(\Omega_T)}\lesssim 1$, so  the result follows directly from Propositions \ref{32a} and \ref{35a}.

\end{proof}

The formulas \eqref{d15} and the estimates of this section  imply:
\begin{prop}\label{39a}
Let $\Omega_T$ be as in Proposition \ref{35a} and suppose $\Omega_T\subset \Omega^0$.  For any $r\in\mathbb{N}_0$ we have
\begin{align}
|r^\eps_a|_{L^\infty\cap N^r(\Omega_T)}\lesssim \sqrt{\eps},
\end{align}
where the implied constant is independent of $0<T<T_1$.
\end{prop}

\subsection{Estimate of the error $w^\eps=u^\eps-u^\eps_a$.}\label{weps} 
In this section we finish the proof of Theorem \ref{mr}. 

For $(L(y,\partial_y),B)$ as in section \ref{assumptions}  consider the linear boundary problem on $\mathbb{R}^{1+n}_+$:
\begin{align}\label{37a}
\begin{split}
&L(y,\partial_y)v=f \text{ in }y_n>0\\
& B(y')v|_{y_n=0}=g \\
& v=0\text{ in }y_0<0,
\end{split}
\end{align}
where $f$ and $g$ are zero in $y_0<0$.    We will use the following estimate from \cite{metajm}.\footnote{This estimate combines 
the estimates of Corollary 4.1.4 and Proposition 5.1.2 of  \cite{metajm}.} In this general result we let $\Omega_T$ denote any domain of determinacy as in section \ref{bpexact}.   

\begin{prop}\label{37b}
If $T_0$ as in \eqref{a1z}  is small enough,  for $m>\frac{n+5}{2}$ and $0<T_1<T_0$  there exists $C>0$ such that 
for any $0<T\leq T_1$ the solution $u$ of \eqref{37a} satisfies
\begin{align}\label{37bb}
|v|_{L^\infty\cap N^m(\Omega_T)}\leq C\left[\sqrt{T}|f|_{L^\infty\cap N^m(\Omega_T)}+\langle g\rangle_{L^\infty(b\Omega_T)}+\sqrt{T}\langle g \rangle_{N^m(b\Omega_T})\right].
\end{align}
\end{prop}

We will apply this proposition to the error problem \eqref{error} satisfied by $w^\eps=u^\eps-u^\eps_a$:
\begin{align}\label{error2}
\begin{split}
&L(y,\partial_y)w^\eps=D(y,u^\eps,u^\eps_a)w^\eps-r^\eps_a:=f^\eps\text{ in }y_n>0\\
&B(y')w^\eps|_{y_n=0}=-\eps B(y')U^\eps_1(y',0):=g^\eps\\
&w^\eps=0\text{ in }y_0<0.
\end{split}
\end{align}
Suppose now that the parameters $T_0$ and $\alpha$ as in \eqref{a1z}  are chosen small enough so that Propositions \ref{ex2} and \eqref{37b} apply, and fix some $T_1$ with $0<T_1<T_0$ such that  
the exact solution $u^\eps$ 
of \eqref{ia1} satisfies $|u^\eps|_{L^\infty\cap N^m(\Omega_{T_1})}\lesssim 1$.   We also require $\Omega_{T_1}\subset \Omega^0$. 
  We have
\begin{align}\label{37c}
|u^\eps_a|_{L^\infty\cap N^m(\Omega_{T_1})}\lesssim 1\text{ and }|r^\eps_a|_{L^\infty\cap N^m(\Omega_{T_1})}\lesssim \sqrt{\eps}
\end{align}
by Propositions \ref{32a} and \ref{39a} and 
\begin{align}\label{37d}
\langle g^\eps\rangle_{L^\infty\cap N^m(\Omega_{T_1})}\lesssim \eps
\end{align}
by Propositions \ref{32h} and \ref{33cc}.     We use Proposition \ref{35a}(a),(b) to estimate for $0<T\leq T_1$: 
\begin{align}\label{37e}
|D(y,u^\eps,u^\eps_a)w^\eps|_{L^\infty\cap N^m(\Omega_T)}\lesssim |w^\eps|_{L^\infty\cap N^m(\Omega_T)}\Rightarrow |f^\eps|_{L^\infty\cap N^m(\Omega_T)}\lesssim |w^\eps|_{L^\infty\cap N^m(\Omega_T)}+\sqrt{\eps}.
\end{align}
Applying the estimate \eqref{37bb} we obtain:
\begin{align}\label{37f}
|w^\eps|_{L^\infty\cap N^m(\Omega_T)}\leq  C\left[\sqrt{T}(|w^\eps|_{L^\infty\cap N^m(\Omega_T)}+\sqrt{\eps})+\eps+\sqrt{T}\;\eps\right].
\end{align}
Decreasing $T$ if necessary so that $C\sqrt{T}\leq \frac{1}{2}$,  we thus obtain
\begin{align}\label{37g}
|w^\eps|_{L^\infty\cap N^m(\Omega_T)}\leq 2C(\sqrt{T}\sqrt{\eps}+\eps+\sqrt{T}\eps)=
O(\sqrt{\eps}).
\end{align}

\section{Extension to general nonlinear terms $f(y,u)$.}\label{genf} 
In the preceding sections we worked with a function $f(y,u)$ that was quadratic in $u$, mainly in order to be able to see pulse interactions explicitly.   We discuss here how the results of this paper can be extended to the case where $f$ is any element of $C^\infty(\mathbb{R}^{1+n}_+\times \mathbb{R}^N,\mathbb{R}^N)$  such that $f(y,0)=0$.   \\

\textbf{Construction of $U_0$.} With $U_0(y,\theta)=\sum_k\sigma_k(y,\theta_k) r_k(y)$ we write with slight abuse $f(y,U_0)=f(y,\sigma)$.     Recall the following definition:
\begin{defn}\label{41aa}
If $f(y,\sigma)$ is any element of $C^\infty(\mathbb{R}^{1+n}_+\times\mathbb{R}^N,\mathbb{R}^N)$  (or $C^\infty(\mathbb{R}^{1+n}_+\times\mathbb{R}^N,\mathbb{R})$) such that $f(y,0)=0$, 
set $f^m(y,\sigma_m):=f(y,\sigma_me_m)$, where $e_m$ is the $m-$th standard basis vector of $\mathbb{R}^N$.
\end{defn}

 We define $Ef=PSf$, where 
\begin{align}\label{41a}
\begin{split}
Sf=\sum_mf^m(y,\sigma_m)\\
PSf=\sum_m\pi_m f^m(y,\sigma_m).
\end{split}
\end{align}
Writing $f(y,\sigma)=\sum_k f_k(y,\sigma)r_k$ and $f^m(y,\sigma_m)=\sum_k f^m_k(y,\sigma_m)r_k$, we have
\begin{align}\label{41b}
(I-S)f=\sum_k\left[f_k(y,\sigma)-\sum_m f^m_k(y,\sigma_m)\right]r_k.
\end{align}

With the above definition of the operator $E$, the construction of $U_0$  in section \ref{U0} can be repeated with no significant change.  The nonlinear term in the interior profile equation \eqref{c12} for $\sigma_k$ is now $-l_k f^k(y,\sigma_k)$.\\

\textbf{Construction of $V$.} With $\mathcal{F}=L(y,\partial_y)U_0-f(y,U_0)$ we write as before:
\begin{align}\label{41c}
\begin{split}
&(I-E)\mathcal{F}=H_1(y,\theta)+H_2(y,\theta)\text{ where }\\
&H_1(y,\theta)=(I-P)S\mathcal{F}=\sum_k H_{1k}(y,\theta_k)\\
&H_2(y,\theta)=(I-S)\mathcal{F}=-(I-S)f(y,U_0)=\sum_k H_{2k}(y,\theta)r_k.
\end{split}
\end{align}

We have 
\begin{align}\label{41d}
H_{1k}(y,\theta_k)=(1-\pi_k)\left[(L(y,\partial_y)\sigma_k)r_k+\sigma_kL(y,\partial_y)r_k-f^k(y,\sigma_k)\right].
\end{align}
Since $f^k(y,0)=0$, it follows that for any $\theta_k$, $\sigma_k(y,\theta_k)=0\Rightarrow H_{1k}(y,\theta_k)=0$.    Thus, Remark \ref{decV}(b) still applies, and the construction of $V$ (in $\mathcal{U}^\eps_1=V+W^\eps$) goes through unchanged.\\

\textbf{Construction of $W^\eps$. }As before set 
\begin{align}\label{41dd}
\mathcal{H}_2(y,\theta_0,\xi_n):=H_2(\theta_0+\omega_1(y')\xi_n,\dots,\theta_0+\omega_N(y')\xi_n)=\sum_k\mathcal{H}_{2k}(y,\theta_0,\xi_n)r_k(y). 
\end{align}
Using \eqref{41b} and \eqref{41c}, we see that 
\begin{align}\label{41e}
\begin{split}
&(a) -H_{2k}(y,\theta)=f_k(y,\sigma(y,\theta))-\sum_m f^m_k(y,\sigma_m(y,\theta_m)), \text{ so }\\
&(b) -\mathcal{H}_{2k}(y,\theta_0,\xi_n)=f_k(y,\sigma_1(y,\theta_0+\omega_1(y')\xi_n),\dots,\sigma_N(y,\theta_0+\omega_N(y')\xi_n))-\sum_m f^m_k(y,\sigma_m(y,\theta_0+\omega_m(y')\xi_n)).
\end{split}
\end{align}
As before we define real-valued functions $\mathcal{H}_2^m(y,\theta_0,\xi_n)$ by 
\begin{align}\label{41f}
\begin{split}
&(a) \mathcal{H}_2(y,\theta_0,\xi_n)=\sum_m\mathcal{H}^m_2(y,\theta_0,\xi_n)r_m(y',0), \text{ so from }\eqref{41dd}\\
&(b)  \mathcal{H}^m_2(y,\theta_0,\xi_n)=\sum_k\mathcal{H}_{2k}(y,\theta_0,\xi_n)l_m(y',0)r_k(y).
\end{split}
\end{align}
We now define $W(y,\theta_0,\xi_n)$ and $t_k(y,\theta_0,\xi_n)$ exactly as in \eqref{w}.     Using \eqref{41e}(b) and \eqref{41f}(b), we see that $\mathcal{H}^k_2(y,\theta_0,\xi_n)$ is a finite sum of terms of the form
\begin{align}\label{41g}
h(y,\sigma_1(y,\theta_0+\omega_1(y')\xi_n),\dots,\sigma_N(y,\theta_0+\omega_N(y')\xi_n))-\sum_m h^m(y,\sigma_m(y,\theta_0+\omega_m(y')\xi_n)),
\end{align}
where $h(y,\sigma)\in C^\infty$ satisfies $h(y,0)=0$.    Similarly, \eqref{41e}(a) implies that $H_{2k}(y,\theta)$ has the same form where each $\sigma_j(y,\theta_0+\omega_j(y')\xi_n)$ is replaced by $\sigma_j(y,\theta_j)$.

The expression \eqref{41g} is our replacement in this general setting for the products in \eqref{d2y} .    Define $b\in C^\infty(\mathbb{R}^{1+n}_+\times \mathbb{R}^N,\mathbb{R})$ by  
\begin{align}\label{41gg}
b(y,\sigma)=h(y,\sigma)-\sum_m h^m(y,\sigma_m).
\end{align}
    Then for a given $k$, $t_k(y,\theta_0,\xi_n)$ is a finite sum of terms of the form
\begin{align}\label{41h}
\int^{\xi_n}_{+\infty}b(y,\sigma_1(y,\theta_0+\omega_k(y')\xi_n+s(\omega_1(y')-\omega_k(y'))),\dots,\sigma_N(y,\theta_0+\omega_k(y')\xi_n+s(\omega_N(y')-\omega_k(y'))))ds.
\end{align}
We claim that each $t_k(y,\theta_0,\xi_n)$ continues to satisfy Remark \ref{d10yy}.   To see this, observe first that
\begin{align}\label{42a}
b^k(y,\sigma_k)=0\text{ for all }k.
\end{align}
That is, $b(y,\sigma)=0$ whenever all but one component of $\sigma$ are zero.
For any  $m\neq k$ the $s-$support of $\sigma_m(y,\theta_0+\omega_k(y')\xi_n+s(\omega_m(y')-\omega_k(y')))$ is contained in 
$$K_m:=\{s:|\theta_0+\omega_k(y')\xi_n+s(\omega_m(y')-\omega_k(y'))|\leq 1\},$$ 
so by \eqref{42a} the $s-$support of the integrand in \eqref{41h} is contained in $K:=\cup_{m\neq k}K_m$.  
This is a compact set whose length is bounded independently of $(y,\theta_0,\xi_n)\in\Omega^0\times\mathbb{R}^2_{\theta,\xi_n}$.
This completes the construction of $W^\eps(y,\theta_0,\xi_n)=\chi^\eps(y_0,y_n)W(y,\theta_0,\xi_n)$.\\

\textbf{Error analysis. } 
Using \eqref{42a} again, we can apply Proposition \ref{suppprop} (a),(b) to conclude as before that for every $k$
\begin{align}\label{42d}
H_{2k}(y,\theta)|_{\theta=\frac{\psi}{\eps}} \text{ and }\mathcal{H}^k_2(y,\theta_0,\xi_n)|_{\theta_0=\frac{y_0}{\eps},\xi_n=\frac{y_n}{\eps}}\text{ are supported in the interaction region }I_\eps.
\end{align}
Moreover, the proof of Lemma \ref{28g} shows that there exists $M>0$ such that for each $k$,  if $|y_0+\omega_k(y')|\geq M\eps$, then no $s\in\mathbb{R}$ lies in the $s-$support of $\mathrm{supp}\;\sigma_j\left(y,\frac{y_0+\omega_k(y')y_n}{\eps}+s(\omega_j(y')-\omega_k(y'))\right)$ for two distinct $j$.     With this \eqref{42a} implies that for each $k$ 
\begin{align}\label{42e}
|y_0+\omega_k(y')y_n|\lesssim \eps \text{ on the support of }t_k(y,\theta_0,\xi_n)|_{\theta_0=\frac{y_0}{\eps},\xi_n=\frac{y_n}{\eps}}.
\end{align}
That is, Lemma \ref{28g} continues to hold.   Having the support conditions \eqref{42d} and \eqref{42e},  it is now straightforward to repeat the remaining arguments of section \ref{ea}, letting $b(y,\sigma)$ play the role of the products \eqref{d2y}.

\section{A class of problems locally reducible to the model system \eqref{ia1}.}\label{geometric}

Here we explain how to reduce a general class of pulse generation problems to the the model problem \eqref{ia1}.

On a small enough neighborhood $U\ni 0$ in $\mathbb{R}^{1+n}$ 
consider the pulse generation problem \eqref{pcp} on $\mathcal{D}\cap U$ described in Remark \ref{geo}:
\begin{align}\label{pcp2}
\begin{split}
&\mathcal{P}(z,\partial_z)u^\eps=f(z,u^\eps)\text{ in }\beta>0\\
&B(z)u^\eps|_{\beta=0}=g\left(z,\frac{\psi_0(z)}{\eps}\right)\text{ on }\beta=0\\
&u^\eps=0 \text{ in }\alpha<0.
\end{split}
\end{align}
Let $S=\{\alpha=0\}$ and $b\mathcal{D}=\{\beta=0\}$ be the given smooth spacelike surface and noncharacteristic boundary described there.    The transversal intersection assumption allows us to suppose  $d\alpha$ and $d\beta$ are linearly independent on $U$.   Choosing new coordinates $y=(y_0,y'',y_n)$  so that $y_0=\alpha$, $y_n=\beta$, we arrange so that 
\begin{align}\label{45a}
S=\{y_0=0\},\; b\mathcal{D}=\{y_n=0\},\; \Delta=\{y_0=y_n=0\}. 
\end{align}
Since $y_n=0$ is noncharacteristic, the problem \eqref{pcp2} can be written in the new coordinates as 
\begin{align}\label{45aa}
\begin{split}
&L(y,\partial_{y})u^\eps=\partial_nu^\eps+\sum^{n-1}_{j=0}B_j(y)\partial_ju^\eps=f(y,u^\eps)\text{ in }y_n>0,\\ 
&B(y')u^\eps|_{y_n=0}=g(y',\theta_0)|_{\theta_0=\frac{\psi_0(y')}{\eps}}\\
&u^\eps=0\text{ in }y_0<0.
\end{split}
\end{align}
for slightly modified $f$, $g$, and $\psi_0$.

The boundary phase $\psi_0$ is assumed to be a smooth defining function for $\Delta\subset \{y_n=0\}$, so in the new coordinates it has the form
\begin{align}\label{45d}
\psi_0(y')=y_0q(y'),
\end{align}
for some smooth function $q\neq 0$.  Replacing $\psi_0$ by $-\psi_0$ if necessary, we arrange so that $q>0$.   Now consider the eikonal problem
\begin{align}\label{45e}
\begin{split}
&\partial_n\psi_k=-\lambda_k(y,d'\psi_k)\\
&\psi_k|_{y_n=0}=y_0q(y'),
\end{split}
\end{align}
 for $\lambda_k(y,\eta')$, $k=1,\dots,N$ as  defined as in section \ref{eikonal}.\footnote{In defining the $\lambda_k(y,\eta')$ we work with the matrix symbol $L(y,\eta)=\eta_n I+\mathcal{A}(y,\eta')$ of $L(y,\partial_y)$ as in \eqref{45aa}.    Since $y_n=0$ is noncharacteristic, Remark \ref{hyp} applies to yield $N$ distinct real nonzero eigenvalues $\lambda_k(y,\eta')$ of $\mathcal{A}(y,\eta')$, which are  defined for $\eta'$ such that $|\eta''|\leq \delta|\eta_0|$ when $\delta>0$ is small enough.}

If we  change coordinates again, modifying only the time-variable (or $y_0$-coordinate) by taking it to be $y_0q(y')$, then the  pulse continuation problem \eqref{pcp2} takes  the form \eqref{ia1}. In particular, the phase in the boundary datum is now $\frac{y_0}{\eps}$.
Moreover, coordinate invariance of the uniform Lopatinski condition implies that this condition is satisfied by \eqref{45aa} at $0$. 
Thus, we have verified that the transformed problem \eqref{45aa} satisfies all the  hypotheses of Theorem \ref{mr}.

As in \eqref{a5} for each $k=1,\dots,N$ we can write
\begin{align}\label{45f}
\psi_k(y)=(y_0-\phi_k(y'',y_n))\beta_k(y)
\end{align}
for some smooth $\beta_k\neq 0$, where $\phi_k$ satisfies the eikonal problem \eqref{a7}.    The argument of section \ref{O0}, in particular the fact that the numbers $\gamma_k$ in \eqref{phases} are distinct, real, and \emph{nonzero}, shows that near $0$ the surfaces $\Sigma_j=\{\psi_j=0\}=\{y_0-\phi_j=0\}$ satisfy
\begin{align}\label{45g}
\Sigma_j\cap \Sigma_k=\Delta \text{ for }j\neq k, \;\Sigma_j\cap\{y_n=0\}=\Delta, \; \Sigma_j\cap \{y_0=0\}=\Delta,
\end{align}
where all intersections are transversal.

\section{Application to pulse reflection}\label{application}
Consider the pulse reflection problem \eqref{pcp3}, where a single outgoing pulse $u^\eps_0$ concentrated on a characteristic surface $\Sigma=\{\zeta=0\}\ni 0$ reflects transversally off a boundary $b\mathcal{D}=\{\beta=0\}\ni 0$ \emph{starting} at ``time" $\alpha=0$, where $\alpha(0)=0$.  In order to define the characteristic phases $\psi_k$ and surfaces $\Sigma_k$  that appear in Theorem \ref{reflection}, we begin by showing  that near the first time of intersection,  the codimension two manifold $\tilde\Delta:=\Sigma\cap b\mathcal{D}$ is contained in a new spacelike surface $\tilde S\ni 0$ transverse to $b\mathcal{D}$ such that\footnote{This was stated in  \cite{metajm}; here we provide some detail.   In general, the surface $\tilde S$ is not equal to $S=\{\alpha=0\}$.} 
 
\begin{align}\label{46a} 
\tilde\Delta:=\Sigma\cap b\mathcal{D}=\tilde S\cap b\mathcal{D}.
\end{align}
Moreover, there exist $N$ characteristic phases $\psi_k$ and surfaces $\Sigma_k=\{\psi_k=0\}$ such that
\begin{align}\label{46b}
\Sigma_j\cap \Sigma_k=\tilde \Delta \text{ for }j\neq k, \;\Sigma_j\cap b\mathcal{D}=\tilde \Delta, \; \Sigma_j\cap \tilde S=\tilde\Delta,
\end{align}
where all intersections are transversal and $\Sigma$ equals one of the $\Sigma_k$.   Recall that in formulating \eqref{pcp3} we assumed that $S=\{\alpha=0\}$ was spacelike at $0$, but we did \emph{not} assume
that $S\cap b\mathcal{D}=\Sigma\cap b\mathcal{D}$.     After this geometric preparation  we prove the pulse reflection theorem, Theorem \ref{reflection}.

   In new coordinates such that $y_0=\alpha$ and $y_n=\beta$ the problem \eqref{pcp3} takes the form 
\begin{align}\label{46}
\begin{split}
&L(y,\partial_{y})u^\eps=\partial_nu^\eps+\sum^{n-1}_{j=0}B_j(y)\partial_ju^\eps=f(y,u^\eps)\text{ in }y_n>0,\\ 
&B(y')u^\eps|_{y_n=0}=0\\
&u^\eps=u^\eps_0 \text{ in }y_0<-\gamma.
\end{split}
\end{align}
By assumption $L(y,\partial_y)$ is strictly hyperbolic with respect to $y_0$ and $B_n$ is invertible.  
We are given that $y=0$ is a point in the ``first-reflection set" $\Sigma\cap \{y_n=0\}\cap\{y_0=0\}$.
Since the intersection $\tilde\Delta:=\Sigma\cap\{y_n=0\}$ is transversal, some component of $\partial_{y'}\zeta(0)$ is nonzero.  If  $\partial_j\zeta(0)\neq 0$ for some $j\in\{1,\dots,n-1\}$, we can write $\zeta(y)=(y_j-\phi(\tilde y))h(y)$ near $0$ for some smooth $\phi$ and $h$ such $h(0)\neq 0$,   where $\tilde y$ includes all the components of $y$ except $y_j$.  But then $$\tilde \Delta=\{y:y_n=0, y_j=\phi(\tilde y)\}$$ near $0$.     This contradicts the fact that the intersection is empty in $y_0<0$, because $y_0$ is a component of $\tilde y$.
Thus, we must have $\partial_{y_0}\zeta(0)\neq 0$, which implies $\zeta(y)=(y_0-\phi(y'',y_n))h(y)$ so 
$$\tilde\Delta=\{y:y_n=0, y_0=\phi(y'',0)\}.$$
Moreover, $\partial_{y''}\phi(0,0)=0$, since $\phi(y'',0)$ has a \emph{minimum} at $y''=0$, and we can suppose $h(0)>0$ after replacing $\zeta$ by $-\zeta$ if necessary.

Now let $\tilde S=\{y:t(y):=\zeta(y',0)=0\}$.  
We have $\partial_yt(0)=(1,0,\dots,0)h(0)$, so $\tilde S$ is spacelike at $0$ and we can take $t(y)$ as a new time (i.e., $y_0$) coordinate.   In the new coordinates $u^\eps$ again satisfies a problem of the form \eqref{46} near $0$, but with $\gamma$ replaced, for example, by $2\gamma h(0)$.\footnote{We have $t(y_0,0)=y_0h(0)<-\gamma h(0)\Leftrightarrow y_0<-\gamma$.}  Also $\zeta(y',0)=y_0$.  Defining $\mathcal{A}(y,\eta')$ as in \eqref{a2}, we see by Remark \ref{hyp} that $\mathcal{A}(y,\eta')$ has $N$ distinct real nonzero eigenvalues $\lambda_j(y,\eta')$ defined for $|\eta''|\leq \delta |\eta_0|$ for $\delta$ small enough.  Since $\zeta$ is a characteristic phase, for some index, say $N$,  it satisfies
\begin{align}\label{46c}
\begin{split}
&\partial_n\zeta=-\lambda_N(y,d'\zeta)\\
&\zeta|_{y_n=0}=y_0
\end{split}
\end{align}
near $0$.    If we define $\psi_k$ to be the solution of 
\begin{align}\label{46d}
\begin{split}
&\partial_n\psi_k=-\lambda_k(y,d'\psi_k)\\
&\psi_k|_{y_n=0}=y_0,
\end{split}
\end{align}
we see that $\zeta=\psi_N$, and the surfaces $\Sigma_k:=\{\psi_k=0\}$ satisfy \eqref{46b} by the argument given at the end of section \ref{geometric}. 

The assumption that \eqref{pcp3} satisfies the uniform Lopatinski condition at $z=0$, when $\alpha$ is taken as the time coordinate, implies that $(L(y,\partial_y),B(y'))$, when written in the final $y-$coordinates constructed above, satisfies the uniform Lopatinski condition at $y=0$ when $y_0$ is the time coordinate.  In particular, the number of incoming phases $\psi_k$ is the same as the number of positive eigenvalues of $B_0(y)$, namely $p$.    After relabeling if necessary, we  take $\psi_1,\dots,\psi_p$ as the incoming phases.


\begin{proof}[Proof of Theorem \ref{reflection}]
\textbf{1. }We consider the problem \eqref{46} in the final coordinates where \eqref{46c} and \eqref{46d} hold,  $\zeta=\psi_N$, the phases $\psi_1,\dots,\psi_p$ are incoming, and the remaining $\psi_j$ are outgoing.\footnote{Here we have renamed $-2\gamma h(0)$ as $-\gamma$.}   We will write out the proof for the case of a quadratic nonlinearity $f(y,u)$ as in \eqref{c1}.   The extension to general nonlinear functions $f(y,u)$ is then obtained as in section \ref{genf}.

Parallel to the proof of Theorem \ref{mr}, we look for an approximate solution to \eqref{46} of the form
\begin{align}\label{r1}
\begin{split}
&u^\eps_a(y)=\left[U_0(y,\theta)+\eps\mathcal{U}^\eps_1(y,\theta,\theta_0,\xi_n)\right]|_{\theta=\frac{\psi}{\eps},\theta_0=\frac{y_0}{\eps},\xi_n=\frac{y_n}{\eps}}=U_0\left(y,\frac{\psi}{\eps}\right)+\eps U^\eps_1(y), \text{ where }\\
&\quad U_0(y,\theta)=\sum^N_{k=1}\sigma_k(y,\theta_k)r_k(y)
\end{split}
\end{align}
for profiles $\sigma_k$ to be determined and $\mathcal{U}^\eps_1$ as in \eqref{r6a}.
Now we set $\sigma_N(y,\theta_N)=s_N(y,\theta_N)+\tau(y,\theta_N)$, where $s_N$ is unknown and $\tau$ is the given leading profile of the outgoing pulse $u^\eps_o$.   We have the initial conditions
\begin{align}\label{r2}
\sigma_{1},\dots,\sigma_{N-1}=0 \text{ in }y_0<-\gamma,\;\;s_N=0\text{ in }y_0<-\gamma, \text{ and }\tau|_{y_n=0}=0\text{ in }y_0<-\gamma.
\end{align}

\textbf{2. }The computations of section \ref{U0} yield the interior transport equations \eqref{c10}, where the coefficients $c_k$, $d_k$, $e_k$ are just as before.  As before we conclude that the outgoing profiles $\sigma_{p+1},\dots,\sigma_{N-1}$ are zero.  The transport equation for $\sigma_N$ can be written
\begin{align}\label{r3}
X_N(y,\partial_y)s_N+c_Ns_N-\left[d_Ns_N+e_N(s_N^2+2\tau s_N)\right]=-\left[X_N(y,\partial_y)\tau+c_N\tau-(d_N\tau+e_N\tau^2)\right].
\end{align}
The equation 
\begin{align}\label{r4}
X_N(y,\partial_y)\tau+c_N\tau-(d_N\tau+e_N\tau^2)=0
\end{align}
is exactly the interior equation that the (nonzero) profile $\tau$ is constructed to satisfy so that \eqref{out} holds, 
so we conclude from \eqref{r2} and \eqref{r3}  that $s_N$ is zero.   Henceforth we will sometimes write $\sigma_N$ in place of $\tau$. 
The boundary condition $B(y',0)U_0(y',0,\theta_0,\dots,\theta_0)=0$ yields in place of \eqref{c11}:
\begin{align}\label{r5} 
\begin{split}
&\begin{pmatrix}Br_1&\dots&Br_p\end{pmatrix}\begin{pmatrix}\sigma_1\\\vdots\\\sigma_p\end{pmatrix}=-\tau Br_N.
\end{split}
\end{align}
As before this determines the boundary values of the incoming profiles.   With the transport equations \eqref{c10} and initial conditions \eqref{r2}, the $\sigma_j$, $j\leq p$ are now uniquely determined and we have
\begin{align}
U_0(y,\theta)=\sum^{p}_{k=1}\sigma_k(y,\theta_k)r_k(y)+\tau(y,\theta_N) r_N(y).
\end{align}
Observe that 
\begin{align}\label{r6}
\mathrm{supp}\;\sigma_j\subset \{y_0\geq -\gamma\} \text{ for }j=1,\dots,p\text{ and }\mathrm{supp}\;\sigma_N|_{y_n=0}\subset \{y_0\geq -\gamma\}.
\end{align}

\textbf{3. }The profile of the corrector 
\begin{align}\label{r6a}
\mathcal{U}^\eps_1(y,\theta,\theta_0,\xi_n)=V(y,\theta)+W^\eps(y,\theta_0,\xi_n)
\end{align}
is constructed exactly as in sections \ref{V} and \ref{multicorrector}.   We now have
\begin{align}
\begin{split}
&V(y,\theta)=\sum^p_{k=1}V_k(y,\theta_k)+V_N(y,\theta_N)\\
&W^\eps(y,\theta_0,\xi_n)=\chi^\eps(y_0,y_n)W(y,\theta_0,\xi_n),
\end{split}
\end{align}
where $V_k$ (resp. $W$) is given by the formula \eqref{Vk} (resp. \eqref{w} and \eqref{d10z}).  Inspection of these formulas shows 
\begin{align}\label{r7}
\begin{split}
&\mathrm{supp}\;V_k\subset \{y_0\geq -\gamma\} \text{ for }k=1,\dots,p\text{ and }\mathrm{supp}\;V_N|_{y_n=0}\subset \{y_0\geq -\gamma\}\\
&\mathrm{supp}\;W\subset \{y_0\geq -\gamma\},
\end{split}
\end{align}
where the support condition on $W$ reflects the fact that products appearing in the integrand of terms like \eqref{d10z} always involve at least one incoming profile.

 \textbf{4. }This yields an approximate solution $u^\eps_a$ as in \eqref{r1} defined on a set $\Omega^0\ni 0$ chosen to satisfy $\Omega^0\subset \mathcal{O}$ and all the conditions in section \ref{O0}.   With $r^\eps_a$ given by \eqref{d15} we have 
 \begin{align}\label{r8}
 \begin{split}
 &L(y,\partial_y)u^\eps_a=f(y,u^\eps_a)+r^\eps_a\\
 &B(y')u^\eps_a|_{y_n=0}=\eps B(y')U^1_\eps(y',0):=b^\eps(y')\\
 &u^\eps_a=\sigma_N\left(y,\frac{\psi_N}{\eps}\right)r_N(y)+\eps V_N\left(y,\frac{\psi_N}{\eps}\right)\text{ in }y_0<-\gamma.
 \end{split}
 \end{align}
The estimates of Propositions \ref{32a} and \ref{39a} give for any $r\in \mathbb{N}_0$:
\begin{align}\label{r9}
|u^\eps_a|_{L^\infty\cap N^r(\Omega^0)}\lesssim 1, \;\langle b^\eps\rangle_{L^\infty\cap N^r(b\Omega^0)}\lesssim \eps, \;|r^\eps_a|_{L^\infty\cap N^r(\Omega_T)}\lesssim \sqrt{\eps},
\end{align}
where $\Omega_T\subset \Omega^0$ is as in Proposition \ref{35a}.

\textbf{5.  Exact solution. }By \eqref{out} we have $u^\eps_0=\sigma_N\left(y,\frac{\psi_N}{\eps}\right)r_N(y)+w^\eps_0$, so\footnote{For the estimate \eqref{r10} we used $|w^\eps_o|_{L^\infty\cap N^m(\Omega_0\cap\{y_0<-\gamma\})}\leq |w^\eps_o|_{L^\infty\cap N^m(\Omega_0\cap\{y_0<-\gamma\},\mathcal{M}_{\Sigma_N})}$, where  $m>\frac{n+5}{2}$ is as in Theorem \ref{reflection}.} 
\begin{align}\label{r10}
|u^\eps_o|_{L^\infty\cap N^m(\Omega_0\cap\{y_0<-\gamma\})}\leq 1.
\end{align}
Thus, as in Proposition \ref{ex2} we can use Theorem 2.1.1 of \cite{metajm} to conclude that \eqref{46} has an exact solution $u^\eps$ such that for some $\Omega_{T_1}\subset \Omega^0$ with $0<T_1<T_0$:
\begin{align}
|u^\eps |_{L^\infty\cap N^m(\Omega_{T_1})}\lesssim 1.
\end{align}

\textbf{6.  Error problem. }In $y_0<-\gamma$ we compute
\begin{align}\label{r11}
\begin{split}
&u^\eps-u^\eps_a=u^\eps_0-u^\eps_a=\sigma_N\left(y,\frac{\psi_N}{\eps}\right)r_N(y)+w^\eps_0    -\left[\sigma_N\left(y,\frac{\psi_N}{\eps}\right)r_N(y)+\eps V_N\left(y,\frac{\psi_N}{\eps}\right)\right]=\\
&\qquad w^\eps_0-\eps V_N\left(y,\frac{\psi_N}{\eps}\right):=w^\eps_1, \text{ where }|w^\eps_1|_{L^\infty\cap N^m(\Omega_0\cap\{y_0<-\gamma\})}=o(1).
\end{split}
\end{align}
Let  $w^\eps:=u^\eps-u^\eps_a$.  Parallel to \eqref{error2} we have
\begin{align}\label{r12}
\begin{split}
&L(y,\partial_y)w^\eps=D(y,u^\eps,u^\eps_a)w^\eps-r^\eps_a\\
&B(y',0)w^\eps|_{y_n=0}=-b^\eps(y')\\
&w^\eps=w^\eps_1\text{ in }y_0<-\gamma.
\end{split}
\end{align}
The estimates of $r^\eps_a$, $b^\eps$, and $w^\eps_1$ in \eqref{r9} and \eqref{r11} imply, by an argument similar to that of section \ref{weps}, that 
\begin{align}
|w^\eps|_{L^\infty\cap N^m(\Omega_T)}=o(1)\text{ as }\eps\to 0.
\end{align}
for some $0<T\leq T_1$.  If as in the construction of \cite{ar2} we have $|w^\eps_o|_{L^\infty\cap N^m(\Omega_0\cap\{y_0<-\gamma\},\mathcal{M}_{\Sigma_N})}\lesssim \eps$, then we obtain the rate of convergence 
\begin{align}
|w^\eps|_{L^\infty\cap N^m(\Omega_T)}\lesssim \sqrt{\eps}\text{ for }\eps\in (0,\eps_0]
\end{align}
for some $\eps_0>0$.

\end{proof}

\bibliographystyle{alpha}
\bibliography{bib2}





\end{document}